\documentclass[10pt]{amsart}
\usepackage[left=1in,right=1in,top=1in,bottom=1in]{geometry}
\usepackage[utf8]{inputenc}
\usepackage{setspace}
\usepackage{graphicx} 
\usepackage{xcolor}
\usepackage{amssymb}
\usepackage{amsmath}
\usepackage{mathtools}
\usepackage{amsthm}
\usepackage[noabbrev,capitalize,nameinlink]{cleveref}
\usepackage{bbm}
\usepackage{centernot}
\usepackage[OT1]{fontenc}
\usepackage{appendix}
\usepackage{tikz}
\usetikzlibrary{decorations.pathreplacing,angles,quotes}

\newtheorem{thm}{Theorem}
\newtheorem{conjecture}[thm]{Conjecture}
\newtheorem{prop}[thm]{Proposition}
\newtheorem{lemma}[thm]{Lemma}

\newtheorem{cor}[thm]{Corollary}

\theoremstyle{remark}

\theoremstyle{definition}

\newcommand{\lam}{\lambda}
\newcommand{\surj}{\twoheadrightarrow}
\newcommand{\Prob}{\mathbb{P}}
\newcommand{\R}{\mathbb{R}}
\newcommand{\Z}{\mathbb{Z}}
\newcommand{\N}{\mathbb{N}}
\newcommand{\diam}{\mathrm{diam}}

\newcommand{\vol}{\mathrm{vol}}
\newcommand{\im}{\mathrm{im}}
\newcommand{\lowerlambda}{\underline{\lam}}
\newcommand{\upperlambda}{\overline{\lam}}

\newcommand{\Ghat}{\widehat{G}}
\newcommand{\Cech}{\check{C}}

\title{Detection of first homology via
random geometric  graphs in the thermodynamic regime}
\author{Christian Gorski}
\address{University of Washington}
\email{cgorski1@uw.edu}

\date{August 2026}

\begin{document}

\begin{abstract}
    Consider a random geometric graph 
    $G_M(n;r)$ on a compact 
    Riemannian manifold $M$,
    whose vertices are a cloud of
    $n$ independently sampled points,
    and whose edges connect vertices at distance $\le r$.
    
    We show that, in the thermodynamic 
    (i.e. bounded expected average degree) regime, if $G_M(n;r)$
    is supercritical in the sense of 
    continuum percolation, then the first homology group $H_1(M)$ of $M$ can be correctly inferred from 
    $G_M(n;r)$ with high probability as $n \to \infty$.
    Specifically, one can obtain $H_1(M)$ by taking the cycle space
    of $G_M(n;r)$ and quotienting out all the cycles of
    metric diameter $O(r|\log r|)$ 
    (or of graph diameter
    $O(|\log r|)$).
    Our method 
    of estimating $H_1(M)$
    exploits a coarse-topological fact about supercritical percolation,
    as opposed to usual methods, which
    examine the topology of neighborhoods
    of the point cloud.
    Whereas previous methods 
    use combinatorial models which require $O(n \log n)$
    edges,
    our method only requires $O(n)$ edges.

    We also show that, 
    in all phases of the 
    thermodynamic regime, 
    if one instead takes the quotient by
    cycles of metric diameter $o(r|\log r|)$,
    with high probability,
    one will not recover $H_1(M)$. Thus $\Theta(r|\log r|)$ is the ``right scale.''

    On the way, we show
    that an arbitrary compact $d$-dimensional Riemannian manifold
    has a \emph{first homological percolation threshold}
    in the sense of Bobrowski and Skraba \cite{BS2020}
    which coincides with the continuum percolation threshold on $\R^d$,
    a result previously only known for the flat torus.
    This strongly suggests that our results are optimal, in
    the sense that
    $H_1(M)$ cannot be inferred from $G_M(n;r)$
    in the subcritical thermodynamic regime.

    All results hold for homology with arbitrary coefficients.
\end{abstract}

\maketitle

\section{Introduction}

A fundamental question in topological data analysis is the following:
given a collection of data points
(``point cloud'')
sampled from a manifold $M$,
can we infer the topology of the manifold?
A natural approach is to form a simplicial complex out of our collection of points by taking our $0$-skeleton (vertices)
to be our given point cloud,
connecting sufficiently close pairs
of points by $1$-simplices (edges),
filling in sufficiently small
triangles by $2$-simplices,
and so on.
One then hopes that our combinatorial model (i.e. the constructed simplicial
complex) shares topological features
with the underlying manifold $M$
that we wish to detect.

Some classical topological invariants are given by the \emph{homology groups} ~{$(H_k(M)=Z_k(M)/B_k(M))_{k \in \N}$} of $M$. A brief review of homology is given below in \Cref{sec:homology defs}, but for the present discussion it suffices to know that $H_k(X)$ is a module that encodes the ``$k$-dimensional holes'' of 
a space $X$, and that an embedding of our simplicial complex $\mathcal{C}$ into the manifold $M$ induces homomorphisms $H_k(\mathcal{C}) \to H_k(M)$. If this natural map is an \emph{isomorphism} with high probability, we can recover the topological feature $H_k(M)$
of the original manifold from the combinatorial model $\mathcal{C}$.
However, for the usual choices of complex $\mathcal{C}$, in order for this map to be an isomorphism,
the point cloud must
be rather dense in $M$.
For instance, in the case that $\mathcal{C}$ is the \v{C}ech complex, Bobrowski \cite{Bobrowski22} shows that in order for the induced map $H_1(\mathcal{C}) \to H_1(M)$ to be an isomorphism, the expected average degree of $\mathcal{C}$ (thought of as a graph) must be on the order of the logarithm of the number of vertices.\footnote{In fact, \cite{Bobrowski22} gives precise thresholds for the property that the natural map $H_k(\mathcal{C}) \to H_k(M)$ is an isomorphism for each $k$.}
One might therefore conjecture that $H_k(M)$ cannot be recovered from $\mathcal{C}$ unless $\mathcal{C}$ has diverging average degree.
The goal of this paper
is to show that, on the contrary,
at least in the case $k=1$,
one can still correctly deduce
$H_1(M)$ with high probability
from our combinatorial model
in a regime with \emph{constant} average degree.

More precisely,
let $M$ be a $d$-dimensional 
connected\footnote{Since
a general compact manifold is 
a disjoint union of 
finitely many compact
connected components,
analogous
theorems for general compact
Riemannian manifolds
will follow immediately
from our results.}
compact Riemannian manifold of volume 1.
For $n \in \N$ and $r>0$
define the \emph{random geometric graph} $G_M(n;r)$ as follows:
sample $n$ points independently from
the Riemannian volume measure
$dM$  on $M$, and declare this to be
the vertex set of $G_M(n;r)$;
edges are then given by pairs of vertices with Riemannian distance at most $r$.
We allow $r$ to depend on $n$ and imagine that $n \to \infty$.

The regime where $nr^d$
stays bounded away from $0$ and $\infty$ is known as the \emph{thermodynamic regime} for geometric random graphs,
and this regime is the focus of our paper.
This is equivalently the regime
where the expected average degree of $G_M(n;r)$ stays bounded away from $0$ and $\infty$.
In this regime, $G_M(n;r)$ is with high probability disconnected,
and has a close relationship to a continuum percolation
model introduced by Gilbert in 1961 \cite{gilbert1961random},
which we refer to as the \emph{Poisson Boolean model}. For a detailed picture of 
connectivity properties of random geometric
complexes in various regimes,
see for instance the survey \cite{ComplexesSurvey}.

For each $d \ge 2$,
there exists $0<\lam_c(d)<\infty$
(depending on $d$ but \emph{not}
on $M$)
such that for $nr^d =:\lam>\lam_c$,
with probability tending to $1$
there is a unique ``giant'' connected component of $G_M(n;r)$
which pervades $M$,
and for $nr^d =:\lam < \lam_c$,
all connected components
of $G_M(n;r)$ are small.
We will show that we can
infer $H_1(M)$ from $G_M(n;r)$
in the thermodynamic regime
with high probability as $r \to 0$
as long as $\lam > \lam_c$.

To describe how to compute $H_1(M)$ from $G_M(n;r)$, first
recall the \emph{cycle space}
$Z_1(G_M(n;r))$ of $G_M(n;r)$, which is generated by edge-cycles in the usual graph theoretic sense (for more details, refer to \cref{sec:homology defs}).
Our main theorem will say that one can obtain $H_1(M)$
by quotienting out the ``small cycles'' from $Z_1(G_M(n;r))$.
More precisely,
for $s>0$,
let $Z_1^s(G_M(n;r))$
be the submodule of 
$Z_1(G_M(n;r))$
generated by all ${z \in Z_1(G_M(n;r))}$
with $\diam(|z|) \le s$,
where $|z| \subseteq M$ is the support of $z$ (see \Cref{sec:homology notation} for more details).\footnote{Here,
$\diam(|z|)$ is measured with
respect to the Riemannian metric $d_M$.
We also provide two versions of \cref{isomorphism} which do not 
make reference to the Riemannian metric: one in terms of graph distance, and one in terms of Euclidean distance (assuming $M$ is
embedded in Euclidean space).
See \Cref{sec: objections}
for statements and a more thorough discussion.
}
We then have:
\begin{thm} \label[thm]{isomorphism}
    Let $d \ge 2$ and fix $\lam_0 > \lam_c(d)$.
    Given $0 < \kappa < \infty$,
    there exist $0< Q'(\lam_0,\kappa,d) <\infty$
    and $r_0(\lam_0,\kappa,M)>0$
    such that, if $0<r \le r_0$
    and $n r^d \ge \lam_0$,
    \[
        \Prob(
        H_1(M)
        \cong
        Z_1(G_M(n;r))/Z^{Q'r|\log r|}_1(G_M(n;r)))
        \ge 1 - O(r^\kappa).
    \]
\end{thm}
Note that if $G_M(n;r)$ is in the thermodynamic regime, $r = O(n^{-1/d})$
decays to $0$ as $n \to \infty$; but \cref{isomorphism} holds in denser regimes as well, and is interesting as long as $r = o(1)$.

Let us briefly remark on how the above
theorem differs from standard methods of approximating $H_k(M)$.
Most often, one creates a complex called the \v{C}ech complex
out of the given point cloud to approximate the topology of $M$.\footnote{Another complex known as the \emph{Vietoris-Rips complex} is also a popular choice, although it seems that the reasoning is often just that this is a less computationally expensive proxy for the \v{C}ech complex.}
For sufficiently small $r>0$, the \v{C}ech complex is homotopy equivalent to a union of balls around our point cloud; one then can find conditions under which this union of balls has topological features (e.g. homology groups) which match those of $M$.
One could also try to interpret our construction above as approximating $H_1(M)$ using a complex; indeed, $Z_1(G_M(n;r))/Z^{Q'r|\log r|}_1(G_M(n;r))$
is equal to the first homology of a complex obtained from $G_M(n;r)$
by ``filling in'' every $1$-cycle of diameter 
at most $Q' r|\log r|$ by a $2$-cell.
However,
in our case, there is no specific subspace $X \subseteq M$ such that
this complex is deterministically homotopy equivalent to $X$.
Instead, we are taking advantage of a \emph{probabilistic}
coarse-topological fact about supercritical continuum percolation,
namely that in a box with side length $R$, 
any $1$-cycle can be written as a sum of of
$1$-cycles of diameter at most $O(\log R)$.
This idea is the key innovation in this paper.
It also appears that this particular coarse-topological property of continuum percolation has not been explicitly identified in the literature previously.

To obtain the isomorphism
in \cref{isomorphism},
we first recall that 
there exists
a natural map 
$Z_1(G_M(n;r)) \to H_1(M)$.
One should imagine that the natural map takes a (combinatorial) edge-cycle in $G_M(n;r)$
to the corresponding piecewise-geodesic closed path in $M$, which may or may not bound a surface in $M$. For more details on these points,
see \Cref{homology review}.
The first step to showing \cref{isomorphism}
is showing that this natural map is \emph{surjective}
in the supercritical regime.
\begin{thm} \label[thm]{surjectivity}
    Let $d \ge 2$ and let $\lam_0 > \lam_c(d)$. 
    Then, if $nr^d \ge \lam_0$, we have
    \[
        \Prob(Z_1(G_M(n;r)) \to H_1(M)
        \mbox{ is surjective}) \to 1
    \]
    as $n \to \infty$.

    On the other hand,
    if $\lam_0 < \lam_c(d)$
    and $nr^d \le \lam_0$
    then 
    \[
    \Prob(Z_1(G_M(n;r)) \to H_1(M)
        \mbox{ is the } 0 \mbox{ map} ) \to 1
    \]
    as $n \to \infty$.
\end{thm}

In addition to providing the first step
towards recovering $H_1(M)$,
\cref{surjectivity}
has a relatively clear geometric meaning.
Namely, when $\lam > \lam_c$,
``giant $1$-cycles appear'',
that is, $G_M(n;r)$
is connected enough to ``wrap around''
the whole manifold, and therefore
has cycles which represent
the ``macroscopic'' cycles of the
manifold.
The second part of \cref{surjectivity}
is not necessary to prove \cref{isomorphism},
but it does provide the complementary
result that when $\lam < \lam_c$,
$G_M(n;r)$ does not admit any cycles
large enough to wrap around a ``true'' $1$-dimensional hole in $M$.
\cref{surjectivity} is 
a generalization
of an analogous result
showed by Bobrowski and Skraba
on the flat $d$-torus \cite{BS2020}.
In the terminology of \cite{BS2020}, this shows
that compact Riemannian manifolds admit
a $1$-dimensional homological percolation threshold,
and that this threshold 
coincides with the threshold for Poisson Boolean continuum percolation on $\R^d$;
notably, the threshold only depends on the manifold $M$
through its dimension.

\cref{isomorphism} will 
follow from \Cref{surjectivity} and the
First Isomorphism Theorem
once we establish
that the kernel of the
natural map $Z_1(G_M(n;r)) \to H_1(M)$
is exactly
$Z_1^{Q' r|\log r|}(G_M(n;r))$.
When $r>0$ is sufficiently small
compared to the injectivity radius of 
$M$, we will automatically
have that $Z_1^{Q' r|\log r|}(G_M(n;r))$
is \emph{contained} in the kernel of $Z_1(G_M(n;r)) \to H_1(M)$.
Thus, the real hard work of the theorem
is to show the other inclusion,
that is, that any element of 
the kernel of the natural map
is generated by $Q'r|\log r|$-small 
cycles.

Our last main theorem
tells us that
$O(r |\log r|)$
is the ``right''
scale at which to quotient out cycles in the thermodynamic regime,
in the sense that
quotienting out
by $Z^{o(r|\log r|)}_1(G^n_M(r))$
is not sufficient:
\begin{thm} \label[thm]{need at least rlogr}
    Let $0< \lowerlambda \le \upperlambda < \infty$.
    Then there
    exists $q(\lowerlambda,\upperlambda, d) > 0$
    such that,
    if $\lowerlambda \le nr^d \le \upperlambda$, then
    \[
        \Prob(
        \exists z \in \ker(Z_1(G_M(n;r)) \to H_1(M)) \setminus Z_1^{q r|\log r|}(G_M(n;r)))
        \to 1
    \]
    as $r \to 0$.
    Moreover,
    if $\lowerlambda >\lam_c$,
    $z$ above can be chosen
    to lie in the largest component $\mathring{G}_M(n;r)$ of $G_M(n;r)$.
    In particular, in this case we have
    \[
        \Prob\left(\begin{array}{c} Z_1(G_M(n;r))/Z_1^{qr|\log r|}(G_M(n;r))
        \not \cong H_1(M) \mbox{ and } \\ Z_1(\mathring{G}_M(n;r))/Z_1^{qr|\log r|}(\mathring{G}_M(n;r))
        \not \cong H_1(M)
        \end{array}\right)
        \to 1
    \]
    as $r \to 0$.
\end{thm}
Something to note is that \cref{need at least rlogr}
will give us another proof that, in the thermodynamic regime,
standard models of random geometric complexes will not have
first homology coinciding with $H_1(M)$, and in fact
are not much more useful for computing $H_1(M)$
than the random geometric graph $G_M(n;r)$; see \cref{sec:complexes} for a detailed discussion.
It is for this reason that most of the  of this paper
deals with the graph $G_M(n;r)$ instead of a geometric
random complex.

\subsection{Proof outline}

The proofs of \Cref{isomorphism} and  \Cref{surjectivity} proceed as follows. 
First, standard arguments (recalled in \Cref{sec:main proofs}) will
allow us to deduce the desired theorems
for $G_M(n;r)$
from 
their analogues in 
a random geometric graph model $G^n_M(r)$, whose vertex set is a \emph{Poisson point process} on $M$ with intensity measure $ndM$.
Hence we focus on showing the analogous theorems for $G_M^n(r)$.
In \Cref{sec:probabilistic input},
we start by establishing some facts for supercritical percolation on Riemannian manifolds: there is a unique giant component (denoted $\mathring{G}^n_M(r)$), any point of $M$ lies close to a vertex of the giant component, and any pair of points in the same connected component which are close in the Riemannian metric are connected by an edge path in the graph which is not too long.
This is established in a similar way as in recent work of Dubin, the author, and Michelen \cite{DGM26}; using local coordinates, we approximate our random geometric graph by a rescaling of a supercritical Poisson Boolean model on $\R^d$, and then use some results from (possibly inhomogeneous) supercritical Poisson Boolean percolation.

Next, we show that these properties imply an approximation lemma, \Cref{approximate 1-chains}. This lemma says that any $1$-chain $c$ in $M$ which has boundary lying in the giant component $\mathring{G}^n_M(r)$ of our random geometric graph can be approximated by a $1$-chain $c'$ in the giant component $\mathring{G}^n_M(r)$ which is not too far from the original chain; importantly, this lemma also provides a $2$-chain whose boundary is equal to $c'-c$.
This $2$-chain will ensure that in addition to our approximations being geometrically close, they have not changed the \emph{homology class} in $H_1(M)$.
\Cref{approximate 1-chains} applied to the special case that the boundary of $c$ is $0$ will then imply \Cref{surjectivity} almost immediately.

We then turn to showing that
the kernel of the natural map $Z_1(G^n_M(r)) \to H_1(M)$
is generated by cycles of $d_M$-diameter of order $O(r|\log r|)$.
The components besides the giant will
have diameter $O(r|\log r|)$, and so
it will suffice to consider cycles
which lie in the giant component $\mathring{G}^n_M(r)$.
This will be done in two steps:
first, we will show in \Cref{decomposition}
that, for sufficiently nice covers of $M$,
any $z \in Z_1(\mathring{G}^n_M(r)) \cap B_1(M)$ can be written as a sum
of elements of $Z_1(\mathring{G}^n_M(r)) \cap B_1(M)$ which each are supported on a single cover element; then, in \Cref{good cover exists}, we show that a nice cover with each element of diameter $O(r|\log r|)$ exists.

The proof of \Cref{decomposition} is the most technical part of the paper, and together with \Cref{approximate 1-chains}
is the paper's chief technical novelty. The basic idea is as follows:
given an element $U$ of the cover, we wish to inductively replace $z$
by some $z_U \in Z_1(\mathring{G}^n_M(r)) \cap B_1(M)$ 
which is supported entirely on $U$
and
such that $z-z_U$ lies in the union of the other cover elements.
To find $z_U$, we consider a $2$-chain $\sigma \in C_2(M)$ which has $z$ as its boundary, and then consider a related $2$-chain consisting only of the singular $2$-simplices of $\sigma$ which are supported entirely in $U$. The boundary of this related $2$-chain will be close to what we want for $z_U$, but not all of it will lie in $\mathring{G}^n_M(r)$; we rectify this by applying \Cref{approximate 1-chains}.
The technical difficulty comes from the need to keep track of the accumulated error from the approximation and ensure that our approximations do not exit the desired cover elements.

\Cref{sec:characterize kernel} assembles the referenced propositions above into a proof that the kernel of the natural map is indeed generated by small cycles.
\Cref{sec:main proofs} then fills in the rest of the details for obtaining \Cref{isomorphism} and \Cref{surjectivity},
including the ``de-Poissonization'' needed to pass from $G^n_M(r)$ to $G_M(n;r)$.

\Cref{sec: counterexample} proves \Cref{need at least rlogr},
which
is fairly direct; for any $c>0$,
for $q>0$ sufficiently small
we can ``by hand''
construct indecomposable
cycles $z$ with diameter $\ge qr|\log r|$
with probability $\ge r^c$
in a small slab; considering
many disjoint slabs and using independence
gives the desired result.
The 
na\"{i}ve construction
gives loops concentrated
on isolated small components,
but we 
will also
show that 
in the supercritical regime
we can embed
this loop into the giant.

\Cref{sec:complexes} addresses the question
of what happens if one considers a random geometric \emph{complex} rather than a random geometric \emph{graph} in this regime. It turns out that the picture is essentially the same; that is, in the thermodynamic regime, one can only recover $H_1(M)$ by quotienting out by cycles of diameter $\Theta(r|\log r|)$.

\Cref{sec: objections} gives two versions of \Cref{isomorphism} which do not make reference to the Riemannian metric $d_M$. First, one can recover $H_1(M)$ with high probability in the setting that $M$ is embedded in $\R^N$ and one only has access to a point cloud together with \emph{Euclidean} distances.
Second, one can recover $H_1(M)$ with high probability purely from the \emph{graph isomorphism type} of $G_M(n;r)$.

Lastly, \Cref{sec:future} discusses future directions and obstacles, especially in regards to higher-dimensional analogues of the main theorems.

\subsection{Related work}

The topic of inferring the topology of a manifold from a uniformly sampled point cloud is well studied.
Some fundamental work on learning homology in this setting was initiated by Niyogi, Smale and Weinberger in \cite{NSW2008}. 
In \cite{Kahle2011}, using discrete Morse theory, Kahle studied
the ``null hypothesis'' case of geometric random
complexes in $\R^d$ (which has vanishing
$k$\textsuperscript{th} homology for all $k \ge 1$),
and exhibited thresholds for both the appearance and vanishing of $k$\textsuperscript{th} homology
(both outside the thermodynamic regime).
In \cite{MR4329874} and \cite{LerarioMulas2021},
aspects of the \emph{local} topology of random geometric complexes in the thermodynamic regime were studied, including universal laws for normalized counts of homotopy and isotopy types of connected components.

As mentioned above, in \cite{Bobrowski22},
Bobrowski gave exact
thresholds for the property that the $k$\textsuperscript{th}
homology of the complex is isomorphic to $H_k(M)$ (again, in a regime denser than thermodynamic).
In \cite{BS2020},
Bobrowski and Skraba study homological percolation and the 
``emergence of giant $k$-cycles'' in 
the thermodynamic regime, in the flat torus case;
this work inspired \cref{surjectivity}.
We remark that establishing sharpness
of $k$-dimensional homological percolation thresholds remains open; see \cref{sec:future}
for a more detailed discussion.

\subsection{Acknowledgments}

The author thanks Omer Bobrowski and Primoz Skraba for introducing homological percolation to him and for conversations which helped in understanding it.
The author also thanks Marcus Michelen for giving
helpful comments on a preliminary version of this manuscript.

\section{Review of homology} \label{homology review}
Here we give a brief review of singular and simplicial homology
as they pertain to this paper.
A reader familiar with homology may want to skip \Cref{sec:homology defs},
but they should refer to  \Cref{sec:homology notation}, since that establishes important notation which will be used throughout the paper which is not entirely standard.

\subsection{Definitions} \label{sec:homology defs}
The $k$\textsuperscript{th}
\emph{homology group}
$H_k(X)$ of a space $X$
is meant to encode the ``$k$-dimensional holes'' of $X$.
Roughly, a ``$k$-dimensional hole'' should be a $k$-dimensional object in $X$ which has no boundary (is ``closed'')
and which is not the boundary of
a $(k+1)$-dimensional object in $X$. For instance, a $1$-dimensional hole is a closed loop in $X$ which does not bound a surface in $X$.

In our setting, we will be using two different types of homology: \emph{simplicial} homology of a graph (a combinatorial object), and \emph{singular} homology of a manifold (a continuous object).

First, we describe simplicial homology of a graph.
Fix a commutative ring $\mathcal{S}$.\footnote{All of our proofs will work equally well for any choice of $\mathcal{S}$. If the reader does not wish to think
of modules over arbitrary rings, they can either
set $\mathcal{S} = \Z$, in which case $\mathcal{S}$-modules
are exactly the same as abelian groups,
or they can set $\mathcal{S}$ to be their favorite field ($\R, \mathbb{C}, \mathbb{Q}$, etc.), and then modules are
precisely vector spaces over that field.}
Given a graph $G = (V,E)$, let $C_0(G)$ be the
set of formal linear combinations of elements of $V$
with coefficients in $\mathcal{S}$,
where only finitely many of the cofficients are nonzero.
(In other words, $C_0(G)$ is the free $\mathcal{S}$-module
generated by $V$).
Also fix an arbitrary orientation for each edge,
so we can think of each edge $e \in E$
as an \emph{ordered} pair $e = (v_0,v_1)$.
Let $C_1(G)$ similarly be the free $\mathcal{S}$-module
generated by $E$.
We define $C_k(G)=0$ to be the zero module for
any $k \ge 2$.\footnote{Because we only consider \emph{graphs} here,
there is no $k$-dimensional structure for $k \ge 2$.
We could also consider arbitrary \emph{simplicial complexes},
in which case $C_k(G)$ is nonzero in general.
However, for the questions considered in this paper,
this does not end up making much of a difference;
see \Cref{sec:complexes}.}

We then define \emph{boundary maps} $\partial:C_k(G) \to C_{k-1}(G)$
as follows. If $k \ne 1$, $\partial$ is the zero map.
If $k=1$, define $\partial:C_1(G) \to C_0(G)$
to be the unique $\mathcal{S}$-linear map
satisfying $\partial(v_0,v_1) = v_1 - v_0$.
That is, the boundary of an edge $e$ is the formal \emph{difference} of its boundary vertices.

The set of \emph{$k$-cycles} is
$Z_k(G) := \ker(\partial:C_k(G) \to C_{k-1}(G))$.
In our case, $Z_1(G)$ is precisely the \emph{cycle space} of the graph, and its nontrivial elements are (formal sums of) cycles
in the usual graph theoretic sense.
The set of \emph{$k$-boundaries}
is $B_k(G) := \im(\partial:C_{k+1}(G) \to C_k(G))$.
In our current setting, $B_k(G) = 0$ for all $k \ge 1$,
but the analogous continuous concept will be nontrivial.

The \emph{$k$\textsuperscript{th} homology group} $H_k(G)$
of $G$ is $Z_k(G)/B_k(G)$. In our current case,
$H_1(G) \cong Z_1(G)$ 
coincides with the cycle space, since $B_1(G) = 0$. $H_0(G) \cong \mathcal{S}^{\oplus m}$, where $m$ is the number of connected components of $G$.

Next, we define the \emph{singular homology} of 
the manifold $M$. 
First, let $\Delta_k$ be the \emph{standard $k$-simplex},
which we can take to be the convex hull
of the standard basis vectors of $\R^{k+1}$.
Note that $\Delta_0$ is a point, $\Delta_1$ is a segment,
and $\Delta_2$ is a triangle.
Note also that for each $k$,
there are $k+1$ natural embeddings
$\iota_{k,i}:\Delta_{k-1} \to \Delta_{k}$, $0\le i \le k$
corresponding to the $k+1$ $(k-1)$-dimensional faces
of $\Delta_k$.\footnote{For a more detailed description of these maps, see Section 2.1 of Hatcher \cite{Hatcher}.}

Given a manifold (or general topological space) $M$,
a \emph{singular $k$-simplex} is a continuous map $s:\Delta_k \to M$.
For each $k \ge 0$, the set of \emph{singular $k$-chains}
$C_0(M)$ is the free $\mathcal{S}$-module
generated by all singular $k$-simplices in $M$.
Note that composing a singular $k$-simplex $s:\Delta_k \to M$
with an embedding $\Delta_{k-1} \to \Delta_k$
gives a singular $(k-1)$-simplex, and thus
we can define a natural boundary map $\partial:C_k(M) \to C_{k-1}(M)$
by taking $\partial$ to be the unique $\mathcal{S}$-linear map
such that for each singular $k$-simplex $s:\Delta_k \to M$,
we have
\[
    \partial s = \sum_{i=0}^k (-1)^{i}s \circ \iota_{k,i}
    \in C_{k-1}(M).
\]
In particular, if $s \in C_1(M)$
is a path from $x$ to $y$ in $M$,
$\partial s = y - x$.
The boundary of a singular $2$-simplex $s \in C_2(M)$
looks like a signed sum of $1$-simplices
representing the edges in the boundary of the triangle $s$.
The signs in the definition of the boundary map
are chosen so that if we have a $2$-chain representing
an oriented \emph{surface}, a cancellation ensures that the boundary of the sum matches the boundary of the surface as usually geometrically understood.

Again we define the \emph{$k$-cycles} and \emph{$k$-boundaries} by
\[
    Z_k(M) := \ker(\partial:C_k(M) \to C_{k-1}(M)),
    \qquad
    B_k(M) := \im (\partial:C_{k+1} \to C_k(M)).
\]
One can check that for every $k$ we have that
$\partial^2:C_{k+1}(M) \to C_{k-1}(M)$ is the zero map,
and therefore $B_k(M)$ is a submodule of $Z_k(M)$. Therefore,
we can define the \emph{$k$\textsuperscript{th} singular
homology group} of $M$ to be $H_k(M) := Z_k(M)/B_k(M)$.
Since in this article we will only be concerned with $H_1(M)$,
we will not encounter $C_k(M)$ for $k \ge 3$.

If $X$ is \emph{deformation retract} of $Y$ or more generally is \emph{homotopy equivalent} to $Y$, then $H_k(X) \cong H_k(Y)$
for all $k \ge 0$; see Hatcher \cite{Hatcher} for details on these concepts.
The key special case that we will
use repeatedly is that if a space $X$ is \emph{contractible}
(deformation retracts to a \emph{point}),
then $H_k(X)=0$ for all $k \ge 1$.
In particular, then $Z_1(X) = B_1(X)$,
and so for any $1$-cycle $z \in Z_1(X)$, there is some $2$-chain $\sigma \in C_2(X)$ with $\partial \sigma = z$.
Any set homeomorphic to a convex set is necessarily contractible.

\subsection{Notation} \label{sec:homology notation}
Recall that we denote
our ring of coefficients
by $\mathcal{S}$.
For $c \in C_k(M)$ and $s$ a singular $k$-simplex, we will write
$t_s(c) \in \mathcal{S}$ for the coefficient of $s$ in $c$;
thus we have $c = \sum_{s} t_s(c) s$ for any $c \in C_k(M)$.
We also write $s \in c$ if $t_s(c) \ne 0$.

For a singular $k$-simplex $s: \Delta_k \to M$,
we write $|s| \subseteq M$ for the image of $s$ in $M$,
also called the \emph{support} of $s$.
For a $k$-chain $c \in C_k(M)$, we define
\[
    |c| = \bigcup_{s \in c} |s|.
\]

Suppose that $G$ is a graph 
whose vertices are points of $M$.
To each $1$-simplex (edge) $e = (v,w)$ of $G$,
we naturally associate a singular $1$-simplex
$s:\Delta_1 \to M$ by taking $s(e_0) = v, s(e_1) = w$,
and $s$ a geodesic in $M$.
If the distance from $v$ to $w$ is smaller than the global injectivity radius $\mathrm{inj}(M)$
of $M$, there is a unique choice of such a geodesic.
In this way, we get natural inclusions
\[
    C_k(G)  
    \hookrightarrow C_k(M),
    \qquad
    Z_k(G)  
    \hookrightarrow Z_k(M).
\]
We will therefore simply identify $1$-chains and $1$-cycles of $G$
with their images in $C_k(M)$ and $Z_k(M)$ respectively.
(Note, however, that the corresponding maps 
on \emph{homology} will not be injective.)

\section{Some facts about supercritical percolation on Riemannian manifolds}
\label{sec:probabilistic input}

Here we provide the probabilistic input needed to prove \cref{isomorphism}
and \cref{surjectivity}.
Although our main theorems are stated for a random geometric graph $G_M(n;r)$
obtained by sampling $n$ independent points from $M$, for our purposes here it will be easier to deal with a random random geometric graph whose vertex set comes from a \emph{Poisson process}.
We will relate this back to $G_M(n;r)$
in \Cref{sec:main proofs}.

We denote by $d_M$ the metric (distance function) on $M$
induced by its Riemannian metric,
and we denote by $dM$ the Riemannian volume measure on $M$,
which we assume has total mass $1$.\footnote{
This assumption is a technical convenience.
If $\mathrm{vol}(M) \ne 1$, we can rescale the
Riemannian metric by a factor of $\mathrm{vol}(M)^{1/d}$
to obtain a manifold with volume $1$,
and hence we can recover our main results
for arbitrary volume by taking
$\lam = nr^d/\vol(M)$ instead of $\lam = nr^d$.
If one wants to treat the geometry of the manifold as completely
unknown, one can instead state the theorems 
in terms of the expected average
degree $\Delta$
of $G_M(n;r)$,
which is equal to 
$(1 + o(1)) \omega_d n r^d/\vol(M) = (1 + o(1)) \omega_d \lam$
(where here $\omega_d$ is the volume of the unit ball in $\R^d$).
One can deduce from work of Bonnet and Gusakova \cite{bonnet2024concentration}
that the average degree of $G_M(n;r)$ concentrates about its expectation, (see Lemma 41 of \cite{DGM26}
for a statement explicitly in terms of $\Delta$) so the parameter $\Delta = \lam \omega_d$
can be inferred up to small error with high probability from
a sample of $G_M(n;r)$.
}
For $n,r > 0$, we let $G^n_M(r)$
be the random geometric graph
whose vertices are given by
a Poisson point process with
intensity measure $n dM$ on $M$,
and whose edges are pairs of
vertices with $d_M$-distance at most $r$.
We denote by $\mathring{G}^n_M(r)$
the largest component of $G^n_M(r)$
(in metric diameter).
We denote by $d_G$ the usual \emph{graph distance}
on $G^n_M(r)$; that is, for vertices $v,w \in G^n_M(r)$,
$d_G(v,w)$ is the minimal number of edges of $G^n_M(r)$
one must cross to travel from $v$ to $w$.

For $1 \le Q, \rho < \infty$, define the following events:
\begin{align*}
    A^1_{Q, \rho}(n,r) &:=
    \left\{
        \begin{array}{c}
            \forall p \in M, \exists \mathring{p} \in \mathring{G}^n_M(r)
            \mbox{ such that } \\ d_M(p, \mathring{p}) \le Q r |\log r|
        \end{array}
    \right\} \\
    A^2_{Q, \rho}(n,r) 
    &:=
    \left\{
        \begin{array}{c}
            \forall p,q \in G^n_M(r) \mbox{ with } d_M(p,q) \le 2Qr| \log r|^2 + 2r, \\ \mbox{ if } p 
            \mbox{ and } q \mbox{ lie in one component of } G^n_M(r) \mbox{, then }\\
             d_G(p,q) \le \rho \max( r^{-1}d_M(p,q), Q |\log r |)
        \end{array}
    \right\} \\
    A^3_{Q, \rho}(n,r)
    &:=\left\{
        \begin{array}{c}
            \mbox{every component of } G^n_M(r) \mbox{ besides} \\
            \mathring{G}^n_M(r) \mbox{ has } d_M\mbox{-diameter} \\
            \mbox{ at most } Qr|\log r| 
        \end{array}
    \right\},
\end{align*}
and define
\begin{equation} \label{eq:key event}
A_{Q, \rho}(n,r) = \bigcap_{i=1}^3 A^i_{Q, \rho}(n,r).
\end{equation}

Later, using geometric arguments,
we will show that the conclusions
of our main theorems hold
on $A_{Q,\rho}(n,r)$. Thus
we want to show that, in the supercritical regime, the event $A_{Q,\rho}(n,r)$ is likely for some choices of $Q,\rho$.
For technical reasons, let us also define
\begin{align*}
    {A^1}'_{Q, \rho}(n,r)
    :=
    \left\{
        \begin{array}{c}
            \forall p \in M, \exists \mathring{p} \in G^n_M(r)
            \mbox{ with }  d_M(p, \mathring{p}) \le Q r |\log r| \\
            \mbox{such that the connected component of }\mathring{p} \\
            \mbox{has } d_M \mbox{-diameter at least } 2Qr|\log r|
        \end{array}
    \right\},
\end{align*}
and note that
${A^1}'_{Q, \rho}(n,r) \cap A^2_{Q, \rho}(n,r) \cap A^3_{Q, \rho}(n,r)
\subseteq A_{Q,\rho}(n,r)$. We will show that each event in this intersection has high probability for some choice of $Q,\rho$ in the supercritical regime.

\begin{prop} \label[prop]{A_3 likely}
    Let $\kappa > 0$, $\lam_0 > \lam_c(d)$.
    Then there exist 
    $1 \le Q(\kappa,\lam_0,d), 
    \rho(\kappa,\lam_0,d) < \infty$,
    and $r_0(\kappa,\lam_0,M)>0$
    such that whenever
    $n r^d \ge \lam_0$ and $0<r \le r_0$, we have
    \[
        \Prob(A^3_{Q, \rho}(n;r))
        \ge 1 - r^\kappa.
    \]
\end{prop}
\begin{proof}
By Lemma 14 in \cite{DGM26},
there is $c(\lam_0)$ 
not depending on $Q$ such that
$\Prob(A^3_{Q, \rho}(n,r)) \ge 1 - r^{-d}e^{cQ\log r}$,
and so choosing $Q > (d+\kappa)/c$
gives the desired result.
\end{proof}

The proofs that events ${A^1}'_{Q, \rho}(n,r)$ and $A^2_{Q, \rho}(n,r)$
happen with high probability are very similar to
those contained in Section 4 of \cite{DGM26}.
Because the statements are not quite identical, we briefly give the proofs here.

The general strategy
is as follows:
The relevant events are implied by ``local'' versions which hold on an open cover of $M$.
Near any point, at a sufficiently small scale, the geometry of the manifold is approximately flat and the volume form is approximately equal to Lebesgue measure.
Therefore, locally, $G^n_M(r)$
dominates a supercritical percolation process on a patch of $\R^d$;
we then use results from continuum percolation on $\R^d$.

To facilitate this comparison, we introduce some notation.
Let $U=B_M(z,\delta) \subseteq M$ be a metric ball in $M$,
and let $u:U \to B(0,\delta)$ be geodesic normal coordinates at $z \in M$.
We have the following geometric lemma:
\begin{lemma}[Microscopic balls on $M$ are almost flat, see \cite{DGM26}] \label[lemma]{lem:almostflat}
    For each $z \in M$ and $\eta > 0$ there exists $\delta_0 \in (0,1)$
	such that for every $0<\delta \le \delta_0$,
	any ball $U = B_M(z,\delta)$ in $M$ of radius $\delta$ has the following properties. 
	Let
	\[
	u : U \to B_{\R^d}(0,\delta)
	\]
	be geodesic normal coordinates. 
	Then $u$
	is a diffeomorphism with
	\[
	e^{-\eta} d_M(x,y) 
	\le d_{\R^d}( u(x), u(y) )
	\le e^{\eta}d_M(x,y)
	\]
	for all $x,y \in U$.
	Moreover if $\mathcal{L}$ is the Lebesgue measure on $\R^d$ then the Radon-Nikodym derivative of the pushforward $u_\ast dM$ of the Riemannian volume measure $dM$ satisfies $$e^{-\eta} \leq \frac{u_\ast dM}{d\mathcal{L}} \leq e^{\eta}$$ on $u(U)$.
\end{lemma}

Choose  $\eta > 0$ sufficiently small that $\lam_0 e^{-(d+1)\eta} =: \lam_0' > \lam_c$.
For $x \in M$, let $u, \delta > 0$ be as guaranteed by \Cref{lem:almostflat}.  We define $S := u(B_M(x,\delta))$ and 
\begin{equation} \label{eq:G-S-hat-def}
    G_S := u(G^{n}_M(r) \cap B_M(x,\delta)), \qquad 
    \Ghat_S := (e^{-\eta} r)^{-1} G_S\,.
\end{equation}

We can then identify 
the rescaling $\Ghat_S$ as the restriction
of a random geometric graph on $\R^d$
coming from a metric and measure of bounded distortion:

\begin{lemma}[Lemma 17 in \cite{DGM26}]\label[lemma]{lem:identify-G-S}
    Write $\lam = nr^d$.
    There is a metric $\sigma$ on $\R^d$ and measure $\nu$ satisfying $$ \lambda e^{-(d+1)\eta}  \leq \frac{d\nu}{d\mathcal{L}} \leq \lambda e^{(d+1)\eta}\,,\qquad  1 \leq \frac{\|p - q\|}{\sigma(p,q)} \leq e^{2\eta} \quad \text{ for all }p,q \in \R^d$$ so that if we set $G^\nu$ to be the geometric random graph on $\R^d$ whose vertex set is the Poisson point process of intensity measure $\nu$ and whose edges are between pairs $p,q$ with $\sigma(p,q) \leq 1$ then $$\widehat{G}_S \equiv G^\nu \cap (e^{-\eta} r)^{-1} S\,.$$
\end{lemma}

The above gives our recipe for understanding patches of random geometric graphs on Riemannian manifolds.
If we restrict $G^n_M(r)$
to a patch $U$ of $M$,
and then pushforward by geodesic normal coordinates, we get a graph $G_S$
on $\R^d$.
If we rescale, we get another graph $\Ghat_S$
on $\R^d$, which is equal in distribution 
to the restriction of $G^\nu$
to $(e^{-\eta} r)^{-1}S$,
where the intensity measure $\nu$
underlying $G^\nu$ is close to $\lam = nr^d$
times Lebesgue measure, and the underlying metric $\sigma$ is close to Euclidean distance on $\R^d$.
Importantly, $G^\nu$
stochastically dominates the supercritical homogeneous random geometric graph $G^{\lam_0'} = G^{\lam_0'}_{\R^d}(1)$,
in the sense that there is a coupling of $G^\nu$ and $G^{\lam_0'}$
such that $G^{\lam_0'}$
is almost surely a subgraph of $G^\nu$.
With this in mind, we can proceed with our proofs.

\begin{prop} \label[prop]{A_1 likely}
    Let $\kappa > 0$, $\lam_0 > \lam_c(d)$.
    Then there exist 
    $1 \le Q(\kappa,\lam_0,d), $
    and $r_0(\kappa,\lam_0,M)>0$
    such that whenever
    $n r^d \ge \lam_0$ and $0<r \le r_0$, we have
    \[
        \Prob({A^1}'_{Q, \rho}(n;r))
        \ge 1 - r^\kappa.
    \]
\end{prop}
\begin{proof}
    First, for a ball $U = B_M(x,\epsilon) \subseteq M$, define the local event
    \begin{align*}
    A_1(U)
    =
    \left\{
        \begin{array}{c}
            \forall p \in M \mbox{ such that }
            d_M(p,U^c) \ge 4Qr|\log r|, \\\exists \mathring{p} \in G^n_M(r)
            \mbox{ with }  d_M(p, \mathring{p}) \le Q r |\log r| \\
            \mbox{such that the connected component of }\mathring{p} \\
            \mbox{has } d_M \mbox{-diameter at least } 2Qr|\log r|
        \end{array}
    \right\}.
\end{align*}
Note that if $\mathcal{O}$ is a finite open cover of $M$, then for sufficiently small $r>0$
we have
\[
    \bigcap_{U \in \mathcal{O}} A_1(U) \subseteq {A^1}'_{Q,\rho}.
\]
Thus it suffices to prove the statement locally.

Given $\lam_0 > \lam_c$,
choose $\eta>0$
sufficiently small that 
\[
e^{-\eta(d+1)}\lam_0
=: \lam_0' > \lam_c.
\]
Then, for each $x$, by 
\Cref{lem:almostflat} and
\Cref{lem:identify-G-S},
there exists $\epsilon_x > 0$,
$\sigma$ a metric $\sigma$ on $\R^d$
with metric distortion at most $e^\eta$
and $\nu$ a measure on $\R^d$
which dominates $\lam_0' \mathcal{L}$,
such that we can couple the
associated graph $G^\nu$ with $G^n_M(r)$
so that
\[
    (e^{-\eta} r)^{-1}u(G^n_M(r) \cap B(x,\epsilon_x))=: \Ghat_S = G^\nu \cap (e^{-\eta}r)^{-1}S,
\]
where $S = u(B(x,\epsilon_x)) = B(0,\epsilon_x)$.

Moreover, using Poisson thinning,
we also have a coupling such that we have the subgraph inclusion
\[
    G_{\R^d}^{\lam_0'}(1) \cap (e^{-\eta}r)^{-1}S \subseteq \Ghat_S .
\]
Denoting by $\mathring{G}^{\lam_0'}_{\R^d}(1)$ the (unique) infinite component of $G^{\lam_0'}_{\R^d}(1)$, 
we recall the following lemma from Yao, Chen, and Guo ~\cite{YCG11}, which states that each point is not too far from $\mathring{G}^{\lam_0'}_{\R^d}(1)$ with high probability:
\begin{lemma}[Lemma 3.3, Yao, Chen, and Guo~\cite{YCG11}] \label[lemma]{Yao o close in infinite comp}
    Suppose $\lambda > \lambda_c$ and let $C_\infty$ denote the (almost-surely) unique infinite component of $G_{\R^d}^\lambda(1)$. Then there exists a constant $c = c_\lambda>0$ such that for all $R \geq c^{-1}$ we have 
    \[
        \Prob(B_{\R^d}(0,R) \cap C_\infty = \emptyset) \le \exp(-c R^{d-1}).
    \]  
\end{lemma}

The subgraph inclusion above 
and our metric distortion bounds then give us
\[
    \Prob(A_1(U)) \ge
    \Prob\left(
    \begin{array}{c}
    \forall p \in (e^{-\eta}r)^{-1}S,
    \exists \mathring{p} 
    \in \mathring{G}^{\lam_0'}_{\R^d}(1)\\
    \mbox{such that } \|p - \mathring{p}\| \le e^{-2\eta}Q |\log r|
    \end{array}
    \right).
\]
Taking $X \subseteq (e^{-\eta}r)^{-1}S$
to be a $1$-dense subset with size $|X| = O(r^{-d})$,
we see that
\begin{align*}
    \Prob\left(
    \begin{array}{c}
    \exists p \in (e^{-\eta}r)^{-1}S, 
    \mbox{ such that } \\ B(p,e^{-2\eta}Q |\log r|) \cap  \mathring{G}^{\lam_0'}_{\R^d}(1)= \emptyset
    \end{array}
    \right) 
    &\le
    \Prob\left(
    \begin{array}{c}
    \exists p \in X \mbox{ such that } \\
    B(p,e^{-2\eta}Q |\log r| + 1) \cap  \mathring{G}^{\lam_0'}_{\R^d}(1) = \emptyset
    \end{array}
    \right) \\
    \le O(r^{-d}) e^{-c[e^{-2\eta}Q|\log r|+1]}
\end{align*}
for some $c(\lam_0')>0$ not depending on $Q$,
by \Cref{Yao o close in infinite comp}.
Thus for a sufficiently large choice of $Q$
this is $O(r^{\kappa+1})$.
Note that this choice of $Q$ only depends on $d, \lam_0, \eta$ and not on $x \in M$.

Now, for each $x \in U$, choose $\epsilon_x>0$
such that
$A_1(B(x,\epsilon_x))$ has probability $1-O(r^{\kappa+1})$.
$\{B(x,\epsilon)\}_{x \in M}$
forms a cover of $M$,
so by compactness there exists
a finite subcover $\mathcal{O}$.
Thus we have
\[
    \Prob(A_1^c) \le \sum_{U \in \mathcal{O}} O(r^{\kappa+1}) = O(r^{\kappa+1}) \le r^\kappa,
\]
for $0< r \le r_0(\lam_0,d,M)$,
as desired.
\end{proof}

\begin{prop} \label[prop]{A_2 likely}
    Let $\kappa > 0$, $\lam_0 > \lam_c(d)$.
    Then there exist 
    $1 \le Q(\kappa,\lam_0,d), 
    \rho(\kappa,\lam_0,d) < \infty$,
    and $r_0(\kappa,\lam_0,M)>0$
    such that whenever
    $n r^d \ge \lam_0$ and $0<r \le r_0$, we have
    \[
        \Prob(A^2_{Q,\rho}(n,r)) \ge 1 - r^\kappa.
    \]
\end{prop}
\begin{proof}
    Again, we define a local version of our
    desired event:
\begin{align*}
    A_2(U) =     
        \left\{
        \begin{array}{c}
            \forall p,q \in G^n_M(r)
            \mbox{ with } d_M(p,q) \le 2Qr|\log r|^2 + 2r \\ \mbox{ and } d_M(\{p,q\},U^c) \ge \rho(2Qr|\log r|^2 + 2r), \\ \mbox{if } p \mbox{ and } q \mbox{ lie in one component of } G^n_M(r) \mbox{, then } \\
             d_G(p,q) \le \rho \max( r^{-1}d_M(p,q), Q |\log r |)
        \end{array}
    \right\}.
\end{align*}
We will use the following (continuum, inhomogeneous) analogue of the Antal-Pisztora theorem from percolation \cite{AP96}, for which a proof sketch
is included in the appendix of \cite{DGM26}
(see also Lemma 3.4 in \cite{YCG11}):
\begin{lemma}[Lemma 24 of \cite{DGM26}] \label[lemma]{non-uniform-Antal-Pisztora}
    For each $\lam_0 > \lam_c(d)$ there is $c(d,\lam_0) > 0, \rho(d,\lam_0) < \infty$
    such that the following holds.

    Let $\sigma$ be a metric on $\R^d$
    which satisfies
    \[
        \sigma(x,y) \le \|x - y\| \le 2\sigma(x,y)
    \]
    for all $x,y \in \R^d$.
    Let $\nu$ be a positive measure on
    $\R^d$ which satisfies 
    $\frac{d\nu}{d\mathcal{L}} \ge \lam_0$.
    Let $G$ be the random geometric
    graph
    whose vertices are the
    points of a Poisson process on $\R^d$
    with intensity measure $\nu$
    and whose edge set consists
    of pairs of vertices with
    $\sigma$-distance at most 1.

    Then for any $x,y \in \R^d$
    with $\|x - y\| \ge 1$
    and any $t \ge \rho\|x-y\|$ we have
    \[
        \Prob\left( 
        \begin{array}{c}\exists x' \in x+[-1,1]^d, y' \in y+[-1,1]^d \\ \mbox{such that } x' \mbox{ and } y' \mbox{ lie in one} \\ \mbox{component of } G, \mbox{ but } d_G(x',y') \ge t
        \end{array}
        \right)
        \le \exp(-ct).
    \]
\end{lemma}
The above is stated slightly differently in \cite{DGM26} (as a statement conditional on $x,y$ lying in the Poisson process), but the proof immediately yields this statement as well.

Again, given $\lam_0 > \lam_c$,
take $\eta > 0$
sufficiently small that $\lam_0':=e^{-(d+1)\eta}\lam_0 > \lam_c$.
Let $U = B(x,\epsilon)$ be a sufficiently small ball that \Cref{lem:almostflat} and \Cref{lem:identify-G-S} hold and take $G^\nu$ and $\Ghat_S$
as in \cref{lem:identify-G-S}. 
Since $\nu$ dominates $\lam_0' \mathcal{L}$ with $\lam_0' > \lam_c$, take
$\rho'$ and $c>0$
as guaranteed by \cref{non-uniform-Antal-Pisztora},
and set $\rho = e^{2\eta} \rho'$.

Then we have that
\begin{align*}
    \Prob(A_2(U)^c)
    &\le
    \Prob\left( \bigcup_{x,y \in \Z^d \cap 
    (e^{-\eta}r)^{-1}S} 
    \left\{ 
    \begin{array}{c}
    \exists x' \in x+[-1,1]^d, y' \in y+[-1,1]^d \\ \mbox{such that } \|x - y\| \le e^{2\eta} Q |\log r|^2  
    \mbox{ and} \\
    d_G(x,y) > \rho' \max( \|x-y\|,Q|\log r|)
    \end{array}
    \right\} \right) \\
    &\le
    \sum_{x,y \in \Z^d \cap 
    (e^{-\eta}r)^{-1}S}
    e^{-cQ|\log r|} \\
    &\le O(r^{cQ-d}),
\end{align*}
which is $O(r^{\kappa+1})$
for a sufficiently large choice of $Q$.
Using the same finite cover argument and
noting that for any fixed finite cover $\mathcal{O}$,
\[
    \bigcap_{U \in \mathcal{O}} A_3(U) \subset
    A_3
\]
for sufficiently small $r>0$, we are done.
\end{proof}

\begin{prop} \label[prop]{key event is likely}
    Fix $\lam_0 > \lam_c(d)$
    and $\kappa > 0$.
    Then, there 
    exist $1 \le Q, \rho < \infty$
    depending on $\lam_0, \kappa,$ and $d$
    and $r_0 > 0$ depending on $\lam_0, \kappa,$ and $M$
    such that
    \[
        \Prob(A_{Q,\rho}(n,r))
        \ge 1 - r^\kappa.
    \]
\end{prop}
\begin{proof}
Using \Cref{A_1 likely}, \Cref{A_2 likely}, and \Cref{A_3 likely}, choose $1 \le Q, \rho < \infty$ and $r_0'>0$
so that for all $0 < r \le r_0'$, if $nr^d \ge \lam_0$,
we have
\[
    \Prob({A^1}'_{Q, \rho}(n,r)^c),
    \Prob(A^2_{Q, \rho}(n,r)^c),
    \Prob(A^3_{Q, \rho}(n,r)^c)
    \le r^{\kappa+1}.
\]
Then, since $A_{Q, \rho}(n,r) \supset {A^1}'_{Q, \rho}(n,r) \cap A^2_{Q, \rho}(n,r) \cap A^3_{Q, \rho}(n,r)$,
a union bound gives
\[
    \Prob(A_{Q, \rho}(n,r)^c) \le 3r^{\kappa + 1} \le r^{\kappa},
\]
where the last line holds as long as $0< r \le r_0 := \min(r_0',1/3)$.
\end{proof}

\section{Approximation of $1$-chains
and surjectivity}
\subsection{Approximation of $1$-chains}
A fundamental tool throughout
this article
will be a lemma which
allows us to approximate
arbitrary $1$-chains
in $M$ by $1$-chains in 
the giant component $\mathring{G}^n_M(r)$
in the supercritical regime.

Recall the definition of $A_{Q, \rho}(n,r)$
from \eqref{eq:key event}.
We show that on this event, we can approximate
$1$-chains in $M$
up to $O(r|\log r|)$ error.
For the following, recall that the \emph{injectivity radius} $\mathrm{inj}(M)$ of $M$
satisfies that for any $x \in M$, the exponential map $\exp_x:T_xM \to M$ at $x$
is injective when restricted to the ball $B(0,\mathrm{inj}(M)) \subseteq T_xM$.

\begin{lemma} \label[lemma]{approximate 1-chains}
    Suppose $r_0 > 0$ is
    sufficiently small
    that
    \begin{equation} \label{r small enough}
        3/|\log r| \le 1, \qquad
        4 Q \rho r|\log r|
        \le \mathrm{inj}(M)
    \end{equation}
    for all $0<r\le r_0$.
    Then, for all $0 < r \le r_0$,
    on the event $A_{Q,\rho}(n,r)$, 
    the following holds. 
    
    Suppose that 
    $c = \sum_{i=1}^N t_i \ell_i \in C_1(M)$
    and suppose that
    each $\diam (|\ell_i|) \le 2r$
    and 
    $\partial c \in C_0(\mathring{G}^n_M(r))$.

    Then, there exists $c' \in C_1(\mathring{G}^n_M(r))$
    and $\sigma \in C_2(M)$
    such that
    \[
        \partial \sigma = c' - c
    \]
    and for all $x' \in |\sigma|$,
    there exists $x \in |c|$
    such that $d_M(x',x) \le 4Q\rho r|\log r|$.

    In particular,
    $\partial c' = \partial c$
    and $|c'|$
    lies within the $4Q\rho r|\log r|$-neighborhood of $|c|$.
\end{lemma}

\begin{proof}
    Define $\phi:C_0(M) \to C_0(\mathring{G}^n_M(r))$
    as follows: for each
    $0$-simplex $p \in C_0(M)$,
    $\phi(p)$ is the $d_M$
    nearest vertex
    of $\mathring{G}^n_M(r)$,
    breaking ties arbitrarily;
    then extend by linearity.
    Note that if
    $\eta \in C_0(\mathring{G}^n_M(r))$,
    then $\phi(\eta) = \eta$.
    
    Further define
    $\phi:C_1(M) \to C_1(\mathring{G}^n_M(r))$
    as follows.
    If $\ell$ is a singular $1$-simplex
    in $M$ with $\partial\ell = b - a$,
    then let $\phi(\ell) \in C_1(\mathring{G}^n_M(r))$
    be given by a graph geodesic (shortest edge-path)
    in $\mathring{G}^n_M(r)$
    from $\phi(a)$ to $\phi(b)$;
    if there are multiple candidates,
    consider the one of least
    total Riemannian length.
    If $\phi(b) = \phi(a)$,
    just take $\phi(\ell)=0$.
    Then extend by linearity.

    Note that with these definitions,
    $\phi$ is a \emph{chain map},
    that is, $\phi(\partial \eta) = \partial \phi(\eta)$
    for any $\eta \in C_1(M)$.

    Now suppose $c = \sum_{i=1}^N t_i \ell_i \in C_1(M)$
    with each $|\diam(\ell_i)| \le 2r$
    and $\partial c \in C_0(\mathring{G}^n_M(r))$.
    We will 
    take $c' = \phi(c)$.
    We need to construct
    $\sigma \in C_2(M)$
    such that $\partial \sigma = c' - c$
    and such that $|\sigma|$
    stays close to $|c|$.

    To this end, first define
    $\psi:C_0(M) \to C_1(M)$
    as follows: for each $0$-simplex
    $a$ in $M$, 
    if $\phi(a) \ne a$, take $\psi(a)$
    to be the singular $1$-simplex
    given by the $d_M$-geodesic\footnote{
    The event $A_{Q,\rho}$
    plus our bound on $r$ compared to the injectivity radius will imply that this geodesic is unique, although
    if we broke ties arbitrarily
    this would not harm the argument.}
    from $a$ to $\phi(a)$.
    If $a = \phi(a)$,
    simply take $\psi(a) = 0$,
    Note that in either case we 
    have $\partial \psi(a) = \phi(a)-a$.
    
    Now, consider 
    $\ell_i \in c$
    with $\partial \ell_i = b_i - a_i$.
    Define
    \[
        z_i := \phi(\ell_i)
        - \psi(b_i) - \ell_i + \psi(a_i)
        = \phi(\ell_i) - \ell_i - \psi(\partial \ell_i),
    \]
    and note
    that $\partial z_i = 0$.
    From here we will
    show that $z_i$
    has small enough
    diameter that it must
    be contained
    in a contractible subset of $M$,
    and therefore is a boundary.

    \begin{figure}
        \centering

\tikzset{every picture/.style={line width=0.75pt}} 

\begin{tikzpicture}[x=0.75pt,y=0.75pt,yscale=-1,xscale=1]

\draw  [color={rgb, 255:red, 155; green, 155; blue, 155 }  ,draw opacity=1 ][fill={rgb, 255:red, 210; green, 210; blue, 220 }  ,fill opacity=1 ] (305,105) -- (269,113) -- (281,60.5) -- cycle ;
\draw  [color={rgb, 255:red, 155; green, 155; blue, 155 }  ,draw opacity=1 ][fill={rgb, 255:red, 210; green, 210; blue, 220 }  ,fill opacity=1 ] (435,167) -- (384,126) -- (388.5,63.5) -- (417.5,60.5) -- (439,90.5) -- cycle ;
\draw  [color={rgb, 255:red, 155; green, 155; blue, 155 }  ,draw opacity=1 ][fill={rgb, 255:red, 210; green, 210; blue, 220 }  ,fill opacity=1 ] (483,128) -- (435,167) -- (439,90.5) -- (468.5,103.5) -- cycle ;
\draw  [color={rgb, 255:red, 155; green, 155; blue, 155 }  ,draw opacity=1 ][fill={rgb, 255:red, 210; green, 210; blue, 220 }  ,fill opacity=1 ] (147,102) -- (148,138) -- (111.5,90.5) -- (138.5,81) -- cycle ;
\draw  [color={rgb, 255:red, 155; green, 155; blue, 155 }  ,draw opacity=1 ][fill={rgb, 255:red, 210; green, 210; blue, 220 }  ,fill opacity=1 ] (177.5,105) -- (207.5,81.5) -- (206,97) -- (148,138) -- (147,102) -- cycle ;
\draw  [color={rgb, 255:red, 155; green, 155; blue, 155 }  ,draw opacity=1 ][fill={rgb, 255:red, 210; green, 210; blue, 220 }  ,fill opacity=1 ] (223,57.5) -- (247.5,45.5) -- (281,60.5) -- (269,113) -- (206,97) -- (207.5,81.5) -- cycle ;
\draw  [color={rgb, 255:red, 155; green, 155; blue, 155 }  ,draw opacity=1 ][fill={rgb, 255:red, 210; green, 210; blue, 220 }  ,fill opacity=1 ] (345,45.5) -- (374.5,54) -- (388.5,63.5) -- (384,126) -- (305,105) -- (281,60.5) -- (312.5,48.5) -- cycle ;
\draw    (148,138) -- (206,97) ;
\draw    (206,97) -- (269,113) ;
\draw    (269,113) -- (305,105) ;
\draw    (305,105) -- (384,126) ;
\draw    (384,126) -- (435,167) ;
\draw    (435,167) -- (483,128) ;
\draw  [fill={rgb, 255:red, 255; green, 0; blue, 0 }  ,fill opacity=1 ] (144.75,138) .. controls (144.75,136.21) and (146.21,134.75) .. (148,134.75) .. controls (149.79,134.75) and (151.25,136.21) .. (151.25,138) .. controls (151.25,139.79) and (149.79,141.25) .. (148,141.25) .. controls (146.21,141.25) and (144.75,139.79) .. (144.75,138) -- cycle ;
\draw  [fill={rgb, 255:red, 255; green, 0; blue, 0 }  ,fill opacity=1 ] (202.75,97) .. controls (202.75,95.21) and (204.21,93.75) .. (206,93.75) .. controls (207.79,93.75) and (209.25,95.21) .. (209.25,97) .. controls (209.25,98.79) and (207.79,100.25) .. (206,100.25) .. controls (204.21,100.25) and (202.75,98.79) .. (202.75,97) -- cycle ;
\draw  [fill={rgb, 255:red, 255; green, 0; blue, 0 }  ,fill opacity=1 ] (265.75,113) .. controls (265.75,111.21) and (267.21,109.75) .. (269,109.75) .. controls (270.79,109.75) and (272.25,111.21) .. (272.25,113) .. controls (272.25,114.79) and (270.79,116.25) .. (269,116.25) .. controls (267.21,116.25) and (265.75,114.79) .. (265.75,113) -- cycle ;
\draw  [fill={rgb, 255:red, 255; green, 0; blue, 0 }  ,fill opacity=1 ] (301.75,105) .. controls (301.75,103.21) and (303.21,101.75) .. (305,101.75) .. controls (306.79,101.75) and (308.25,103.21) .. (308.25,105) .. controls (308.25,106.79) and (306.79,108.25) .. (305,108.25) .. controls (303.21,108.25) and (301.75,106.79) .. (301.75,105) -- cycle ;
\draw  [fill={rgb, 255:red, 255; green, 0; blue, 0 }  ,fill opacity=1 ] (380.75,126) .. controls (380.75,124.21) and (382.21,122.75) .. (384,122.75) .. controls (385.79,122.75) and (387.25,124.21) .. (387.25,126) .. controls (387.25,127.79) and (385.79,129.25) .. (384,129.25) .. controls (382.21,129.25) and (380.75,127.79) .. (380.75,126) -- cycle ;
\draw  [fill={rgb, 255:red, 255; green, 0; blue, 0 }  ,fill opacity=1 ] (431.75,167) .. controls (431.75,165.21) and (433.21,163.75) .. (435,163.75) .. controls (436.79,163.75) and (438.25,165.21) .. (438.25,167) .. controls (438.25,168.79) and (436.79,170.25) .. (435,170.25) .. controls (433.21,170.25) and (431.75,168.79) .. (431.75,167) -- cycle ;
\draw  [fill={rgb, 255:red, 0; green, 0; blue, 0 }  ,fill opacity=1 ] (479.75,128) .. controls (479.75,126.21) and (481.21,124.75) .. (483,124.75) .. controls (484.79,124.75) and (486.25,126.21) .. (486.25,128) .. controls (486.25,129.79) and (484.79,131.25) .. (483,131.25) .. controls (481.21,131.25) and (479.75,129.79) .. (479.75,128) -- cycle ;
\draw    (148,138) -- (111.5,90.5) ;
\draw  [fill={rgb, 255:red, 0; green, 0; blue, 0 }  ,fill opacity=1 ] (108.25,90.5) .. controls (108.25,88.71) and (109.71,87.25) .. (111.5,87.25) .. controls (113.29,87.25) and (114.75,88.71) .. (114.75,90.5) .. controls (114.75,92.29) and (113.29,93.75) .. (111.5,93.75) .. controls (109.71,93.75) and (108.25,92.29) .. (108.25,90.5) -- cycle ;
\draw  [fill={rgb, 255:red, 0; green, 0; blue, 255 }  ,fill opacity=1 ] (108.25,90.5) .. controls (108.25,88.71) and (109.71,87.25) .. (111.5,87.25) .. controls (113.29,87.25) and (114.75,88.71) .. (114.75,90.5) .. controls (114.75,92.29) and (113.29,93.75) .. (111.5,93.75) .. controls (109.71,93.75) and (108.25,92.29) .. (108.25,90.5) -- cycle ;
\draw  [fill={rgb, 255:red, 0; green, 0; blue, 255 }  ,fill opacity=1 ] (143.75,102) .. controls (143.75,100.21) and (145.21,98.75) .. (147,98.75) .. controls (148.79,98.75) and (150.25,100.21) .. (150.25,102) .. controls (150.25,103.79) and (148.79,105.25) .. (147,105.25) .. controls (145.21,105.25) and (143.75,103.79) .. (143.75,102) -- cycle ;
\draw  [fill={rgb, 255:red, 0; green, 0; blue, 255 }  ,fill opacity=1 ] (204.25,81.5) .. controls (204.25,79.71) and (205.71,78.25) .. (207.5,78.25) .. controls (209.29,78.25) and (210.75,79.71) .. (210.75,81.5) .. controls (210.75,83.29) and (209.29,84.75) .. (207.5,84.75) .. controls (205.71,84.75) and (204.25,83.29) .. (204.25,81.5) -- cycle ;
\draw  [fill={rgb, 255:red, 0; green, 0; blue, 255 }  ,fill opacity=1 ] (277.75,60.5) .. controls (277.75,58.71) and (279.21,57.25) .. (281,57.25) .. controls (282.79,57.25) and (284.25,58.71) .. (284.25,60.5) .. controls (284.25,62.29) and (282.79,63.75) .. (281,63.75) .. controls (279.21,63.75) and (277.75,62.29) .. (277.75,60.5) -- cycle ;
\draw  [fill={rgb, 255:red, 0; green, 0; blue, 255 }  ,fill opacity=1 ] (385.25,63.5) .. controls (385.25,61.71) and (386.71,60.25) .. (388.5,60.25) .. controls (390.29,60.25) and (391.75,61.71) .. (391.75,63.5) .. controls (391.75,65.29) and (390.29,66.75) .. (388.5,66.75) .. controls (386.71,66.75) and (385.25,65.29) .. (385.25,63.5) -- cycle ;
\draw  [fill={rgb, 255:red, 0; green, 0; blue, 255 }  ,fill opacity=1 ] (435.75,90.5) .. controls (435.75,88.71) and (437.21,87.25) .. (439,87.25) .. controls (440.79,87.25) and (442.25,88.71) .. (442.25,90.5) .. controls (442.25,92.29) and (440.79,93.75) .. (439,93.75) .. controls (437.21,93.75) and (435.75,92.29) .. (435.75,90.5) -- cycle ;
\draw [color={rgb, 255:red, 30; green, 120; blue, 30 }  ,draw opacity=1 ]   (111.5,90.5) -- (138.5,81) -- (147,102) -- (177.5,105) -- (207.5,81.5) -- (223,57.5) -- (247.5,45.5) -- (281,60.5) -- (312.5,48.5) -- (345,45.5) -- (374.5,54) -- (388.5,63.5) -- (417.5,60.5) -- (439,90.5) -- (468.5,103.5) -- (483,128) ;
\draw  [fill={rgb, 255:red, 0; green, 0; blue, 255 }  ,fill opacity=1 ] (479.75,128) .. controls (479.75,126.21) and (481.21,124.75) .. (483,124.75) .. controls (484.79,124.75) and (486.25,126.21) .. (486.25,128) .. controls (486.25,129.79) and (484.79,131.25) .. (483,131.25) .. controls (481.21,131.25) and (479.75,129.79) .. (479.75,128) -- cycle ;
\draw [color={rgb, 255:red, 128; green, 128; blue, 255 }  ,draw opacity=1 ]   (147,102) -- (148,138) ;
\draw [color={rgb, 255:red, 128; green, 128; blue, 255 }  ,draw opacity=1 ]   (207.5,81.5) -- (206,97) ;
\draw [color={rgb, 255:red, 128; green, 128; blue, 255 }  ,draw opacity=1 ]   (281,60.5) -- (269,113) ;
\draw [color={rgb, 255:red, 128; green, 128; blue, 255 }  ,draw opacity=1 ]   (281,60.5) -- (305,105) ;
\draw [color={rgb, 255:red, 128; green, 128; blue, 255 }  ,draw opacity=1 ]   (388.5,63.5) -- (384,126) ;
\draw [color={rgb, 255:red, 128; green, 128; blue, 255 }  ,draw opacity=1 ]   (439,90.5) -- (435,167) ;

\draw (332,120.9) node [anchor=north west][inner sep=0.75pt]    {$\ell _{i}$};
\draw (374,128.4) node [anchor=north west][inner sep=0.75pt]  [font=\small,color={rgb, 255:red, 255; green, 0; blue, 0 }  ,opacity=1 ]  {$b_{i}$};
\draw (299.5,107.9) node [anchor=north west][inner sep=0.75pt]  [font=\small,color={rgb, 255:red, 255; green, 0; blue, 0 }  ,opacity=1 ]  {$a_{i}$};
\draw (268.5,33.4) node [anchor=north west][inner sep=0.75pt]  [font=\small,color={rgb, 255:red, 0; green, 0; blue, 255 }  ,opacity=1 ]  {$\phi ( a_{i})$};
\draw (322.5,17.4) node [anchor=north west][inner sep=0.75pt]  [color={rgb, 255:red, 30; green, 120; blue, 30 }  ,opacity=1 ]  {$\phi ( \ell _{i})$};
\draw (380,40.4) node [anchor=north west][inner sep=0.75pt]  [font=\small,color={rgb, 255:red, 0; green, 0; blue, 255 }  ,opacity=1 ]  {$\phi ( b_{i})$};
\draw (339,68.4) node [anchor=north west][inner sep=0.75pt]  [color={rgb, 255:red, 255; green, 255; blue, 255 }  ,opacity=1 ]  {$\sigma _{i}$};
\draw (388,88.9) node [anchor=north west][inner sep=0.75pt]  [font=\small,color={rgb, 255:red, 255; green, 0; blue, 0 }  ,opacity=1 ]  {$\textcolor[rgb]{0.5,0.5,1}{\psi ( b_{i})}$};
\draw (291.5,67.4) node [anchor=north west][inner sep=0.75pt]  [font=\small,color={rgb, 255:red, 255; green, 0; blue, 0 }  ,opacity=1 ]  {$\textcolor[rgb]{0.5,0.5,1}{\psi }\textcolor[rgb]{0.5,0.5,1}{( a}\textcolor[rgb]{0.5,0.5,1}{_{i}}\textcolor[rgb]{0.5,0.5,1}{)}$};

\end{tikzpicture}
        \caption{Construction of the approximation $c' = \phi(c) \in C_1(\mathring{G}^n_M(r))$ of
        $c \in C_1(M)$, and $\sigma \in C_2(M)$
        such that $\partial \sigma = c' - c$.}
        \label{fig:approximation}
    \end{figure}
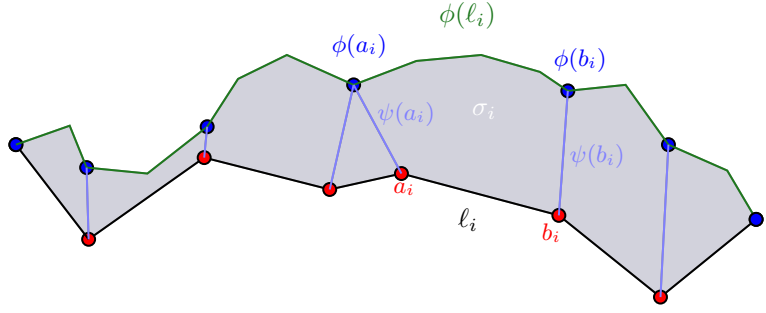

    Since we assumed that
    $\diam(\ell_i) \le 2r$,
    on the event $A_{Q,\rho}(r)$,
    the triangle inequality gives
    \begin{align*}
        d_M(\phi(a_i), \phi(b_i))
        &\le d_M(\phi(a_i),a_i)
        + d_M(a_i, b_i)
        + d_M(b_i, \phi(b_i)) \\
        &\le 2 Q r|\log r| + 2r,
    \end{align*}
    which implies the graph distance bound
    $d_G(\phi(a_i),\phi(b_i))
    \le 2\rho(Q |\log r| + 1)$.
    In particular
    each $x \in |\phi(\ell_i)|$
    has
    \[
        \min(d_M(x, \phi(a_i)), d_M(x, \phi(b_i)) \le 2\rho(Q |\log r| + 1)r.
    \]
    Since
    \[
        \max( d_M(a_i, \phi(a_i)),
        d_M(b_i, \phi(b_i))
        \le Qr|\log r|
    \]
    and since $\diam(\ell_i) \le 2r$,
    we overall have 
    \[
        \sup \{ d_M(x,y) : x \in |\ell_i|, y \in |\phi(\ell_i)| \}
        \le
        [(2 \rho + 1)Q |\log r| + 3]r.
    \]
    Since $\psi(a_i)$
    and $\psi(b_i)$
    are $d_M$-geodesics,
    we also have the bound
    \[
        \sup \{ d_M(x,y) : x \in |\psi(a_i)| \cup |\psi(b_i)|, 
        y \in |\ell_i| \}
        \le (Q|\log r| + 2)r.
    \]
    This all shows that
    $|z|$
    is contained in the metric ball
    $$B(a_i, 4Q\rho r|\log r|)
    := \{ p \in M : d_M(a_i, p) < 4Q\rho r|\log r|\}.$$
    
    Now 
    by our assumption
    on $r$ compared to the injectivity
    radius of $M$, the exponential map
    gives a diffeomorphism between
    $B(a_i, 4 Q \rho r|\log r|)$
    and
    $B(0,4 Q \rho r|\log r|) \subseteq T_{a_i}M$;
    in particular,
    $B(a_i, 4 Q \rho r|\log r|)$
    is contractible and
    therefore has trivial first
    homology.
    Thus $z_i \in Z_1(B(a_i, 4Q \rho r|\log r|)) = B_1(B(a_i, 4Q \rho r|\log r|))$ and there
    exists some 
    $\sigma_i \in C_2(B(a_i, 4Q \rho r|\log r|)) \subseteq C_2(M)$
    such that
    $\partial \sigma_i = z_i$.
    Note that by construction,
    all of $|\sigma_i|$
    lies within distance
    $4Q \rho r|\log r|$ of $a_i \in |c|$.

    Let us define 
    \[
       \sigma = \sum_{i=1}^N t_i \sigma_i.
    \]
    Again, by construction,
    for each $x' \in |\sigma|$,
    there exists $x \in |c|$
    such that
    $d_M(x,x') \le 4Q \rho r|\log r|$,
    as desired.

    We also compute
    \begin{align*}
        \partial \sigma &=
        \sum_{i=1}^N \partial t_i\sigma_i 
        = \sum_{i=1}^N t_i z_i \\
        &= \sum_{i=1}^N \phi(t_i\ell_i) - t_i\ell_i - \psi(\partial t_i\ell_i) \\
        &= \phi(c) - c - \psi(\partial c) \\
        &= c' - c.
    \end{align*}
    In the last line,
    we used that by assumption
    $\partial c \in C_0(\mathring{G}^n_M(r))$,
    and therefore
    $\psi(\partial c) = 0$.
\end{proof}

\subsection{Constructing giant 1-cycles}
Surjectivity, or ``emergence of giant $1$-cycles,'' now follows
almost immediately from
our approximation lemma.

\begin{prop} \label[prop]{prop:surj}
For some $1 \le \rho, Q < \infty$, let $r_0 > 0$ satisfy \eqref{r small enough}.

Then, for any $0 < r \le r_0$,
on the event
$A_{Q,\rho}(r)$,
the natural map
$Z_1(G^n_M(r)) \to H_1(M)$
is surjective.
\end{prop}
\begin{proof}
Suppose $r_0$ satisfies \eqref{r small enough} and $0 < r \le r_0$.

Let $h \in H_1(M)$.
Choose $z \in Z_1(M)$
such that
$h= [z]$
and for all $\ell \in z$
we have $\diam(|\ell|) \le r.$\footnote{This is possible
by e.g. barycentric subdivision
(see e.g. Section 2.1,
proof of Proposition 2.21 in
Hatcher \cite{Hatcher}).}

Since $\partial z = 0 \in C_0(G^n_M(r))$,
the hypotheses of \Cref{approximate 1-chains} hold,
and so on the event $A_{Q,\rho}(r)$,
there exist $z' \in C_1(G^n_M(r))$
and $\sigma \in C_2(M)$
such that $\partial \sigma = z' - z$.
Note that in fact $z' \in Z_1(G^n_M(r))$, since
\[
    \partial z' = \partial(\partial \sigma + z)
    = \partial^2 \sigma + \partial z= 0.
\]

Moreover, $[z'] = [z] = h$.
Thus every $h \in H_1(M)$
has some representative
$z' \in Z_1(G^n_M(r))$,
as desired.
\end{proof}

\section{Decomposing cycles with respect to a cover}
We come to the most technically involved part of our argument.
We show that, if an open cover of $M$ has large enough \emph{Lebesgue number} and small enough \emph{ply},
then any $1$-cycle $z \in Z_1(\mathring{G}^n_M(r))$
which maps to a boundary
in $M$
can be decomposed as a sum
of $1$-cycles $z_1 + \cdots + z_N$
where each $|z_i|$ is completely
contained in some element of the cover.
Later we will show
that we can construct suitable
covers where each element
has small diameter,
and this will give us
a decomposition of such cycles into ``small'' cycles,
establishing our characterization
of the kernel of $Z_1(G^n_M(r)) \to H_1(M)$.

Note that throughout this section we will use the notation $\bar{B}(x,s)$
to denote the \emph{closed} metric ball
\[
    \bar{B}(x,s) := \{y \in M: d_M(x,y) \le s \}.
\]

\begin{lemma} \label[lemma]{decomposition}
    Let $\mathcal{O}$
    be a finite open
    cover of $M$.
    
    Let $L > 0$
    be such that
    for any $x \in M$
    there exists some $U = U(x) \in \mathcal{O}$ such that
    \[
        \bar{B}(x,L) \subseteq U.
    \]

    Moreover, for some $Q, \rho \ge 1$, let
    $K = 4Q\rho r|\log r|$, and suppose that $\Delta$ satisfies
    \[
        \sup_{x \in M}
        \#
        \{U \in \mathcal{O}:B(x,(\Delta+1)K) \cap U \ne \emptyset \}
        \le 
        \Delta.
    \]

    If $L \ge (\Delta + 1)K + 2r$
    and $r$ satisfies
    \eqref{r small enough},
    then, on $A_{Q,\rho}(r)$, the following holds:

    For any $z \in Z_1(\mathring{G}^n_M(r))$
    such that $z = \partial \sigma$
    for some $\sigma \in C_2(M)$,
    we have a decomposition
    \[
        z = \sum_{U \in \mathcal{O}} z_U,
    \]
    where for each $U \in \mathcal{O}$ we have 
    $z_U \in Z_1(\mathring{G}^n_M(r))$
    and $|z_U| \subseteq U$.
\end{lemma}

The rough idea behind the proof of \Cref{decomposition}
is as follows.
Going one cover element $U$ at a time, we find 
$z_U \in Z_1(\mathring{G}^n_M(r))$
with $|z_U| \subseteq U$ and $|z - z_U|$
contained in the union of the other elements;
replacing $z$ by $z' = z - z_U$,
replacing $U$ by another 
cover element $U'$, and applying the
same argument inductively, we expect to
get the desired decomposition.

The subtlety is as follows. The cycles $z_U$
that we construct lie within $U$
but necessarily also have parts
which lie near $U^c$; thus 
we must be careful to allow some ``wiggle room''
so that the cycles we construct
using \Cref{approximate 1-chains}
do not wander outside of $U$.
The function of the Lebesgue number
$L$ of $\mathcal{O}$ is to allow such 
wiggle room.

On the other hand, in order
to perform our induction, we want to discard
each open set $U$ after we have constructed
$z_U$; and once the collection of open sets
no longer covers $M$, the Lebesgue number
as na\"{i}vely defined is not positive.
By being careful in our construction,
we can keep track of our ``wiggle room''
through our induction. Each
time we construct a $z_U$,
the ``wiggle room'' for the remaining
cycle may be slightly diminished,
\emph{but only in that part which intersects} $U$.
This is where the \emph{ply}\footnote{The usual definition of the \emph{ply} of a cover is simply the maximal number of distinct cover elements a single point of $M$ can lie in. Here we have to control something slightly more complicated, since the place where error is being introduced moves around slightly with each modification. Our condition on $\Delta$
essentially says that we cannot have a sequence of $\Delta + 1$ modifications, all of which lie within distance $K$ of one another, and all of which  take place within distinct cover elements.} 
$\Delta$ comes in;
if each part of $z$ lies near a bounded number
(at most $\Delta$)
of different cover elements, this gives
us a bound on the total error accumulated
through our entire inductive construction,
and ensures that no component of the decomposition
escapes the desired cover element.

The inductive step of the proof of \Cref{decomposition} is given by the following lemma.
Below, one should think of $\mathcal{D}$
as the collection of those cover elements $U$ for which we
have already constructed $z_U$.
In particular, our given cycle $z$ should be covered by $\mathcal{O}\setminus \mathcal{D}$, but parts of $z$ which are near
$U \in \mathcal{D}$ may have already accumulated some error
(at most $d(x)$ many times) 
and therefore have less ``wiggle room''
inside $\mathcal{O}\setminus \mathcal{D}$.
The lemma then constructs $z_W = z-z'$
for some cover element $W \in \mathcal{O} \setminus \mathcal{D}$.
\begin{lemma} \label[lemma]{induction}
    Let $\mathcal{O}, L, \Delta$
    be as in the statement
    of \Cref{decomposition}, and
    assume $L \ge (\Delta + 1)4Q\rho r|\log r| + 2r$,
    $r$ satisfies
    \eqref{r small enough}, and
    the event $A_{Q,\rho}(r)$ holds.
    To reduce clutter, 
    write $K = 4Q \rho r|\log r|$.

    Suppose we are given $\mathcal{D} \subseteq \mathcal{O}$,
    $z \in Z_1(\mathring{G}^n_M(r))$,
    and $\sigma \in C_2(M)$
    such that the following holds:
    \begin{enumerate}
        \item $\partial \sigma = z$.
        \item For any $x \in |\sigma|$,
        there exists
        $V \in \mathcal{O} \setminus \mathcal{D}$
        such that 
        \[
            \bar{B}(x, L - d(x)K)
            \subseteq V,
        \]
        where we define
        \[
            d(x) := 
            \sup \left\{ m : 
            \begin{array}{c}
            \exists U_1,...,U_m \in \mathcal{D} \mbox{ distinct, } x_1,...,x_m, x_{m+1}\in M \\
            \mbox{
            with } x_{m+1}=x, \mbox{ such that for all } 1 \le i \le m, \\
            x_i \in U_i
            \mbox{ and }
            d_M(x_i, x_{i+1}) \le K 
            \end{array}
            \right\}.
        \]
    \end{enumerate}

    Then, for any $W \in \mathcal{O} \setminus \mathcal{D}$,
    there exist
    $z' \in Z_1(\mathring{G}^n_M(r))$ and
    $\sigma' \in C_2(M)$
    such that,
    setting 
    \[
        \mathcal{D}' = \mathcal{D} \cup \{W\},
    \]
    we have:
    \begin{enumerate}
        \item $|z - z'| \subseteq W$.
        \item $\partial \sigma' = z'$.
        \item For any $x \in |\sigma'|$,
        there exists
        $V' \in \mathcal{O} \setminus \mathcal{D'}$
        such that 
        \[
            \bar{B}(x, L - d'(x)K)
            \subseteq V',
        \]
        where we define
        \[
            d'(x) := 
            \sup \left\{ m : 
            \begin{array}{c}
            \exists U_1,...,U_m \in \mathcal{D'} \mbox{ distinct, } x_1,...,x_m, x_{m+1}\in M \\
            \mbox{
            with } x_{m+1}=x, \mbox{ such that for all } 1 \le i \le m, \\
            x_i \in U_i
            \mbox{ and }
            d_M(x_i, x_{i+1}) \le K
            \end{array}
            \right\}.
        \]
    \end{enumerate}
\end{lemma}

Before proving \Cref{induction},
let us show how it implies \Cref{decomposition}.

\begin{proof}[Proof of \Cref{decomposition} given \Cref{induction}]
    Let $\mathcal{O}, L, \Delta$
    be as in the hypothesis
    of \Cref{decomposition}.
    Let $z \in Z_1(\mathring{G}^n_M(r))$
    and suppose $z = \partial \sigma$
    for some $\sigma \in C_2(M)$.

    For any $\tilde{z} \in Z_1(\mathring{G}^n_M(r),
    \tilde{\sigma} \in C_2(M)$, and any $\mathcal{D} \subseteq \mathcal{O}$, let us denote by
    $P(\tilde{z}, \tilde{\sigma}, \mathcal{D})$
    the statement that the hypotheses of 
    \Cref{induction} hold for $\tilde{z}, \tilde{\sigma}, \mathcal{D}$; that is,
    $P(\tilde{z}, \tilde{\sigma}, \mathcal{D})$
    is true if and only if we have
    \begin{enumerate}
        \item $\partial\tilde{\sigma} = \tilde{z}$.
        \item For any $x \in |\tilde{\sigma}|$,
            there exists $V \in \mathcal{O} \setminus \mathcal{D}$ such that
            \[
                \bar{B}(x, L - d(x) K)
                \subseteq V,
            \]
            where
            \[
            d(x) := 
            \sup \left\{ m : 
            \begin{array}{c}
            \exists U_1,...,U_m \in \mathcal{D} \mbox{ distinct}, x_1,...,x_m, x_{m+1}\in M \\
            \mbox{
            with } x_{m+1}=x, \mbox{ such that for all } 1 \le i \le m, \\
            x_i \in U_i
            \mbox{ and }
            d_M(x_i, x_{i+1}) \le K 
            \end{array}
            \right\}.
        \]
    \end{enumerate}

    Fix some arbitrary enumeration
    $\mathcal{O} = \{W_1,...,W_N\}$.
    We use induction to prove the following 
    statement for each $n=0,1,...,N$:
\begin{multline*}
    \mbox{For each } 1\le i \le n\mbox{, there exist }
    z_i \in Z_1(\mathring{G}^n_M(r))
    \mbox{ such that } |z_i| \subseteq W_i \\
    \mbox{ and there exists } \sigma_n \in C_2(M) 
    \mbox{ such that }
    \\
    P\left( z - \left(\sum_{i=1}^n z_i\right), \sigma_n, 
    \{W_1,...,W_n\}\right)
    \mbox{ holds}. \\
\end{multline*}

First, note that the base case $n=0$ holds, since
$P(z, \sigma, \emptyset)$ is
just the statement that $\partial \sigma = z$
(which is true by our construction of $\sigma$)
and that for each $x \in |\sigma|$
there exists some $V \in \mathcal{O}$
such that $\bar{B}(x,L) \subseteq V$,
which is true by our assumption on $L$.

Next, assume that our statement holds
for some $n \in \{0,...,N-1\}$.
Since $P\left(z - \left(\sum_{i=1}^n z_i\right), \sigma_n, 
\{W_1,...,W_n\}\right)$ holds, using
\Cref{induction} with $W = W_{n+1}$ gives us
$\sigma', z'$ such that
\[
    \left|z - \left(\sum_{i=1}^n z_i\right) - z'\right|
    \subseteq W_{n+1}
\]
and $P(\sigma', z', \{W_1,...,W_{n+1}\})$
holds.

Therefore, if we define 
\[
    z_{n+1} := z - \left(\sum_{i=1}^n z_i\right) - z',
\]
and $\sigma_n = \sigma'$, we have
established our desired statement for $n+1$.

By induction, we in particular obtain
$z_1,...,z_N \in Z_1(\mathring{G}^n_M(r))$
such that each $|z_i| \subseteq W_i$ and
$\sigma_N \in C_2(M)$ such that
\[
    P\left(z - \left(\sum_{i=1}^N z_i \right),
    \sigma_N, \mathcal{O}
    \right)
\]
holds.
But this implies that $|\sigma_N| = \emptyset$,
since for any $x \in |\sigma_N|$,
$\bar{B}(x, L-d(x) K)$
must be contained in some $V \in \mathcal{O} \setminus \mathcal{O} = \emptyset$, which is impossible.

Therefore, we have that
\[
    z - \left(\sum_{i=1}^N z_i \right)
    = \partial \sigma_N = 0,
\]
that is, $z = \sum_{i=1}^N z_i$
with each $z_i \in Z_1(\mathring{G}^n_M(r))$ 
and $|z_i| \subseteq W_i$, as desired.
\end{proof}

Now let us show \Cref{induction}.
The idea is to split up $\sigma \in C_2(M)$
into two parts: one consisting of the simplices contained comfortably within $W$,
the other consisting of the rest.
Taking the boundary of each of these parts
gives a decomposition of $z$ into a part contained
in comfortably within $W$
and a part contained in the other cover elements.
However, the cycles thus constructed
do not a priori lie in $Z_1(\mathring{G}^n_M(r))$;
we therefore use \Cref{approximate 1-chains}
to approximate the ``ideal'' decomposition,
taking care to account for the error in the inductive hypothesis.

\begin{proof}[Proof of \Cref{induction}]
    Suppose our stated hypotheses hold.
    We again write
    $K = 4Q\rho r|\log r|$
    to reduce clutter.

    First, replacing $\sigma \in C_2(M)$
    by an appropriately subdivided $2$-chain
    if necessary, we assume without loss of 
    generality
    that each $c \in \sigma$
    has $\diam(c) \le 2r$.\footnote{
    We can assume this by doing 
    a subdivision of $2$-simplices
    which is akin to barycentric subdivision, except that we do not
    subdivide $1$-simplices lying
    in $z$; this ensures that 
    our subdivided $2$-chain
    still has boundary equal to $z$.
    To
    see details of the subdivision procedure,
    refer to \Cref{sec:subdivision}
    of the appendix.}

    Let $W \in \mathcal{O} \setminus \mathcal{D}$,
    and write $\mathcal{D}' = \mathcal{D} \cup \{W\}$.
    Define $\tilde{\sigma} \in C_2(M)$ by
    \[
       \tilde{\sigma} :=
       \sum \left\{
       t_c(\sigma) c :
       \exists x \in |c| \mbox{ such that } \bar{B}(x,K) \not\subseteq W
       \right\}.
    \]

        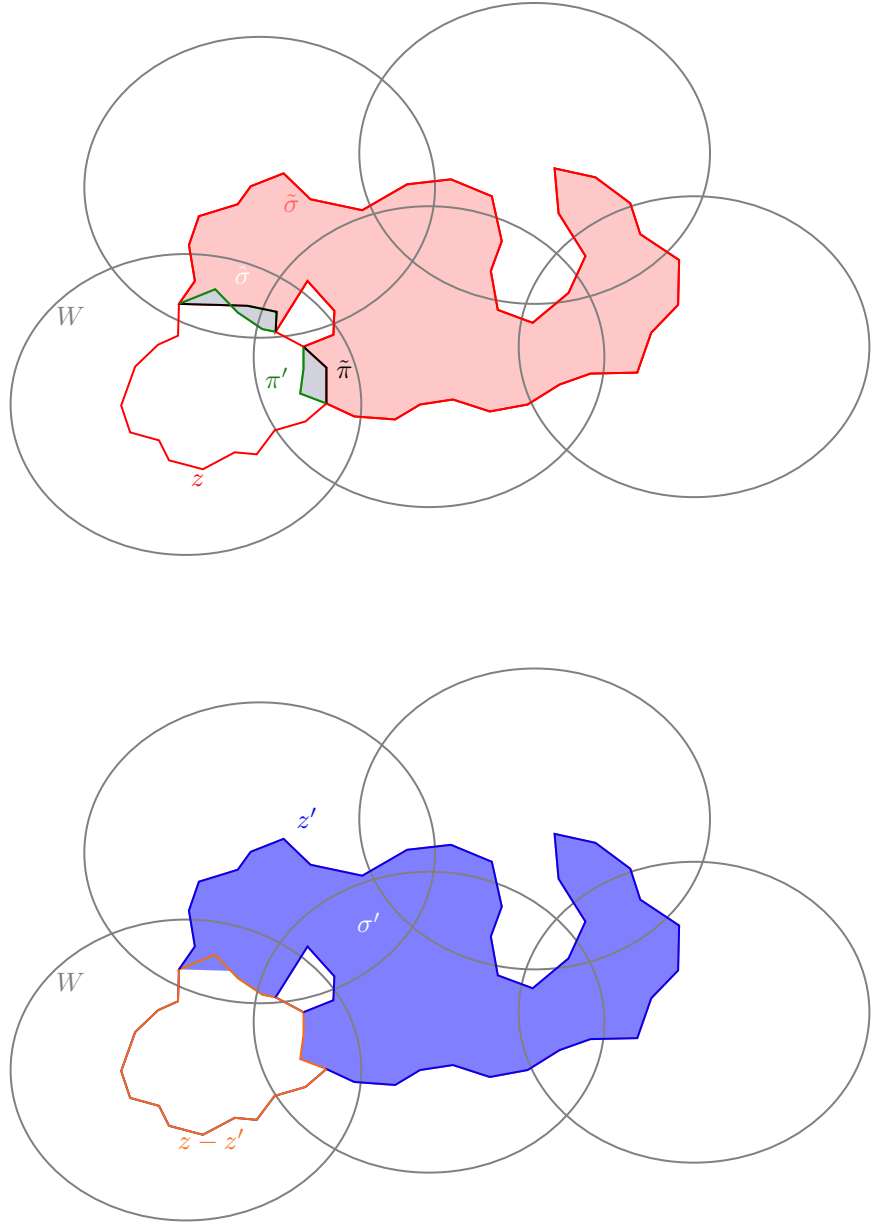
\begin{figure}

\tikzset{every picture/.style={line width=0.75pt}} 

\begin{tikzpicture}[x=0.75pt,y=0.75pt,yscale=-1,xscale=1]

\draw [color={rgb, 255:red, 0; green, 255; blue, 0 }  ,draw opacity=1 ][fill={rgb, 255:red, 128; green, 128; blue, 255 }  ,fill opacity=1 ]   (147.5,488) -- (165.67,480.67) -- (177,492.33) -- (189.33,500.67) -- (196,502) -- (196.5,492) -- (182,489) ;
\draw [color={rgb, 255:red, 0; green, 255; blue, 0 }  ,draw opacity=1 ][fill={rgb, 255:red, 210; green, 210; blue, 220 }  ,fill opacity=1 ]   (210,175.5) -- (210,186.67) -- (208.33,199) -- (221.5,204) -- (221.5,186) ;
\draw [color={rgb, 255:red, 255; green, 0; blue, 0 }  ,draw opacity=1 ][fill={rgb, 255:red, 255; green, 200; blue, 200 }  ,fill opacity=1 ]   (221.5,204) -- (235.5,210.5) -- (256,212) -- (268.5,204.5) -- (285,202) -- (303.5,208) -- (322.5,204.5) -- (338.5,194.5) -- (354,189) -- (377.5,188.5) -- (384.5,168.5) -- (398,154.5) -- (398.5,132) -- (379,119) -- (374,103.5) -- (356.5,90.5) -- (336,86) -- (338,108.5) -- (351.5,130) -- (343,148.5) -- (325,163.5) -- (307.5,157) -- (304,137.5) -- (309.5,122.5) -- (304.5,100) -- (284,91.5) -- (262,94) -- (239.5,107) -- (213.5,101.5) -- (200,88.5) -- (183.5,95) -- (177,104) -- (157.5,110) -- (152.5,124.5) -- (155.5,142.5) -- (147.5,154) -- (182,155) -- (196.5,158) -- (196,168) -- (212,142.5) -- (225.5,157.5) -- (225,169.5) -- (210,175.5) -- (221.5,186) -- cycle ;
\draw [color={rgb, 255:red, 0; green, 255; blue, 0 }  ,draw opacity=1 ][fill={rgb, 255:red, 210; green, 210; blue, 220 }  ,fill opacity=1 ]   (147.5,154) -- (165.67,146.67) -- (177,158.33) -- (189.33,166.67) -- (196,168) -- (196.5,158) -- (182,155) ;
\draw  [color={rgb, 255:red, 128; green, 128; blue, 128 }  ,draw opacity=1 ] (100,95.5) .. controls (100,53.8) and (139.4,20) .. (188,20) .. controls (236.6,20) and (276,53.8) .. (276,95.5) .. controls (276,137.2) and (236.6,171) .. (188,171) .. controls (139.4,171) and (100,137.2) .. (100,95.5) -- cycle ;
\draw  [color={rgb, 255:red, 128; green, 128; blue, 128 }  ,draw opacity=1 ] (238,78.5) .. controls (238,36.8) and (277.4,3) .. (326,3) .. controls (374.6,3) and (414,36.8) .. (414,78.5) .. controls (414,120.2) and (374.6,154) .. (326,154) .. controls (277.4,154) and (238,120.2) .. (238,78.5) -- cycle ;
\draw  [color={rgb, 255:red, 128; green, 128; blue, 128 }  ,draw opacity=1 ] (318,175.5) .. controls (318,133.8) and (357.4,100) .. (406,100) .. controls (454.6,100) and (494,133.8) .. (494,175.5) .. controls (494,217.2) and (454.6,251) .. (406,251) .. controls (357.4,251) and (318,217.2) .. (318,175.5) -- cycle ;
\draw  [color={rgb, 255:red, 128; green, 128; blue, 128 }  ,draw opacity=1 ] (185,180.5) .. controls (185,138.8) and (224.4,105) .. (273,105) .. controls (321.6,105) and (361,138.8) .. (361,180.5) .. controls (361,222.2) and (321.6,256) .. (273,256) .. controls (224.4,256) and (185,222.2) .. (185,180.5) -- cycle ;
\draw  [color={rgb, 255:red, 128; green, 128; blue, 128 }  ,draw opacity=1 ] (63,204.5) .. controls (63,162.8) and (102.4,129) .. (151,129) .. controls (199.6,129) and (239,162.8) .. (239,204.5) .. controls (239,246.2) and (199.6,280) .. (151,280) .. controls (102.4,280) and (63,246.2) .. (63,204.5) -- cycle ;
\draw [color={rgb, 255:red, 255; green, 0; blue, 0 }  ,draw opacity=1 ]   (123,218.5) -- (137.5,222.5) -- (142.5,232.5) -- (159.5,237) -- (175.5,228.5) -- (186.5,229.5) -- (195.5,217.5) -- (211,213) -- (221.5,204) -- (235.5,210.5) -- (256,212) -- (268.5,204.5) -- (285,202) -- (303.5,208) -- (322.5,204.5) -- (338.5,194.5) -- (354,189) -- (377.5,188.5) -- (384.5,168.5) -- (398,154.5) -- (398.5,132) -- (379,119) -- (374,103.5) -- (356.5,90.5) -- (336,86) -- (338,108.5) -- (351.5,130) -- (343,148.5) -- (325,163.5) -- (307.5,157) -- (304,137.5) -- (309.5,122.5) -- (304.5,100) -- (284,91.5) -- (262,94) -- (239.5,107) -- (213.5,101.5) -- (200,88.5) -- (183.5,95) -- (177,104) -- (157.5,110) -- (152.5,124.5) -- (155.5,142.5) -- (147.5,154) -- (147,170) -- (137,174.5) -- (125.5,185.5) -- (118.5,205) -- cycle ;
\draw [color={rgb, 255:red, 255; green, 0; blue, 0 }  ,draw opacity=1 ]   (196,168) -- (210,175.5) -- (225,169.5) -- (225.5,157.5) -- (212,142.5) -- cycle ;
\draw [color={rgb, 255:red, 0; green, 0; blue, 0 }  ,draw opacity=1 ]   (147.5,154) -- (166.5,154.5) -- (182,155) -- (196.5,158) -- (196,168) ;
\draw [color={rgb, 255:red, 0; green, 0; blue, 0 }  ,draw opacity=1 ]   (210,175.5) -- (221.5,186) -- (221.5,204) ;
\draw [color={rgb, 255:red, 30; green, 120; blue, 30 }  ,draw opacity=1 ]   (147.5,154) -- (165.67,146.67) -- (177,158.33) -- (189.33,166.67) -- (196,168) ;
\draw [color={rgb, 255:red, 30; green, 120; blue, 30 }  ,draw opacity=1 ]   (210,175.5) -- (210,186.67) -- (208.33,199) -- (221.5,204) ;
\draw [color={rgb, 255:red, 250; green, 105; blue, 15 }  ,draw opacity=1 ][fill={rgb, 255:red, 128; green, 128; blue, 255 }  ,fill opacity=1 ]   (221.5,538) -- (235.5,544.5) -- (256,546) -- (268.5,538.5) -- (285,536) -- (303.5,542) -- (322.5,538.5) -- (338.5,528.5) -- (354,523) -- (377.5,522.5) -- (384.5,502.5) -- (398,488.5) -- (398.5,466) -- (379,453) -- (374,437.5) -- (356.5,424.5) -- (336,420) -- (338,442.5) -- (351.5,464) -- (343,482.5) -- (325,497.5) -- (307.5,491) -- (304,471.5) -- (309.5,456.5) -- (304.5,434) -- (284,425.5) -- (262,428) -- (239.5,441) -- (213.5,435.5) -- (200,422.5) -- (183.5,429) -- (177,438) -- (157.5,444) -- (152.5,458.5) -- (155.5,476.5) -- (147.5,488) -- (165.67,480.67) -- (177,492.33) -- (189.33,500.67) -- (196,502) -- (212,476.5) -- (225.5,491.5) -- (225,503.5) -- (210,509.5) -- (210,520.67) -- (208.33,533) -- cycle ;
\draw  [color={rgb, 255:red, 128; green, 128; blue, 128 }  ,draw opacity=1 ] (100,429.5) .. controls (100,387.8) and (139.4,354) .. (188,354) .. controls (236.6,354) and (276,387.8) .. (276,429.5) .. controls (276,471.2) and (236.6,505) .. (188,505) .. controls (139.4,505) and (100,471.2) .. (100,429.5) -- cycle ;
\draw  [color={rgb, 255:red, 128; green, 128; blue, 128 }  ,draw opacity=1 ] (238,412.5) .. controls (238,370.8) and (277.4,337) .. (326,337) .. controls (374.6,337) and (414,370.8) .. (414,412.5) .. controls (414,454.2) and (374.6,488) .. (326,488) .. controls (277.4,488) and (238,454.2) .. (238,412.5) -- cycle ;
\draw  [color={rgb, 255:red, 128; green, 128; blue, 128 }  ,draw opacity=1 ] (318,509.5) .. controls (318,467.8) and (357.4,434) .. (406,434) .. controls (454.6,434) and (494,467.8) .. (494,509.5) .. controls (494,551.2) and (454.6,585) .. (406,585) .. controls (357.4,585) and (318,551.2) .. (318,509.5) -- cycle ;
\draw  [color={rgb, 255:red, 128; green, 128; blue, 128 }  ,draw opacity=1 ] (185,514.5) .. controls (185,472.8) and (224.4,439) .. (273,439) .. controls (321.6,439) and (361,472.8) .. (361,514.5) .. controls (361,556.2) and (321.6,590) .. (273,590) .. controls (224.4,590) and (185,556.2) .. (185,514.5) -- cycle ;
\draw  [color={rgb, 255:red, 128; green, 128; blue, 128 }  ,draw opacity=1 ] (63,538.5) .. controls (63,496.8) and (102.4,463) .. (151,463) .. controls (199.6,463) and (239,496.8) .. (239,538.5) .. controls (239,580.2) and (199.6,614) .. (151,614) .. controls (102.4,614) and (63,580.2) .. (63,538.5) -- cycle ;
\draw [color={rgb, 255:red, 0; green, 0; blue, 255 }  ,draw opacity=1 ]   (123,552.5) -- (137.5,556.5) -- (142.5,566.5) -- (159.5,571) -- (175.5,562.5) -- (186.5,563.5) -- (195.5,551.5) -- (211,547) -- (221.5,538) -- (235.5,544.5) -- (256,546) -- (268.5,538.5) -- (285,536) -- (303.5,542) -- (322.5,538.5) -- (338.5,528.5) -- (354,523) -- (377.5,522.5) -- (384.5,502.5) -- (398,488.5) -- (398.5,466) -- (379,453) -- (374,437.5) -- (356.5,424.5) -- (336,420) -- (338,442.5) -- (351.5,464) -- (343,482.5) -- (325,497.5) -- (307.5,491) -- (304,471.5) -- (309.5,456.5) -- (304.5,434) -- (284,425.5) -- (262,428) -- (239.5,441) -- (213.5,435.5) -- (200,422.5) -- (183.5,429) -- (177,438) -- (157.5,444) -- (152.5,458.5) -- (155.5,476.5) -- (147.5,488) -- (147,504) -- (137,508.5) -- (125.5,519.5) -- (118.5,539) -- cycle ;
\draw [color={rgb, 255:red, 0; green, 0; blue, 255 }  ,draw opacity=1 ]   (196,502) -- (210,509.5) -- (225,503.5) -- (225.5,491.5) -- (212,476.5) -- cycle ;
\draw [color={rgb, 255:red, 250; green, 105; blue, 15 }  ,draw opacity=1 ]   (196,502) -- (210,509.5) ;
\draw [color={rgb, 255:red, 250; green, 105; blue, 15 }  ,draw opacity=1 ]   (221.5,538) -- (211,547) -- (195.5,551.5) -- (186.5,563.5) -- (175.5,562.5) -- (159.5,571) -- (142.5,566.5) -- (137.5,556.5) -- (123,552.5) -- (118.5,539) -- (125.5,519.5) -- (137,508.5) -- (147,504) -- (147.5,488) ;

\draw (84.33,154.07) node [anchor=north west][inner sep=0.75pt]  [color={rgb, 255:red, 128; green, 128; blue, 128 }  ,opacity=1 ]  {$W$};
\draw (152,238) node [anchor=north west][inner sep=0.75pt]  [color={rgb, 255:red, 255; green, 0; blue, 0 }  ,opacity=1 ]  {$z$};
\draw (198.5,97.4) node [anchor=north west][inner sep=0.75pt]  [color={rgb, 255:red, 255; green, 100; blue, 100 }  ,opacity=1 ]  {$\tilde{\sigma }$};
\draw (225,180.57) node [anchor=north west][inner sep=0.75pt]    {$\tilde{\pi }$};
\draw (189.33,184.73) node [anchor=north west][inner sep=0.75pt]  [color={rgb, 255:red, 30; green, 120; blue, 30 }  ,opacity=1 ]  {$\pi '$};
\draw (174,133.23) node [anchor=north west][inner sep=0.75pt]  [color={rgb, 255:red, 255; green, 255; blue, 255 }  ,opacity=1 ]  {$\hat{\sigma }$};
\draw (84.33,488.07) node [anchor=north west][inner sep=0.75pt]  [color={rgb, 255:red, 128; green, 128; blue, 128 }  ,opacity=1 ]  {$W$};
\draw (145.5,566.4) node [anchor=north west][inner sep=0.75pt]  [color={rgb, 255:red, 250; green, 105; blue, 15 }  ,opacity=1 ]  {$z-z'$};
\draw (205,405.4) node [anchor=north west][inner sep=0.75pt]  [color={rgb, 255:red, 0; green, 0; blue, 255 }  ,opacity=1 ]  {$z'$};
\draw (235.5,457.9) node [anchor=north west][inner sep=0.75pt]  [color={rgb, 255:red, 255; green, 255; blue, 255 }  ,opacity=1 ]  {$\sigma '$};

\end{tikzpicture}
    \caption{Proof of \Cref{induction}.
    Above: given $z= \partial \sigma$, we
    consider $\tilde{\sigma}$
    the $2$-simplices of $\sigma$
    not lying comfortably within $W$.
    We approximate the part $\tilde{\pi}$
    of $\tilde{z} = \partial \tilde{\sigma}$
    which lies outside of $\mathring{G}^n_M(r)$ using \Cref{approximate 1-chains} to get $\pi'$, and the difference
    between the approximation and the target is the boundary of $\hat{\sigma} \in C_2(M)$. \\
    Below: Taking $\sigma' = \tilde{\sigma} + \hat{\sigma}$ and $z' = \partial{\sigma}'$
    gives a $\sigma'$ which does not
    come too close to the boundary of $\mathcal{O} \setminus (\mathcal{D} \cup W)$, and $|z-z'| \subseteq W$.
    }
    \label{fig:decomposition}
    \end{figure}
    
    Define $\tilde{z} = \partial \tilde{\sigma}$.
    Since it is not
    a priori true that
    $\tilde{z} \in Z_1(\mathring{G}^n_M(r))$,
    we must replace part
    of $\tilde{z}$
    by an approximation
    using \Cref{approximate 1-chains}
    to obtain our desired $z'$.

    To this end, define
    \[
        \tilde{\pi} :=
        \sum_{c \notin z} t_c(\tilde{z}) c,
    \]
    which we can think of as being
    ``the set difference of $\tilde{z}$
    and $z$.''
    Then,
    \begin{align*}
        -\partial \tilde{\pi} =
        \partial(\tilde{z} - \tilde{\pi})
        =
        \sum_{c \in z} t_c(\tilde{z})\partial c
        \in C_0(\mathring{G}^n_M(r)).
    \end{align*}

    Therefore, the hypotheses
    of \Cref{approximate 1-chains}
    apply to $\tilde{\pi}$,
    and so there exists
    $\pi' \in C_1(\mathring{G}^n_M(r))$
    and $\hat{\sigma} \in C_2(M)$ such that
    \[
    \partial\hat{\sigma} = \pi' - \tilde{\pi}
    \]
    and for all $x \in |\hat{\sigma}|$,
    there exists $x' \in |\tilde{\pi}|$
    such that $d_M(x',x) \le K$.

    We claim that setting
    \[
        \sigma' = \tilde{\sigma} + \hat{\sigma}
    \]
    and
    \[
        z' = (\tilde{z} - \tilde{\pi}) + \pi'
    \]
    gives our desired conclusion.

    First, note that
    \[
        \partial \sigma'
        = \partial \tilde{\sigma}
        + \partial \hat{\sigma}
        = \tilde{z} + \pi' - \tilde{\pi}
        = z',
    \]
    as desired.
    This also implies that
    $\partial z' = 0$,
    and since $(\tilde{z} -\tilde{\pi}),
    \pi' \in C_1(\mathring{G}^n_M(r))$,
    we then have $z' \in Z_1(\mathring{G}^n_M(r))$,
    as desired.

    Next, we want to show
    that $|z-z'| \subseteq W$. 
    To this end, note that
    \begin{equation} \label{pi-tilde in W}
        \sum_{c \notin z} t_c(\partial(\sigma - \tilde{\sigma})) c 
        = \sum_{c \notin z} (t_c(z) - t_c(\tilde{z}))c\\
        = - \sum_{c \notin z} t_c(\tilde{z})c
        = -\tilde{\pi},
    \end{equation}
    and
    \begin{equation} \label{comfortably inside W}
        \sigma - \tilde{\sigma}
        = \sum
        \{ t_c(\sigma)c : \forall x \in |c|,
        \bar{B}(x, K) \subseteq W\}.
    \end{equation}

    Since by \eqref{pi-tilde in W} we have
    \begin{align*}
        z - z' &= z - \tilde{z} + \tilde{\pi} - \pi' \\
        &= \partial (\sigma - \tilde{\sigma})
        - \sum_{c \notin z} t_c(\partial(\sigma - \tilde{\sigma})) c - \pi' \\
        &= \sum_{c \in z} t_c(\partial(\sigma - \tilde{\sigma}))c - \pi',
    \end{align*}
    it suffices to show that
    $|\sigma - \tilde{\sigma}|,|\pi'| \subseteq W$.

    But \eqref{comfortably inside W} shows that
    for any $x \in |\sigma - \tilde{\sigma}|$,
    we have $\bar{B}(x,K) \subseteq W$,
    so in particular $|\sigma - \tilde{\sigma}| \subseteq W$.
    Moreover, \eqref{pi-tilde in W}
    shows that $|\tilde{\pi}| 
    \subseteq|\sigma - \tilde{\sigma}|$.
    Then, for any $x' \in |\pi'|$,
    there exists $x \in |\tilde{\pi}|$
    such that $d_M(x,x') \le K$,
    and so $x' \in \bar{B}(x,K) \subseteq W$,
    as desired.
    So we have shown $|z - z'| \subseteq W$.

    Lastly, let $x' \in |\sigma'| \subseteq|\tilde{\sigma}| \cup |\hat{\sigma}|$.
    We want to find $V' \in \mathcal{O} \setminus \mathcal{D}'$ such that
    $\bar{B}(x',L-d'(x')K) \subseteq V'$.

    First, suppose that $x' \in |\tilde{\sigma}|$.
    By our hypotheses, we have that there
    exists $V' \in \mathcal{O} \setminus \mathcal{D}$
    such that
    \[
        \bar{B}(x', L - d(x')K)
        \subseteq V'.
    \]
    We need to rule out 
    the possibility that $V' = W$.
    However, since 
    there is some $x$ within distance $2r$ of $x'$ lying in the same $2$-simplex such that $\bar{B}(x,K) \not\subseteq W$,
    we have that
    \[
        \bar{B}(x', L - d(x')K) \supseteq
        \bar{B}(x', L - \Delta K) \supseteq
        \bar{B}(x',K + 2r) \supseteq
        \bar{B}(x,K)
    \]
    intersects $W^c$, and therefore $V' \ne W$,
    as desired.\footnote{
    An astute reader may notice that,
    in case $\mathcal{O}\setminus \mathcal{D} = \{W\}$,
    we have just shown a contradiction.
    Thus, in this case,  $|\tilde{\sigma}|=\emptyset$; continuing with the argument below then shows that $|\hat{\sigma}| = \emptyset$ as well,
    all of which is to say,
    our desired condition on $|\sigma'|$
    holds vacuously.
    }
    Note that above we have used the fact that 
    \[
        \sup_{x \in M} d(x) \le \Delta;
    \]
    this holds because if, to the contrary, there were a sequence of $(\Delta+1)$
    points $x_1,x_2,...,x_{\Delta+1}$, each within
    distance $K$ of the next, and each
    lying in a distinct element of $\mathcal{O}$,
    we would have that $B(x_1,(\Delta+1)K)$
    intersects more than $\Delta$
    distinct elements of $\mathcal{O}$,
    contradicting our assumption on $\Delta$.
    We have also used the assumption that $L \ge (\Delta +1)K + 2r$.

    Then, since $d'(x') \ge d(x')$, 
    we therefore also have
    \[
        \bar{B}(x', L - d'(x')K)
        \subseteq V',
    \]
    with $V' \in \mathcal{O} \setminus \mathcal{D}'$, as desired.

    Next, suppose that $x' \in |\hat{\sigma}|$.
    Then there exists 
    $x \in |\tilde{\pi}| \subseteq|\tilde{\sigma}| \cap |\sigma -\tilde{\sigma}|$
    with $d(x,x') \le K$ by our 
    construction of $\hat{\sigma}$.
    Since $x \in |\tilde{\sigma}|$,
    the same argument as above
    shows that
    there exists $V' \in \mathcal{O} \setminus \mathcal{D}'$
    such that 
    \[
        \bar{B}(x, L - d(x) K)
        \subseteq V'.
    \]
    Since $d(x,x') \le K$
    and $x \in |\sigma - \tilde{\sigma}| \subseteq W$,
    we also have that
    $d'(x') \ge d(x)+1$, and so
    \[
        \bar{B}(x', L - d'(x')K)
        \subseteq \bar{B}(x,
        L - [d'(x')-1]K)
        \subseteq \bar{B}(x,L - d(x)K)
        \subseteq V',
    \]
    as desired.
\end{proof}

\section{Existence of an appropriate cover}

Now we prove that there exists a finite open cover of $M$ by sets of $O(r|\log r|)$ diameter
which satisfies the
hypotheses of \cref{decomposition}.

\begin{lemma} \label[lemma]{good cover exists}
    For any $\epsilon >0$,
    there exists $r_0(\epsilon,M)$
    such that the following holds
    for all
    $0 < r \le r_0$.
    Taking $\Delta = \lfloor e^{2d\epsilon} 13^d \rfloor $,
    $K=4 Q \rho r |\log r|$
    and $L = (\Delta + 1)K + 2r$,
    there exists a finite open cover
    $\mathcal{O}$
    of $M$
    such that
    for each $x \in M$
    there exists $U$ in $\mathcal{O}$
    such that
    $\bar{B}(x,L) \subseteq U$,
    and we have
    \[
        \sup_{x \in M}
        \#
        \{U \in \mathcal{O}:B(x,(\Delta+1)K) \cap U \ne \emptyset \}
        \le 
        \Delta.
    \]
    Moreover, for each $U \in \mathcal{O}$,
    we have
    $\diam(U) \le 
    (\Delta + 1)16 Q \rho r |\log r| + 8r$.
\end{lemma}
\begin{proof}
    Take $\Delta = \lfloor e^{2d\epsilon} 13^d \rfloor $,
    $K=4Q\rho r |\log r|$, and
    $L = (\Delta + 1)K + 2r$, and let $R = 2L$.

    Let $X \subseteq M$
    be a maximal $\frac{R-(3L/2)}{2}$-packing;
    that is, for all $x,y \in X$ with $x \ne y$, 
    $B(x, \frac{R-(3L/2)}{2}) \cap B(y, \frac{R-(3L/2)}{2}) = \emptyset$,
    and for any $z \notin X$,
    $B(z,\frac{R-(3L/2)}{2})$
    intersects $B(x,\frac{R-(3L/2)}{2})$
    for some $x \in X$.

    The maximality condition
    implies that $X$ is $(R-(3L/2))$-dense
    in $M$, that is, for all $z \in M$,
    there is $x \in X$ such 
    that $d_M(x,z) \le R-(3L/2)$.
    
    We claim that therefore
    \[
        \mathcal{O} :=
        \{ B(x,R) : x \in X\}
    \]
    has the property that
    for any $p \in M$,
    there is some $U \in \mathcal{O}$
    such that $\bar{B}(p,L) \subseteq U$.
    To see this, take $x \in X$
    such that $d_M(x,p) \le R - (3L/2)$
    and note that by the triangle
    inequality, $\bar{B}(p,L) \subseteq B(p,3L/2) \subseteq B(x,R)$.

    We now want to bound the 
    ply of $\mathcal{O}$.
    Choose $0< r_0$
    sufficiently small
    that for all 
    $s < 100L(r_0)$, we have that
    for all $p \in M$,
    \[
        \left|\log \left(\frac{\mathrm{vol}B_M(p,s)}{\mathrm{vol}B_{\R^d}(0,s)}\right)\right|
        \le d\epsilon.\footnote{
        Referring to (for instance) Theorem 3.1 of
        \cite{Gray1974},
        one can see that as $s \to 0$ the volume of any ball $B(p,s)$ in $M$ approaches the volume of a Euclidean ball of radius $s$,
        and that the error in this approximation is controlled by the scalar Ricci curvature at $p$.
        Since $M$ is compact,
        we have a uniform scalar Ricci curvature bound, and therefore for any $\epsilon$
        we can choose $r_0$
        sufficiently small
        that the desired inequality holds.
        }
    \]
    For $p \in M$,
    we want to bound
    the number of distinct
    $x \in X$ such that
    $B(p,(\Delta+1)K) \cap B(x,R) \ne \emptyset$,
    or equivalently,
    such that $x \in B(p,(\Delta+1)K+R)$.
    Since $X$ is a
    $\frac{R-(3L/2)}{2}$-packing,
    we have the inclusion of
    disjoint sets
    \begin{align*}
        \bigsqcup_{x \in X \cap B(p,(\Delta+1)K +R)}
        B\left(x, \frac{R-(3L/2)}{2}\right)
        \subset
        B\left(p, (\Delta+1)K + R +\frac{R - (3L/2)}{2}\right)
    \end{align*}
    which implies the volume bound
    \begin{align*}
        \sum_{x \in X \cap B(p,(\Delta+1)K +R)}
        \mathrm{vol}B\left(x, \frac{R-(3L/2)}{2}\right)
        \le 
        \mathrm{vol} B\left(p,\frac{2(\Delta+1)K+3R-(3L/2)}{2}\right)
    \end{align*}
    and so using
    the comparison with Euclidean 
    volume and rearranging gives
    \begin{align*}
        \# \{ U \in \mathcal{O}:
        B(p,(\Delta+1)K) \cap U \ne \emptyset \}
        &=
        |X \cap B(p,(\Delta+1)K+R)|
        \\
        &\le 
        e^{2d\epsilon}
        \frac{\mathrm{vol} B_{\R^d}(\frac{2(\Delta+1)K+3R - (3L/2)}{2} )}{\mathrm{vol} B_{\R^d}(\frac{R - (3L/2)}{2})} \\
        &= e^{2d\epsilon} \left( \frac{4(\Delta+1)K + 6R - 3L}{2R - 3L}\right)^d \\
        &\le 
        e^{2d\epsilon}
        \left( \frac{L + 6R}{2R - 3L}
        \right)^d \\
        &= e^{2d\epsilon} 13^d,
    \end{align*}
    and since the quantity behing bounded is an integer,
    it is at most $\lfloor e^{2d\epsilon} 13^d \rfloor = \Delta$,
    as desired.

    Finally, each $U=B(x,R) \in \mathcal{O}$
    has diameter
    at most $2R = 4L =(\Delta + 1)16 Q \rho r |\log r| + 8r$,
    as desired.
\end{proof}

\section{$\ker(Z_1(G^n_M(r)) \to H_1(M))$ is generated by small cycles} \label{sec:characterize kernel}
We now combine our work above to show 
that the kernel of the natural map is indeed generated by small cycles on the event $A_{Q,\rho}(n,r)$.

\begin{prop} \label[prop]{choice of Q'}
    Fix $Q, \rho \ge 1$.
    Let $r_0$ be sufficiently small
    that \Cref{good cover exists}
    holds for $\epsilon=1/2$,
    \eqref{r small enough}
    holds,
    and $-\log r_0 \ge 8$.
    Then, on the
    event $A_{Q, \rho}(n,r)$, 
    taking 
    $Q' = [(13e)^d+1]16Q\rho+1$,
    we have
    \[
        Z_1^{Q' r|\log r|}(G^n_M(r))
        \supseteq
        \ker(Z_1(G^n_M(r)) \to H_1(M)).
    \]
\end{prop}
\begin{proof}
    Assume $A_{Q,\rho}$ holds.
    Let $z \in Z_1(G^n_M(r))$
    be in the kernel of the natural map;
    that is, there
    exists $\sigma \in C_2(M)$
    such that $\partial \sigma = z$.

    First,
    consider the case
    that $z \in Z_1(\mathring{G}^n_M(r))$.
    Let $\mathcal{O}$
    be the cover
    constructed in \Cref{good cover exists}.
    Then $\mathcal{O}$
    satisfies the hypotheses of
    \Cref{decomposition},
    and each cover element
    has diameter at most
    $[[(13e)^d+1]16Q\rho+1]r|\log r|=Q'r|\log r|$.
    By \Cref{decomposition}
    we have a decomposition
    \[
        z = \sum_{U \in \mathcal{O}} z_U
    \]
    where each $z_U \in Z_1(\mathring{G}^n_M(r))$
    and each $|z_U| \subseteq U$;
    in particular,
    each $z_U \in Z_1^{Q'r|\log r|}(G^n_M(r))$
    and therefore
    $z \in Z_1^{Q'r|\log r|}(G^n_M(r))$,
    as desired.

    Now consider a
    general
    $z \in Z_1(G^n_M(r)) \cap B_1(M)$.
    We have a decomposition
    \[
        z = \sum_{\kappa} z_{\kappa},
    \]
    where the sum is over
    connected components
    $\kappa$ of $G^n_M(r)$,
    each $z_{\kappa} \in Z_1(G^n_M(r))$,
    and each $|z_\kappa| \subseteq \kappa$.\footnote{To see that this
    decomposition holds,
    for each connected component
    $\kappa$, let $C_0(\kappa)\subseteq C_0(G^n_M(r))$
    be the collection of $0$-chains
    supported on $\kappa$.
    Note that each $1$-simplex
    $\ell \in C_1(G^n_M(r))$
    necessarily belongs to 
    a single connected component
    $\kappa$.
    Therefore 
    we can set $z_\kappa := \sum_{\ell \in z, \partial \ell \in C_0(\kappa)} t_\ell(z)\ell$
    and get the desired decomposition.
    The fact that each
    $\partial z_\kappa = 0$
    comes from the fact
    that each $\partial z_\kappa \in C_0(\kappa)$,
    $\partial z = \sum_\kappa \partial z_\kappa = 0$,
    and $C_0(G^n_M(r)) = \bigoplus_{\kappa} C_0(\kappa)$
    is a direct sum decomposition.}

    If $\kappa \ne \mathring{G}^n_M(r)$,
    then, by $A_{Q,\rho}(n,r)$,
    $\diam(\kappa) \le Qr|\log r|
    \le Q' r|\log r|$,
    and thus
    $z_\kappa \in Z^{Q' r|\log r|}_1(G^n_M(r))$.
    Moreover,
    by our assumption
    on $r_0$
    relative to the injectivity
    radius on $M$,
    $z_\kappa$
    is supported on a contractible
    subset of $M$
    and therefore
    there exists $\sigma_\kappa \in C_2(M)$
    such that $\partial \sigma_\kappa = z_\kappa$.

    Thus
    \[
       \sum_{\kappa \ne \mathring{G}^n_M(r)} z_\kappa
       \in Z^{Q'r|\log r|}_1(G^n_M(r)) \cap B_1(M).
    \]
    Since by assumption $z \in B_1(M)$,
    we then have
    \[
        z_{\mathring{G}^n_M(r)}
        = z - \sum_{\kappa \ne \mathring{G}^n_M(r)}
        z_\kappa \in B_1(M).
    \]
    But we showed above
    that 
    if $z_{\mathring{G}^n_M(r)}
    \in Z_1(\mathring{G}^n_M(r)) 
    \cap B_1(M)$,
    then $z_{\mathring{G}^n_M(r)}
    \in Z^{Q' r|\log r|}_1(G^n_M(r))$,
    and so we are done.
\end{proof}

\section{Proof of \Cref{surjectivity} and \Cref{isomorphism}} \label{sec:main proofs}
\begin{proof}[Proof of \Cref{surjectivity}]
    Recall that $G(n;r)$
    is the random geometric graph whose underlying vertex set is $n$ points sampled independently from $dM$,
    while $G^n_M(r)$ is the random geometric graph whose underlying vertex set is a Poisson point process on $M$ with intensity measure $ndM$.
    Let $N \sim \mathrm{Poi}(n)$ be the number of points in the Poisson point process underlying $G^n_M(r)$.
    Recall that $G_M(n;r)$
    is equal in distribution to $G^n_M(r)|N=n$,
    and that by Stirling's formula we have $\Prob(N=n) = (1 + o(1))(2\pi n)^{-1/2}$.

    Next, note that the event
    \[
    \{ Z_1(G^n_M(r)) \to H_1(M) \mbox{ is surjective} \}
    \]
    is increasing in $r$, while the event
    \[
    \{ Z_1(G^n_M(r)) \to H_1(M) \mbox{ is the } 0 \mbox{ map} \}
    \]
    is decreasing.
    Thus, it will suffice to prove \Cref{surjectivity} under the assumption that $r = \lam_0^{1/d} n^{-1/d}$.

    Assume $\lam_0 > \lam_c$
    and $r = \lam_0^{1/d} n^{-1/d}$.
    Choose $0 < \kappa < \infty$
    sufficiently large that
    $\kappa/d>1/2$.
    By \Cref{key event is likely},
    there exist $1 \le Q, \rho < \infty$
    and $r_0'>0$
    depending
    only on $\lam_0, \kappa,$
    and $M$
    such that $\Prob(A_{Q,\rho}(n,r)^c) \le r^\kappa$ whenever $0 < r \le r_0'$.
    Take $r_0 \in (0, r_0']$
    sufficiently small
    that \eqref{r small enough} holds.
    Then, by \Cref{prop:surj},
    when $n$
    is sufficiently large
    that $r=\lam_0^{1/d} n^{-1/d} \le r_0$,
    we have
    \[
        \Prob(Z_1(G^n_M(r)) \to H_1(M) \mbox{ is not surjective})
        \le \Prob(A_{Q,\rho}(n;r)^c) \le r^\kappa.
    \]

    But then we have
    \begin{align*}
        \Prob\left(
        \begin{array}{c}
        Z_1(G_M(n;r) \to H_1(M) \\
        \mbox{is not surjective}
        \end{array}
        \right)
        &=
        \Prob\left(
        \begin{array}{c}
        Z_1(G^n_M(r)) \to H_1(M) \\
        \mbox{is not surjective}
        \end{array}
        \middle| N = n
        \right)
        \\
        &\le r^\kappa O(n^{1/2})
        = O(n^{1/2 - \kappa/d}),
    \end{align*}
    which tends to $0$ as $n \to \infty$ by our choice of $\kappa$, as desired.

    Now suppose $\lam_0 < \lam_c$
    and $r=\lam_0^{1/d} n^{-1/d}$.
    Any component 
    with diameter at most $\mathrm{inj}(M)$ cannot support a representative of a nontrivial element of $H_1(M)$.
    Since $\lam_0 < \lam_c$ is subcritical, we therefore have
    \begin{align*}
        \Prob(Z_1(G^n_M(r)) \to H_1(M)
        \mbox{ is nonzero}) 
        &\le
        \Prob\left(
        \begin{array}{c}
        G^n_M(r) \mbox{ contains a
        component of } \\ d_M\mbox{-diameter at least } \mathrm{inj}(M)
        \end{array}
        \right) \\
        &\le
        \Prob\left(
        \begin{array}{c}
        G^n_M(r) \mbox{ contains a component} \\ \mbox{with at least } \mathrm{inj}(M)r^{-1} \mbox{ vertices}
        \end{array}
        \right) \\
        &= O(\exp(-cr^{-1}))
        =O(\exp(-c' n^{1/d})
    \end{align*}
    by e.g. Lemma 10.2 of Penrose \cite{Pen03} (and our usual arguments using local coordinates, here made easier by monotonicity).
    The same conditioning trick as above then gives our desired final result.
\end{proof}

\begin{proof}[Proof of \Cref{isomorphism}]
    Let $\lam_0 > \lam_c$.
    First, note that
    the proof of \Cref{surjectivity}
    actually shows that,
    if $nr^d \ge \lam_0 > \lam_c$,
    the probability that surjectivity
    fails decays faster than
    any polynomial in $n^{-1}$;
    therefore, for any $\kappa < \infty$,
    there is some $r_0^{(1)} > 0$
    depending only on $\kappa, \lam_0,$ and $M$ such that
    \[ \Prob(Z_1(G_M(n;r) \to H_1(M) \mbox{ is not surjective})
    \le r^\kappa \]
    if $nr^d \ge \lam_0$ and $0 < r \le r_0^{(1)}$
    (since $nr^d \ge \lam_0$ means $n^{-1}=O(r^d)$).

    Therefore, by the First Isomorphism Theorem, it suffices to find $Q' < \infty$
    and $r_0>0$ such that
    if $nr^d \ge \lam_0$
    and $0 < r \le r_0$, then
    \[
        \ker (Z_1(G_M(n;r)) \to H_1(M))
        = Z_1^{Q' r|\log r|}(G_M(n;r)),
    \]
    except with probability
    at most $r^\kappa$.

    We first cover the
    case that
    $n \le r^{-(d+1)}$.
    Given $0 < \kappa < \infty$,
    choose $0 < \kappa' < \infty$
    so that $\kappa' - (d+1)/2 > \kappa$.
    Then, using \Cref{key event is likely}, we obtain $1 \le Q,\rho < \infty$ 
    depending
    only on $\lam_0, \kappa',$ and $d$
    and $r_0^{(2)} \in (0, r_0^{(1)}]$
    depending only on $\lam_0, \kappa',$ and $M$
    so that whenever $0<r \le r_0^{(2)}$
    and $nr^d \ge \lam_0$,
    we have $\Prob(A_{Q,\rho}(r)) \ge 1 - r^{\kappa'}$.
    Then take $Q'_1$ and $r_0^{(3)} \in (0,r_0^{(2)}]$
    as in \Cref{choice of Q'}
    so that
    \[
    \Prob(\ker (Z_1(G^n_M(r)) \to H_1(M))
        \subseteq Z_1^{Q'_1 r|\log r|}(G^n_M(r))
        \ge
    \Prob(A_{Q,\rho}(n,r)) \ge 1 - r^{\kappa'}.
    \]
    for all $nr^d \ge \lam_0, 0 < r \le r_0^{(3)}$.
    Then, if $n \le r^{-(d+1)}$,
    standard de-Poissonization (as in the proof of \cref{surjectivity} above) gives
    \begin{align*}
        &\Prob(\ker (Z_1(G_M(n;r)) \to H_1(M))
        \not\subseteq Z_1^{Q'_1 r|\log r|}(G_M(n;r))) \\
        &=
        \Prob(\ker (Z_1(G^n_M(r)) \to H_1(M))
        \not\subseteq Z_1^{Q'_1 r|\log r|}(G^n_M(r))|N=n) \\
        &\le
        n^{1/2}r^{\kappa'} \le r^\kappa
    \end{align*}
    when $nr^d \ge \lam_0,0 < r \le r_0^{(3)}$.

    Now consider the ``dense'' case where 
    $n \ge r^{-(d+1)}$.
    Then for each $r>0$
    we can cover $M$ by a collection of
    $O(r^{-d})$ balls of radius $r/2$,
    and with high probability
    as $n \to \infty$, each ball
    contains at least one vertex of
    $G_M(n;r)$. Indeed, for each such ball $B$ we have
    \begin{align*}
        \Prob(B \cap G_M(n;r) = \emptyset)
        &= (1 - \mathrm{vol}(B))^n \\
        &\le e^{-n\mathrm{vol}(B)} 
        = e^{-n \Theta(r^d)} \\
        &= e^{-\Omega(r^{-1})}
        =o(r^{d+\kappa}),
    \end{align*}
    and therefore a union bound
    gives
    \[
    \Prob(\mbox{some } B \mbox{ is unoccupied}) = o(r^\kappa).
    \]
    It is not
    hard to check that
    if every ball $B$
    is occupied, then
    $G_M(n;r)$ is connected and satisfies the
    event $A_{2,2}(n,r)$
    (with $G_M^n(r)$ replaced by $G_M(n;r)$).\footnote{In fact, in this ``dense'' case, the giant component (which is the only
    component) is distance $O(r)$
    away from every point,
    rather than $O(r|\log r|)$.
    Therefore, if we wanted,
    we could show that the kernel
    of the natural map
    is generated by cycles of 
    diameter $O(r)$ with
    high probability when
    $n \ge r^{-(d+1)}$.}
    Then, the same arguments
    as for \Cref{choice of Q'}
    show that we have a $Q'_2$
    (depending only on $\kappa$
    and $d$)
    and $r_0^{(4)} \in (0, r_0^{(3)}]$
    (depending only on $\kappa$
    and $M$)
    such that
    \[
        \Prob(\ker (Z_1(G_M(n;r)) \to H_1(M))
        \not\subseteq Z_1^{Q'_2 r|\log r|}(G_M(n;r)) \le r^{\kappa}
    \]
    for $n \ge r^{-(d+1)}, 0 < r \le r_0^{(4)}$.

    Thus, we have
    found $r_0^{(4)}>0$
    depending only on $\lam_0, \kappa, M$
    and $Q' = \max(Q'_1,Q'_2)$
    depending only on $\lam_0,
    \kappa,d$
    such that 
    if $nr^d \ge \lam_0$ 
    and $0 <r \le r_0^{(4)}$,
    the kernel of the natural map
    is generated by cycles
    of diameter $Q' r|\log r|$
    except with probability
    at most $r^\kappa$.
    On the other hand,
    if $r_0 \in (0,r_0^{(4)}]$ is sufficiently small that for all $0<r \le r_0$ we have
    \[
        Q' r|\log r|
        \le \mathrm{inj}(M),
    \]
    then any cycle $z \in Z_1^{Q' r |\log r|}(M)$
    necessarily lies in $B_1(M)$,
    so
    \[
        Z_1^{Q' r|\log r|}(G_M(n;r))
        \subseteq \ker(Z_1(G_M(n;r)) \to H_1(M))
    \]
    deterministically. 
    Thus, for all $nr^d \ge \lam_0$
    and $0 < r \le r_0$ we have
    \[
        \Prob(
        Z_1^{Q' r|\log r|}(G_M(n;r))
        = \ker(Z_1(G_M(n;r)) \to H_1(M)) )
        \ge 1 - r^\kappa,
    \]
    as desired.
    The First Isomorphism Theorem and our work above then gives that,
    for all $nr^d \ge \lam_0$
    and $0 < r \le r_0$,
    \begin{align*}
        &\Prob( H_1(M) \cong Z_1(G_M(n;r))/Z_1^{Q' r|\log r|}(G_M(n;r)) ) \\
        &\ge
        \Prob\left(
        \left\{\begin{array}{c}
        Z_1(G_M(n;r)) \to H_1(M) \\
        \mbox{is surjective}
        \end{array}\right\}
        \cap
        \left\{
        \begin{array}{c}
        \ker( Z_1(G_M(n;r)) \to H_1(M) )
        = Z_1^{Q' r|\log r|}(G_M(n;r))
        \end{array}
        \right\}
        \right) \\
        &\ge 1 - 2r^\kappa.
    \end{align*}
\end{proof}

\section{Proof of \Cref{need at least rlogr}} \label{sec: counterexample}
In this section we show that $O(r|\log r|)$ is the ``right'' scale at which to quotient out cycles,
in the sense that quotienting out by cycles of diameter $o(r|\log r|)$
is not sufficient to obtain $H_1(M)$.

The idea behind the proof of \Cref{need at least rlogr}
is as follows.
As usual, taking coordinates and rescaling, we will first show an analogous statement for a (slightly inhomogeneous) percolation process on $\R^d$; that is,
we will show that in a box of side length $\Omega(R)$
($R = 1/r$),
it is very likely to have an irreducible cycle of diameter $\Omega(\log R)$.
We can force the existence of an irreducible loop of length $c \log R$ by prescribing that an explicit configuration of boxes be occupied, and that a ``shell'' around this configuration is vacant; by making $c$ sufficiently small, we can make this probability decay more slowly than any polynomial in $R^{-1}=r$.
To connect this to the giant component in the supercritical regime, we can also add a small tunnel from the bottom of the loop through the shell, and then ask that there exists a path from that tunnel of diameter at least $Q \log R$ in a region disjoint from the loop configuration.
The probability of the existence of such a long path is bounded below by a constant in the supercritical regime, since it is implied by a connection to infinity in the lower half-space.
Since the prescribed event depends on a region of diameter $O(\log R)$, there are $\Omega((R/\log R)^d)$
disjoint regions where such an event can happen in 
a box of side length $\Omega(R)$.
Thus, if $c>0$ is chosen sufficiently small that
the probability of success for any one region is $\omega((R/\log R)^{-d})$, we will have at least one success with very high probability as $R \to \infty$ ($r \to 0$).

The least obvious part of the argument above is how to show that our prescribed configuration does have a cycle which is not a sum of small cycles, and that we can observe this event \emph{locally} on a patch of $\R^d$ (i.e. our cycle still cannot be written as a sum of small cycles even if we include cycles not contained entirely within the local patch). This purely topological argument
is given in \cref{topology}.
The rest of the argument outlined above goes through in a straightforward manner.

We now begin the constructions necessary for
the proof of \cref{need at least rlogr}.
Throughout this section, let $\sigma$ be a metric on $\R^d$
with 
\begin{equation} \label{eq:metric condition}   
    \sigma(x,y) \le \|x - y\| \le 2 \sigma(x,y)
\end{equation}
for all $x,y \in \R^d$,
let $\nu$ be a measure on $\R^d$ whose
Radon-Nikodym derivative $\frac{d\nu}{d\mathcal{L}}$
with respect to Lebesgue measure satisfies
\begin{equation} \label{eq:volume condition}
        \lam_- \le \frac{ d\nu}{d\mathcal{L}} \le \lam_+
\end{equation}
for some $0 < \lam_- \le \lam_+ < \infty$,
let $X^\nu$ be a Poisson point process on $\R^d$
with intensity measure $\nu$, and let
$G^\nu$ be the random geometric graph
whose vertices are the points of $X^\nu$
and whose edges connect pairs of points
with $\sigma$-distance at most 1.
For $\lam \in (0,\infty)$

Fix $a = \frac{1}{2\sqrt{d}}$. For any $z \in a\Z^d$,
define $K_z = z + [-a/2,a/2]^d$.
Note that if two adjacent boxes $K_z, K_{z'}$
are occupied, every vertex of $K_z \cap G^\nu$
is connected to every vertex of $K_{z'} \cap G^\nu$.

Fix a constant $c>0$ to be determined later.
Define the rectangle
\[
    \mathcal{R} := [-3,3] \times [0,c \log R] \times \{0\}^{d-2}.
\]
Also define the interval
\[
    I := \{0\} \times [-3,0] \times \{0\}^{d-2}.
\]
We then define the family of cubes
\[
    \mathcal{K} := \{ K_z : z \in a\Z^d, K_z \cap (\partial \mathcal{R} \cup I) \ne 0 \}.
\]

We also define the slightly larger rectangular neighborhood of $\mathcal{R}$:
\[
    \tilde{\mathcal{R}} := [-6,6] \times [-3, c\log R + 3] \times [-3,3]^{d-2},
\]
and the ``shell'' around $\partial \mathcal{R} \cup I$
that we will force to be empty we take to be
\[
    \mathcal{N} := \tilde{\mathcal{R}} \setminus \bigcup_{K \in \mathcal{K}} K.
\]

\begin{figure}
    
\centering

\tikzset{every picture/.style={line width=0.75pt}} 

\begin{tikzpicture}[x=0.75pt,y=0.75pt,yscale=-1,xscale=1]


\draw  [fill={rgb, 255:red, 210; green, 210; blue, 220 }  ,fill opacity=1 ] (260,10) -- (380,10) -- (380,330) -- (260,330) -- cycle ;
\draw   (290,40) -- (350,40) -- (350,290) -- (290,290) -- cycle ;

\draw[decoration={brace,raise=15pt},decorate]
  (260,295) -- node[left=15pt] {$c \log R$} (260,35);

\draw  [fill={rgb, 255:red, 208; green, 2; blue, 27 }  ,fill opacity=1 ] (315,285) -- (325,285) -- (325,295) -- (315,295) -- cycle ;
\draw  [fill={rgb, 255:red, 208; green, 2; blue, 27 }  ,fill opacity=1 ] (325,285) -- (335,285) -- (335,295) -- (325,295) -- cycle ;
\draw  [fill={rgb, 255:red, 208; green, 2; blue, 27 }  ,fill opacity=1 ] (335,285) -- (345,285) -- (345,295) -- (335,295) -- cycle ;
\draw  [fill={rgb, 255:red, 208; green, 2; blue, 27 }  ,fill opacity=1 ] (305,285) -- (315,285) -- (315,295) -- (305,295) -- cycle ;
\draw  [fill={rgb, 255:red, 208; green, 2; blue, 27 }  ,fill opacity=1 ] (295,285) -- (305,285) -- (305,295) -- (295,295) -- cycle ;
\draw  [fill={rgb, 255:red, 208; green, 2; blue, 27 }  ,fill opacity=1 ] (285,285) -- (295,285) -- (295,295) -- (285,295) -- cycle ;
\draw  [fill={rgb, 255:red, 208; green, 2; blue, 27 }  ,fill opacity=1 ] (285,275) -- (295,275) -- (295,285) -- (285,285) -- cycle ;
\draw  [fill={rgb, 255:red, 208; green, 2; blue, 27 }  ,fill opacity=1 ] (285,265) -- (295,265) -- (295,275) -- (285,275) -- cycle ;
\draw  [fill={rgb, 255:red, 208; green, 2; blue, 27 }  ,fill opacity=1 ] (285,255) -- (295,255) -- (295,265) -- (285,265) -- cycle ;
\draw  [fill={rgb, 255:red, 208; green, 2; blue, 27 }  ,fill opacity=1 ] (285,245) -- (295,245) -- (295,255) -- (285,255) -- cycle ;
\draw  [fill={rgb, 255:red, 208; green, 2; blue, 27 }  ,fill opacity=1 ] (285,235) -- (295,235) -- (295,245) -- (285,245) -- cycle ;
\draw  [fill={rgb, 255:red, 208; green, 2; blue, 27 }  ,fill opacity=1 ] (285,225) -- (295,225) -- (295,235) -- (285,235) -- cycle ;
\draw  [fill={rgb, 255:red, 208; green, 2; blue, 27 }  ,fill opacity=1 ] (285,215) -- (295,215) -- (295,225) -- (285,225) -- cycle ;
\draw  [fill={rgb, 255:red, 208; green, 2; blue, 27 }  ,fill opacity=1 ] (285,205) -- (295,205) -- (295,215) -- (285,215) -- cycle ;
\draw  [fill={rgb, 255:red, 208; green, 2; blue, 27 }  ,fill opacity=1 ] (285,195) -- (295,195) -- (295,205) -- (285,205) -- cycle ;
\draw  [fill={rgb, 255:red, 208; green, 2; blue, 27 }  ,fill opacity=1 ] (285,185) -- (295,185) -- (295,195) -- (285,195) -- cycle ;
\draw  [fill={rgb, 255:red, 208; green, 2; blue, 27 }  ,fill opacity=1 ] (285,175) -- (295,175) -- (295,185) -- (285,185) -- cycle ;
\draw  [fill={rgb, 255:red, 208; green, 2; blue, 27 }  ,fill opacity=1 ] (285,165) -- (295,165) -- (295,175) -- (285,175) -- cycle ;
\draw  [fill={rgb, 255:red, 208; green, 2; blue, 27 }  ,fill opacity=1 ] (285,155) -- (295,155) -- (295,165) -- (285,165) -- cycle ;
\draw  [fill={rgb, 255:red, 208; green, 2; blue, 27 }  ,fill opacity=1 ] (285,145) -- (295,145) -- (295,155) -- (285,155) -- cycle ;
\draw  [fill={rgb, 255:red, 208; green, 2; blue, 27 }  ,fill opacity=1 ] (285,135) -- (295,135) -- (295,145) -- (285,145) -- cycle ;
\draw  [fill={rgb, 255:red, 208; green, 2; blue, 27 }  ,fill opacity=1 ] (285,125) -- (295,125) -- (295,135) -- (285,135) -- cycle ;
\draw  [fill={rgb, 255:red, 208; green, 2; blue, 27 }  ,fill opacity=1 ] (285,115) -- (295,115) -- (295,125) -- (285,125) -- cycle ;
\draw  [fill={rgb, 255:red, 208; green, 2; blue, 27 }  ,fill opacity=1 ] (285,105) -- (295,105) -- (295,115) -- (285,115) -- cycle ;
\draw  [fill={rgb, 255:red, 208; green, 2; blue, 27 }  ,fill opacity=1 ] (285,95) -- (295,95) -- (295,105) -- (285,105) -- cycle ;
\draw  [fill={rgb, 255:red, 208; green, 2; blue, 27 }  ,fill opacity=1 ] (285,85) -- (295,85) -- (295,95) -- (285,95) -- cycle ;
\draw  [fill={rgb, 255:red, 208; green, 2; blue, 27 }  ,fill opacity=1 ] (285,75) -- (295,75) -- (295,85) -- (285,85) -- cycle ;
\draw  [fill={rgb, 255:red, 208; green, 2; blue, 27 }  ,fill opacity=1 ] (285,65) -- (295,65) -- (295,75) -- (285,75) -- cycle ;
\draw  [fill={rgb, 255:red, 208; green, 2; blue, 27 }  ,fill opacity=1 ] (285,55) -- (295,55) -- (295,65) -- (285,65) -- cycle ;
\draw  [fill={rgb, 255:red, 208; green, 2; blue, 27 }  ,fill opacity=1 ] (285,45) -- (295,45) -- (295,55) -- (285,55) -- cycle ;
\draw  [fill={rgb, 255:red, 208; green, 2; blue, 27 }  ,fill opacity=1 ] (285,35) -- (295,35) -- (295,45) -- (285,45) -- cycle ;
\draw  [fill={rgb, 255:red, 208; green, 2; blue, 27 }  ,fill opacity=1 ] (345,275) -- (355,275) -- (355,285) -- (345,285) -- cycle ;
\draw  [fill={rgb, 255:red, 208; green, 2; blue, 27 }  ,fill opacity=1 ] (345,265) -- (355,265) -- (355,275) -- (345,275) -- cycle ;
\draw  [fill={rgb, 255:red, 208; green, 2; blue, 27 }  ,fill opacity=1 ] (345,255) -- (355,255) -- (355,265) -- (345,265) -- cycle ;
\draw  [fill={rgb, 255:red, 208; green, 2; blue, 27 }  ,fill opacity=1 ] (345,245) -- (355,245) -- (355,255) -- (345,255) -- cycle ;
\draw  [fill={rgb, 255:red, 208; green, 2; blue, 27 }  ,fill opacity=1 ] (345,235) -- (355,235) -- (355,245) -- (345,245) -- cycle ;
\draw  [fill={rgb, 255:red, 208; green, 2; blue, 27 }  ,fill opacity=1 ] (345,225) -- (355,225) -- (355,235) -- (345,235) -- cycle ;
\draw  [fill={rgb, 255:red, 208; green, 2; blue, 27 }  ,fill opacity=1 ] (345,215) -- (355,215) -- (355,225) -- (345,225) -- cycle ;
\draw  [fill={rgb, 255:red, 208; green, 2; blue, 27 }  ,fill opacity=1 ] (345,205) -- (355,205) -- (355,215) -- (345,215) -- cycle ;
\draw  [fill={rgb, 255:red, 208; green, 2; blue, 27 }  ,fill opacity=1 ] (345,195) -- (355,195) -- (355,205) -- (345,205) -- cycle ;
\draw  [fill={rgb, 255:red, 208; green, 2; blue, 27 }  ,fill opacity=1 ] (345,185) -- (355,185) -- (355,195) -- (345,195) -- cycle ;
\draw  [fill={rgb, 255:red, 208; green, 2; blue, 27 }  ,fill opacity=1 ] (345,175) -- (355,175) -- (355,185) -- (345,185) -- cycle ;
\draw  [fill={rgb, 255:red, 208; green, 2; blue, 27 }  ,fill opacity=1 ] (345,165) -- (355,165) -- (355,175) -- (345,175) -- cycle ;
\draw  [fill={rgb, 255:red, 208; green, 2; blue, 27 }  ,fill opacity=1 ] (345,155) -- (355,155) -- (355,165) -- (345,165) -- cycle ;
\draw  [fill={rgb, 255:red, 208; green, 2; blue, 27 }  ,fill opacity=1 ] (345,145) -- (355,145) -- (355,155) -- (345,155) -- cycle ;
\draw  [fill={rgb, 255:red, 208; green, 2; blue, 27 }  ,fill opacity=1 ] (345,135) -- (355,135) -- (355,145) -- (345,145) -- cycle ;
\draw  [fill={rgb, 255:red, 208; green, 2; blue, 27 }  ,fill opacity=1 ] (345,125) -- (355,125) -- (355,135) -- (345,135) -- cycle ;
\draw  [fill={rgb, 255:red, 208; green, 2; blue, 27 }  ,fill opacity=1 ] (345,115) -- (355,115) -- (355,125) -- (345,125) -- cycle ;
\draw  [fill={rgb, 255:red, 208; green, 2; blue, 27 }  ,fill opacity=1 ] (345,105) -- (355,105) -- (355,115) -- (345,115) -- cycle ;
\draw  [fill={rgb, 255:red, 208; green, 2; blue, 27 }  ,fill opacity=1 ] (345,95) -- (355,95) -- (355,105) -- (345,105) -- cycle ;
\draw  [fill={rgb, 255:red, 208; green, 2; blue, 27 }  ,fill opacity=1 ] (345,85) -- (355,85) -- (355,95) -- (345,95) -- cycle ;
\draw  [fill={rgb, 255:red, 208; green, 2; blue, 27 }  ,fill opacity=1 ] (345,75) -- (355,75) -- (355,85) -- (345,85) -- cycle ;
\draw  [fill={rgb, 255:red, 208; green, 2; blue, 27 }  ,fill opacity=1 ] (345,65) -- (355,65) -- (355,75) -- (345,75) -- cycle ;
\draw  [fill={rgb, 255:red, 208; green, 2; blue, 27 }  ,fill opacity=1 ] (345,55) -- (355,55) -- (355,65) -- (345,65) -- cycle ;
\draw  [fill={rgb, 255:red, 208; green, 2; blue, 27 }  ,fill opacity=1 ] (345,45) -- (355,45) -- (355,55) -- (345,55) -- cycle ;
\draw  [fill={rgb, 255:red, 208; green, 2; blue, 27 }  ,fill opacity=1 ] (295,35) -- (305,35) -- (305,45) -- (295,45) -- cycle ;
\draw  [fill={rgb, 255:red, 208; green, 2; blue, 27 }  ,fill opacity=1 ] (305,35) -- (315,35) -- (315,45) -- (305,45) -- cycle ;
\draw  [fill={rgb, 255:red, 208; green, 2; blue, 27 }  ,fill opacity=1 ] (315,35) -- (325,35) -- (325,45) -- (315,45) -- cycle ;
\draw  [fill={rgb, 255:red, 208; green, 2; blue, 27 }  ,fill opacity=1 ] (325,35) -- (335,35) -- (335,45) -- (325,45) -- cycle ;
\draw  [fill={rgb, 255:red, 208; green, 2; blue, 27 }  ,fill opacity=1 ] (335,35) -- (345,35) -- (345,45) -- (335,45) -- cycle ;
\draw  [fill={rgb, 255:red, 208; green, 2; blue, 27 }  ,fill opacity=1 ] (345,35) -- (355,35) -- (355,45) -- (345,45) -- cycle ;
\draw  [fill={rgb, 255:red, 208; green, 2; blue, 27 }  ,fill opacity=1 ] (345,285) -- (355,285) -- (355,295) -- (345,295) -- cycle ;
\draw  [fill={rgb, 255:red, 208; green, 2; blue, 27 }  ,fill opacity=1 ] (315,295) -- (325,295) -- (325,305) -- (315,305) -- cycle ;
\draw  [fill={rgb, 255:red, 208; green, 2; blue, 27 }  ,fill opacity=1 ] (315,305) -- (325,305) -- (325,315) -- (315,315) -- cycle ;
\draw  [fill={rgb, 255:red, 208; green, 2; blue, 27 }  ,fill opacity=1 ] (315,315) -- (325,315) -- (325,325) -- (315,325) -- cycle ;
\draw  [fill={rgb, 255:red, 208; green, 2; blue, 27 }  ,fill opacity=1 ] (315,325) -- (325,325) -- (325,335) -- (315,335) -- cycle ;
\draw [color={rgb, 255:red, 30; green, 120; blue, 30 }  ,draw opacity=1 ]   (320,330) -- (337,345) -- (330,365.5) -- (342.5,378.5) -- (361.5,383) -- (379.5,396.5) -- (389.5,414) -- (405,423.5) -- (418,439.5) -- (412,455.5) -- (424,471.5) ;

\draw (357,193.4) node [anchor=north west][inner sep=0.75pt]    {$\mathcal{\textcolor[rgb]{0.82,0.01,0.11}{K}}$};
\draw (307,117.4) node [anchor=north west][inner sep=0.75pt]  [color={rgb, 255:red, 255; green, 255; blue, 255 }  ,opacity=1 ]  {$\mathcal{N}$};


\end{tikzpicture}
\caption{On the event $E$, we require all the red boxes (elements of $\mathcal{K}$) to be occupied,
we require the gray region $\mathcal{N}$ to be empty,
and we require the existence of a long edge path headed downward from the bottom red box.
}
\end{figure}
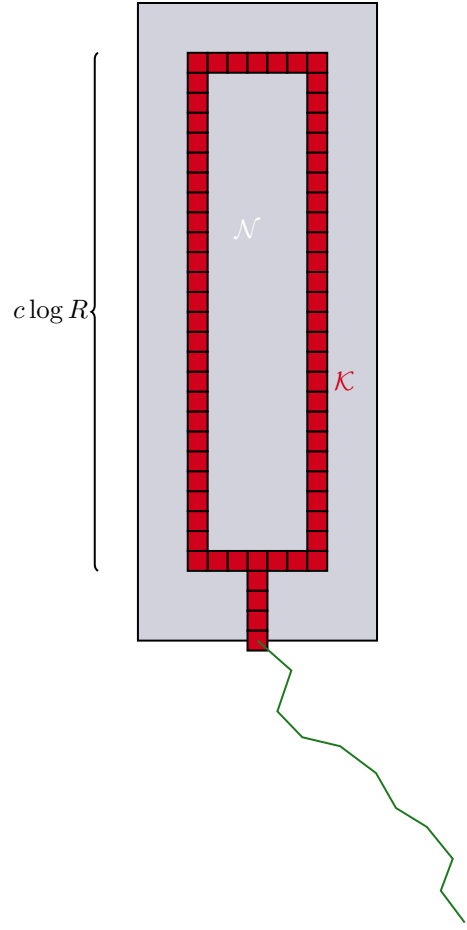

After building our large loop, we will want to connect it to the giant
using a region \emph{disjoint} from our loop construction,
so that we can use independence.
For a subset $K \subseteq \R^d$ and a point $x \in \R^d$, let us say that
$K$ is \emph{downward-connected to $\infty$ from $x$} if some point $y \in K$ lies in an edge path in $G^\nu \cap (\R \times (-\infty,\pi_2(x)] \times \R^{d-2})$ which escapes every bounded set. Here, $\pi_2:\R^d \to \R$ is the standard second coordinate projection.

\begin{lemma} \label[lemma]{downward connection infinite}
    If $\lam_- > \lam_c$, there exists $c_0=c_0(\lam_-,d)>0$
    such that for all $x \in \R^d$,
    \[
    \Prob(x+[-a/2,a/2]^d \mbox{ is downward-connected to } \infty
    \mbox{ from } x)
    \ge c_0.
    \]
\end{lemma}
\begin{proof}
    Note that, by monotonicity, it suffices to prove this statement
    for the homogeneous process $G^{\lam_-} = G^{\lam_-}_{\R^d}(1)$,
    and for that it suffices to show that
    for all $\lam > \lam_c$ we have
    \begin{equation} \label{eq:half space supercritical}
        \Prob( G^{\lam} \cap \R \times (-\infty,0] \times \R^{d-2} \mbox{ has an infinite component}) = 1.
    \end{equation}
    This is because, in case this holds,
    there is some $x_0 \in \R^d$ such that
    $x_0 + [-a/2,a/2]^d$ is connected to infinity in $\R \times (-\infty,0] \times \R^{d-2}$ with positive probability.
    Moreover, 
    \[\Prob([-a/2,a/2]^d \mbox{ is connected to }
     x_0 + [-a/2,a/2]^d
     \mbox{ by a path in }
     {G^\lam \cap \R \times (-\infty,0{]} \times \R^{d-2}}
     ) > 0,
    \]
    and therefore by the FKG inequality
    (see e.g. Theorem 2.2 of \cite{MR1996})
    we see that 
    \[ \Prob([-a/2,a/2]^d \mbox{ is downward-connected to } \infty \mbox{ from } 0) > 0;
    \]
    symmetry then gives the desired result.
    
    One can obtain \eqref{eq:half space supercritical} by mimicking,
    for instance, the proof of Theorem 7.2 in
    Grimmett \cite{Grimmett}.
    Here is a brief sketch of a less general but less complicated argument for \eqref{eq:half space supercritical}.
    As in Antal and Pisztora \cite{AP96}, for a large parameter $R$
    and a point $x \in \Z^d$, consider the event
    \[
        U_{x,R} := \left\{
        \begin{array}{c}
        \mbox{there is at most one component of } G^\lam \cap Rx+[-5R,5R]^d
        \mbox{ of} \\ 
        \mbox{diameter 
        at least } (\log R)^2, 
        \mbox{ and there exists an
        path in } G^\lam { from } \\
        (x+[-4R,4R]^d)^c \mbox{ to } x+[-R,R]^d 
        \end{array} \right\}.
    \]
    Call $x \in \Z^d$ \emph{white} if $U_{x,R}$
    holds. $\Prob(U_{x,R})$ can be made arbitrarily
    close to $1$ by choosing $R$ sufficiently large,
    and the family $\{ U_{x,R}\}_{x \in\Z^d}$
    only has finite range dependence (independent of $R$);
    therefore, by domination by product measures
    (\cite{LSS1997}),
    by choosing $R$ sufficiently large we can ensure
    that the law of the white sites stochastically
    dominates an independent site percolation on 
    $\Z^d$ with $p_0<1$ arbitrarily close to 1.
    Note next that an infinite cluster of white sites
    in $\Z \times ((-\infty,0] \cap \Z) \times \Z^{d-2}$
    implies the existence of an infinite
    cluster in $\R \times [-\infty,R] \times \R^{d-2}$;
    such an infinite cluster exists by a classic Peierls
    argument for sufficiently large $R$.
\end{proof}

Now we want to make a \emph{local} version of the event that the bottom of our constructed loop is connected to the giant, since we will ultimately want to have many independent translates of our event. Let $z_0 = (0,\lfloor -3/a \rfloor a - a,0,...,0)$ be the point in $a\Z^d$ immediately adjacent
to the lowest point of $I$.
To make connection to the giant a local event, we define the half-ball $\mathcal{H}$:
\[
    \mathcal{H} := B(z_0, Q \log R) \cap \R \times (-\infty,\pi_2(z_0)] \times \R^{d-2}.
\]

\begin{lemma} \label[lemma]{downward connection lower bound}
    If $\lam_- > \lam_c$, then for any $Q$ we have
    \[
        \Prob\left( 
        \begin{array}{c}
        z_0 + [-a/2,a/2] \mbox{ is connected to } \\
        B(z_0,Q\log R)^c 
        \mbox{ by a path in } \mathcal{H}
        \end{array}
        \right) \ge c_0
    \]
    where $c_0>0$ is as above.
\end{lemma}
\begin{proof}
    This follows immediately from \Cref{downward connection infinite}.
\end{proof}

We now define the key event that will imply the existence of large irreducible loops in the giant component.
Recall that $X^\nu$ is the Poisson point process on $\R^d$
with intensity measure $\nu$
which is the vertex set of $G^\nu$.
Define
\begin{equation} \label{eq:big loop event}
    E' := \bigcap_{K \in \mathcal{K}} \{ |X^\nu \cap K| \ge 1 \}
    \cap \{ |X^\nu \cap \mathcal{N}| = 0\}
\end{equation}

\begin{equation} \label{eq:big loop event in giant}
    E := E' 
    \cap \left\{ 
    \begin{array}{c}
    z_0 + [-a/2,a/2] \mbox{ is} \\ \mbox{connected
    to }
    B(z_0, Q\log R)^c \\
    \mbox{by a path in } G^\nu \cap \mathcal{H}
    \end{array}
    \right\}.
\end{equation}

Note that the entire construction above does depend on the choice of the constant $c>0$, and $E$
depends on a choice of $Q \ge 1$, even though we have suppressed this in our notation.
\begin{lemma} \label[lemma]{E and E' lower bound}
    There exists $0<\eta=\eta(\lam_-,\lam_+,d,Q)<1$,
    \emph{not} depending on $c>0$ or $R$, such that
    \[
        \Prob(E') \ge \eta^{c \log R}
    \]
    and if $\lam_- > \lam_c$
    we have
    \[
        \Prob(E) \ge c_0\eta^{c \log R},
    \]
    where $c_0 >0$
    is as in \cref{downward connection infinite}.
\end{lemma}
\begin{proof}
    First, note that by  construction, the three conditions in the 
    event $E$ all depend on disjoint regions of $\R^d$.
    For each $K \in \mathcal{K}$ we have
    \[
        \Prob( |X^\nu \cap K| \ge 1)
        = 1 - e^{-\nu(K)} \ge 1 - e^{-\lam_- a^d}.
    \]
    Since there are $O_a(c \log R)$ cubes in $\mathcal{K}$
    (as $R \to \infty$),
    we then have
    \[
        \Prob( \bigcap_{K \in \mathcal{K}} |X^\nu \cap K| \ge 1)
        \ge (1 - e^{-\lam_- a^d})^{O_a(c \log R)}.
    \]
    Next, we also have
    that the volume of $\mathcal{N}$ is $O_d(c \log R)$,
    and therefore
    \[
        \Prob( |X^\nu \cap \mathcal{N}| = 0)
        = e^{-\nu(\mathcal{N})} \ge e^{-O_d(\lam_+ c \log R)}.
    \]
    Combining the above with \Cref{downward connection lower bound} then gives
    \[
        \Prob(E)
        \ge 
        c_0
        [e^{-\lam_+}(1 - e^{-\lam_- a^d})]^{O_{a,d}(c \log R)},
    \]
    and so an appropriate choice of $0<\eta<1$
    gives our desired result.
\end{proof}

For $x \in \R^d$, let $E'_x,E_x$
be the shift of $E',E$ by $x$, that is
\begin{equation*}
    E_x' := 
    \bigcap_{K \in \mathcal{K}} \{ |X^\nu \cap (K+x)| \ge 1 \}
    \cap \{ |X^\nu \cap (\mathcal{N} + x)| = 0\},
\end{equation*}
\begin{equation*}
    E_x := 
    E_x'
    \cap 
    \left\{ 
    \begin{array}{c}
    z_0 + x + [-a/2,a/2] \mbox{ is connected to } \\
    B(z_0 + x, Q\log R)^c \mbox{ by a path in } G^\nu \cap (\mathcal{H} + x)
    \end{array}
    \right\}.
\end{equation*}
Note that the lower bound in \cref{E and E' lower bound} applies to all $E_x, E_x'$,
with constants not depending on $x$.

\begin{lemma} \label[lemma]{very likely to get at least one loop}
    There exist $c, c_1 > 0$
    depending on $\lam_-,\lam_+,d,Q$
    such that for any $\epsilon >0$
    the following holds. 
    There exists $R_0=R_0(\epsilon,\lam_-,\lam_+,d,Q)$
    such that for all $R \ge R_0$ we have
    \[
        \Prob( E'_x \mbox{ holds for some } x \in [-\epsilon R+Q\log R, \epsilon R - Q \log R]^d)
        \ge 
        1-e^{-c_1 \epsilon^d R}
    \]
    and if $\lam_- > \lam_c(d)$,
    \[
        \Prob( E_x \mbox{ holds for some } x \in [-\epsilon R+Q\log R, \epsilon R - Q \log R]^d)
        \ge 
        1-e^{-c_1 \epsilon^d R}.
    \]    
\end{lemma}
\begin{proof}
    First, note that, as long as $c>0$
    is sufficiently small compared to $Q$
    and $Q \log R$ is sufficiently large,
    each $E'_x, E_x$ only depends on 
    $X^\nu \cap (x+[-Q \log R,Q \log R]^d)$.
    In particular, there is some set $S \subseteq [-\epsilon R+Q\log R, \epsilon R - Q \log R]^d$
    with $|S| = \Omega((\epsilon R/Q\log R)^d)$
    such that the family of events $\{E_x\}_{x \in S}$
    is jointly independent.

    Thus, by \cref{E and E' lower bound}, we have
    \begin{align*}
        \Prob\left(
        \begin{array}{c}
        E_x \mbox{ fails for every } \\
        x \in [-\epsilon R+Q\log R, \epsilon R - Q \log R]^d
        \end{array}
        \right)
        &\le
        \Prob(E_x \mbox{ fails for every } x \in S) \\
        &\le
        (1 - c_0 \eta^{c \log R})^{\Omega((\epsilon R/\log R)^d)} \\
        &\le 
        (e^{-c_0 \eta^{c \log R}} )^{C_0 \epsilon^d R^d/(\log R)^d} \\
        &\le 
        \exp( -C_1 \epsilon^d R^{c \log \eta + d} / (\log R)^d).
    \end{align*}
    None of the other constants depend on $c>0$,
    so as long as we choose $c>0$ sufficiently small
    that
    \[
        d + c\log \eta > 1,
    \]
    we get the desired result for $E_x$. The argument for $E'_x$
    is the same.
\end{proof}

Now, we show that $E_x'$ entails the existence 
of a large irreducible loop in $G^n_M(r)$.
More properly, we show that $E_x'$
entails the existence of a cycle
which is not equivalent modulo
small cycles to a cycle
supported near the boundary of $x + [-Q \log R, Q \log R]^d$;
this latter statement
will better transfer to a manifold.

\begin{lemma} \label[lemma]{topology}
    There exists $R_0=R_0(Q,c)< \infty$
    such that for all $R \ge R_0$ the following holds.
    Assume $0 < c < Q/2$ and suppose $E'_x$ holds.
    Then $x + [-Q \log R, Q \log R]^d$
    supports a cycle $z \in Z_1(G^\nu)$
    such that the following hold.
    \begin{enumerate}
        \item If $w \in Z_1(G^\nu)$
        is such that $w-z$ is a sum
        of cycles of diameter
        at most $(c/100) \log R$
        with support lying in $x + [-Q \log R, Q \log R]^d$,
        $|w|$ must intersect
        $x + [-(Q/2)\log R, (Q/2)\log R]^d$.

        \item If, in addition, $E_x$ holds, then $z$ lies in a connected component
        of $G^\nu \cap [-Q \log R, Q \log R]^d$ of Euclidean diameter at least $Q \log R$.
    \end{enumerate}
\end{lemma}
\begin{proof}
    First, note
    that on $E'_x$, since $(x + \mathcal{N}) \cap X^\nu = \emptyset$, 
    we have
    $(x + \mathcal{N}') \cap G^\nu = \emptyset$, where we define
    \[
        \mathcal{N}'
        := \{ y \in \R^d : \forall w \in \mathcal{N}^c, \|y - w\| > 1 \}.
    \]
    Therefore, we have a well-defined
    map
    \[
        Z_1(G^\nu) \to H_1(\R^d \setminus \mathcal{N}').
    \]
    Note also that any cycle in $\R^d \setminus \mathcal{N}'$
    which has diameter at most $(c/100)\log R$
    is necessarily the boundary of some element
    of $C_2(\R^d \setminus \mathcal{N}')$.
    Therefore, the above map factors through a
    well-defined map
    \begin{equation} \label{eq:well-defined-map}
        Z_1(G^\nu) / Z_1^{(c/100) \log R}(G^\nu)
        \to H_1(\R^d \setminus \mathcal{N}'),
    \end{equation}
    where, as always,
    $Z_1^{(c/100)\log R}(G^\nu)$
    is the submodule generated by the cycles
    of diameter at most $(c/100 \log R)$.

    Now, note that on $E'_x$, since each
    $K \in \mathcal{K}$ is occupied,
    there exists $z \in Z_1(G^\nu)$
    which has the same image
    as $\partial \mathcal{R}$
    under the map to $H_1(\R^d \setminus \mathcal{N}')$.
    
    To show that (1) holds, 
    let $w \in Z_1(G^\nu)$
    be such that
    $$w-z \in Z_1^{(c/100) \log R)}(G^\nu \cap x+[-Q\log R,Q\log R]^d).$$
    Suppose \emph{per contra}
    that $|w| \subseteq \R^d \setminus (x+[-(Q/2) \log R, (Q/2) \log R]^d)$.
    Then $w$ and $z$ map to different elements of
    $H_1(\R^d \setminus \mathcal{N}')$.
    If $d \ge 3$, one can see this
    because $\partial \mathcal{R}$
    represents a nonzero element
    of $H_1(\R^d \setminus \mathcal{N}')$, while
    $H_1(\R^d \setminus (x+[-(Q/2) \log R, (Q/2) \log R]^d)) = 0$
    (the space deformation retracts onto $S^{d-1}$,
    which is simply connected).
    If $d=2$, $\R^d \setminus \mathcal{N'}$
    is homotopy equivalent to a bouquet of circles,
    and one can check that one of the circles
    is represented by $\partial \mathcal{R}$,
    and the image of $H_1(\R^d \setminus (x+[-(Q/2) \log R, (Q/2) \log R]^d))$ in $H_1(\R^d \setminus \mathcal{N}')$
    lies in the submodule generated by the other circle.
    
    However, the fact that \eqref{eq:well-defined-map} is a well-defined map
    means that $w$ and $z$ must have
    the same image in $H_1(\R^d\setminus \mathcal{N}')$, so this is a contradiction.
    Thus we have shown that (1) holds.

    For (2), simply note that on the event $E_x$, $z$ is connected to a path 
    which escapes $B(x+z_0,Q \log R)$,
    and so lies in a component of large diameter.
\end{proof}

Now we can show the desired statement for
$G^n_M(r)$.
\begin{prop} \label[prop]{need at least rlogr Poisson version}
    Let $0< \lowerlambda \le \upperlambda < \infty$.
    Then there exist $q>0, c_2 > 0$
    depending on $\lowerlambda,\upperlambda,d$
    such that for all sufficiently small $r > 0$,
    we have
    \begin{equation} \label{eq:need at least rlogr Poisson subcritical}
        \Prob\left( 
        \begin{array}{c} 
        \exists z \in \ker(Z_1(G^n_M(r)) \to H_1(M)) \\
        \mbox{ such that } z \notin Z_1^{q r|\log r|}(G^n_M(r))
        \end{array}
        \right)
        \ge 1 - e^{-c_2 r^{-1}}.
    \end{equation}
    Moreover, if $\lowerlambda > \lam_c$, for any $\kappa > 0$,
    there exists $q>0$
    depending on $\lowerlambda,\upperlambda,d,\kappa$ such that for all sufficiently small $r>0$,
    if $\lowerlambda \le nr^d \le \upperlambda$
    we have
    \begin{equation} \label{eq:need at least rlogr Poisson supercritical}
        \Prob\left( 
        \begin{array}{c} 
        \exists z \in \ker(Z_1(\mathring{G}^n_M(r)) \to H_1(M)) \\
        \mbox{ such that } z \notin Z_1^{q r|\log r|}(G^n_M(r))
        \end{array}
        \right)
        \ge 1 - r^\kappa.
    \end{equation}
\end{prop}
\begin{proof}
    Let $x_0 \in M$.
    Using \Cref{lem:almostflat},
    choose $\epsilon_0 > 0$
    sufficiently small that
    the geodesic normal coordinates map
    \[
        u: U := B_M(x_0, \epsilon_0) \to B(0,\epsilon_0) =: S \subseteq \R^d
    \]
    has the following properties.
    For some $0<\eta<1/10$,
    \[
        |\log (d_M(x,y) / \|u(x) - u(y)\|| \le \eta
    \]
    for all $x,y \in U$.
    If $\lowerlambda > \lam_c$,
    we also require that $\eta$
    be sufficiently small that
    $e^{-(d+1)\eta} \underline{\lam} > \lam_c$.
    Moreover, $\epsilon_0 > 0$
    is sufficiently small that the pushforward measure $u_* dM$ satisfies
    \[
        e^{-\eta} \le \frac{d u_*(dM)}{d\mathcal{L}} \le e^\eta
    \]
    on $B(0,\epsilon_0)$,
    where $\frac{d u_*(dM)}{d\mathcal{L}}$
    is the Radon-Nikodym derivative of 
    $u_*(dM)$ with respect to Lebesgue measure.

    Now let $G_S = u(G^n_M(r) \cap U)$,
    where $U = B(x_0, \epsilon_0)$,
    and note that,
    by \cref{lem:identify-G-S},
    $(e^{-\eta} r)^{-1} u(G_S)$
    is equal in distribution 
    to $G^\nu \cap B(0, (e^{-\eta} r)^{-1} \epsilon_0)$
    for $\nu, \sigma$
    satisfying our conditions
    \eqref{eq:metric condition}
    and \eqref{eq:volume condition}
    if we take
    $\lam_- = e^{-\eta (d+1)} \underline{\lam}$,
    $\lam_+ = e^{\eta (d+1)} \overline{\lam}$.
    Let us couple $G_S$ and $G^\nu$
    so that this is an almost sure equality.
    Then, take $\epsilon > 0$
    sufficiently small that
    $[-\epsilon/r, \epsilon/r]^d \subseteq B(0, (e^{-\eta} r)^{-1} \epsilon_0)$
    (and note that this choice depends on $\epsilon_0$ and $\eta$ but not $r>0$).

    If $\lowerlambda > \lam_c$, let $Q < \infty$ be sufficiently large that
    \[
    \Prob
    \left(
    \begin{array}{c}
    \exists \mbox{ more than one component of } G^n_M(r) \\ \mbox{of } d_M\mbox{-diameter at least } (Q/2) r |\log r|
    \end{array}
    \right) = o(r^{\kappa}),
    \]
    which is possible
    by \Cref{A_3 likely}.
    If $\lowerlambda \le \lam_c$,
    take $Q=2c$.
    Letting $c>0, c_1>0$ be as in \Cref{very likely to get at least one loop}, 
    letting $R = 1/e^{-\eta}r$,
    and letting $E'_x, E_x$ be defined as in \eqref{eq:big loop event}, \eqref{eq:big loop event in giant},
    we have that
    \[
        \Prob( E'_x \mbox{ holds for some } x \in [-\epsilon R+Q\log R, \epsilon R - Q \log R]^d)
        \ge 1 - e^{-c_1 \epsilon^d R},
    \]
    and the same bound
    holds for $E'_x$
    replaced by $E_x$
    if $\lowerlambda > \lam_c$.
    So assume that $E'_x$ holds for some
    $x \in [-\epsilon R+Q\log R, \epsilon R - Q \log R]^d$.
    By \Cref{topology}, there then exists
    $z \in Z_1(G^\nu \cap (x + [-Q \log R, Q \log R]))$
    such that  
    if $w \in Z_1(G^\nu)$
    and $w - z \in Z_1^{(c/100) \log R}(G^\nu \cap (x + [-Q \log R, Q \log R]^d))$,
    then $w$ must intersect $[-(Q/2) \log R, (Q/2) \log R]^d$.

    We claim that
    if $r>0$ is sufficiently small,
    setting
    \[
    z' := [(e^{-\eta}r)u]^{-1}(z),
    \]
    we have $z' \in \ker ( Z_1(G^n_M) \to H_1(M))$
    and $z' \notin Z_1^{(c/10000) r |\log r|}(G^n_M(r))$.

    The first claim is immediate because
    $z'$ has $d_M$-diameter $O(r |\log r|) = o(1)$.
    
    To see the second claim, suppose the 
    contrary,
    and write $z' = \sum_{i=1}^N z'_i$,
    where each $z'_i \in Z_1(G^n_M(r))$
    has $d_M$-diameter
    at most $c/10000$.
    Then set
    \[
        \mathcal{Z} = 
        \{ z'_i : |z'_i| \subseteq [(e^{-\eta}r)^{-1}u]^{-1}([-Q \log R,Q \log R]^d) \}.
    \]
    Define
    \[
        w' := z' - \sum_{z'_i \in \mathcal{Z}} z'_i
        = \sum_{z'_i \notin \mathcal{Z}} z'_i.
    \]
    Note that the description on the left shows that
    $|w'| \subseteq [(e^{-\eta}r)^{-1}u]^{-1}([-Q \log R,Q \log R]^d)$,
    while the description on the right shows that all of $|w'|$
    lies within $d_M$-distance
    $(c/10000)r|\log r|$
    of the complement 
    of $[(e^{-\eta}r)^{-1}u]^{-1}([-Q \log R,Q \log R]^d)$.
    Thus we can define 
    \[
        w := (e^{-\eta}r)^{-1}u(w')
        = z - \sum_{z_i' \in \mathcal{Z}} (e^{-\eta}r)^{-1}u(z'_i)
    \]
    and we see that
    $|w| \subseteq [-Q \log R, Q \log R]^d$,
    and by the metric distortion
    bound for $u$,
    $|w|$ lies within Euclidean distance at most $(c/100) \log R$
    of $[-Q \log R, Q \log R]^d$;
    in particular,
    $|w|$ does not intersect
    $[-(Q/2) \log R, (Q/2) \log R]^d$.
    But again the metric distortion bound
    shows that $w-z$
    is generated by cycles
    of Euclidean diameter
    at most $(c/100)\log R$,
    so we have contradicted
    the condition
    on $z$ guaranteed us
    by \Cref{topology}.
    Thus $z' \notin Z_1^{(c/10000)r|\log r|}$,
    as claimed.
    Taking $q=c/10000$
    and $c_2 = c_1 \epsilon^d$,
    we have established
    \eqref{eq:need at least rlogr Poisson subcritical}
    (where the $z$ mentioned
    in the statement is given by what we have called $z'$
    here in the proof).
    
    For \eqref{eq:need at least rlogr Poisson supercritical},
    assume $\lowerlambda > \lam_c$, and  
    follow the same 
    argument as above,
    but with $E'_x$ replaced by $E_x$.
    With probability at least $1 - o(r^\kappa)$,
    we again obtain $z \in Z_1(G^\nu)$
    and a corresponding 
    $z'
    \in Z_1(G^n_M(r))$
    which lies in the kernel
    of the natural
    map $Z_1(G^n_M(r)) \to H_1(M)$
    but which does not 
    lie in $Z^{qr |\log r|}_1(G^n_M(r))$.
    Moreover,
    by \Cref{topology},
    this time, the constructed $z$
    lies in a component
    of $G^\nu \cap [-Q \log R, Q \log R]^d$
    of diameter at least $Q \log R$.
    But then our distortion bound on $u$
    implies that $z'$
    lies in a component of $G^n_M(r)$
    of diameter at least
    $(Q/2)r|\log r|$.
    By our choice of $Q$,
    off of an event of probability $o(r^\kappa)$, this must be
    the giant component.
    So we are done.
\end{proof}

Lastly, we prove \cref{need at least rlogr} as stated, using the standard ``de-Poissonization''
we performed for the other main theorems:
\begin{proof}[Proof of \Cref{need at least rlogr}]
    Fix $\kappa > d/2$.
    Given $0<\lowerlambda \le \upperlambda < \infty$,
    choose $q>0$
    as in \Cref{need at least rlogr Poisson version}
    so that
    \[
    \Prob( \ker(Z_1(G^n_M(r)) \to H_1(M)) \subseteq Z_1^{q r|\log r|}(G^n_M(r)))
    \le \exp(-c_2 r^{-1}) = o(r^\kappa)
    \]
    and in case $\lowerlambda > \lam_c$,
    \[
    \Prob( \ker(Z_1(\mathring{G}^n_M(r)) \to H_1(M)) \subseteq Z_1^{q r|\log r|}(G^n_M(r)))
    = o(r^\kappa).
    \]
    Then, again, since
    the law of $G_M(n;r)$
    is equal to the law of $G_M^n(r) | N =n$,
    where $N \sim \mathrm{Poi}(n)$,
    we have
    \begin{align*}
        &\Prob( \ker(Z_1(G_M(n;r)) \to H_1(M)) \subseteq Z_1^{q r|\log r|}(G_M(n;r))) \\
        &=
        \Prob( \ker(Z_1(G^n_M(r)) \to H_1(M)) \subseteq Z_1^{q r|\log r|}(G^n_M(r)) | N = n) \\
        &\le O(n^{-1/2} r^\kappa)
        = O(n^{(1/2)-(\kappa/d)}),
    \end{align*}
    which tends 
    to $0$ as as $n \to \infty$
    by our choice of $\kappa$.
    The same argument 
    shows that
    $\ker(Z_1(\mathring{G}_M(n;r)) \to H_1(M)) \not\subseteq Z_1^{q r|\log r|}(G_M(n;r))$
    with high probability
    in case $\lowerlambda > \lam_c$,
    as desired.

    The above shows 
    that with high probability, the natural 
    maps
    \[
    Z_1(G_M(n;r))/Z_1^{qr|\log r|}(G_M(n;r)) \to H_1(M),
    \qquad 
    Z_1(\mathring{G}_M(n;r)/Z_1^{qr|\log r|}(\mathring{G}_M(n;r)) \to H_1(M)
    \]
    have nontrivial kernel
    and hence
    are not isomorphisms.
    If we want the stronger
    statement
    that in the supercritical regime,
    with high probability neither $Z_1(G_M(n;r))/Z_1^{qr|\log r|}(G_M(n;r))$
    nor $Z_1(\mathring{G}_M(n;r)/Z_1^{qr|\log r|}(\mathring{G}_M(n;r))$
    is isomorphic to 
    $H_1(M)$ (via \emph{any} map),
    we can deduce this quickly from the
    following two facts. 
    First, the natural maps
    are with high probability
    surjective (by \cref{surjectivity}).
    Second, we have the
    following well-known purely algebraic
    fact: if $f:A \to A$
    is a surjective homomorphism
    of finitely generated
    modules over a commutative ring
    $\mathcal{S}$,
    then $f$ is an isomorphism.\footnote{This
    fact is due to
    Vasconcelos \cite{Vasconcelos}
    and can quickly be shown using Nakayama's Lemma.
    }
    Combining these shows that,
    with high probability,
    since the above maps
    are surjective but not isomorphisms,
    the source and the target 
    cannot be isomorphic.\footnote{For brevity,
    write $A$
    for either $Z_1(G_M(n;r))/Z_1^{qr|\log r|}(G_M(n;r))$
    or $Z_1(\mathring{G}_M(n;r)/Z_1^{qr|\log r|}(\mathring{G}_M(n;r))$,
    and note that $A$ is by construction
    a finitely generated module
    over the commutative ring $\mathcal{S}$.
    We have shown that with high probability, the natural map 
    $f:A \to H_1(M)$
    is surjective but not an isomorphism.
    Now suppose \emph{per contra}
    that 
    there exists some
    isomorphism $\iota: H_1(M) \to A$.
    Then the 
    composition $\iota \circ f:A \to A$
    is a surjective homomorphism
    which is not an isomorphism,
    a contradiction.
    }
\end{proof}

\section{Homology of the \v{C}ech and Vietoris-Rips complexes} \label{sec:complexes}

Rather than using a \emph{graph} as a combinatorial model of a manifold,
it is arguably more natural to use a \emph{simplicial complex},
since this 
can encode higher-dimensional information about the manifold.
Recall that a \emph{simplicial complex}
is a pair $(X,\mathcal{F})$,
where $X$ is a set, and $\mathcal{F} \subseteq 2^X$
is a collection of subsets of $X$ such that if $S \in \mathcal{F}$ and $S' \subseteq S$, then $S' \in \mathcal{F}$.
Geometrically, we think of $S \in \mathcal{F}$
with $|S|=k$ as a \emph{$(k-1)$-simplex}, and $S' \subseteq S$ are faces of $S$. (In particular $x \in S$
is a vertex in $S$).
Simplicial homology can be defined for general simplicial complexes (see e.g. Chapter 2.1 of Hatcher \cite{Hatcher}),
and so one may hope to detect the homology of a manifold $M$ better by
constructing a combinatorial model of $M$
which is a \emph{simplicial} complex
rather than a graph.
We describe and consider the two most popular constructions below.

Let $X \subseteq M$
be a (typically discrete) subset of a metric manifold, $r>0$.
The \emph{\v{C}ech complex}
$\Cech_{r}(X)$  has underlying set $X$,
and a subset $\{s_0,...,s_k\} \subseteq X$
lies in $\mathcal{F}$
if and only if 
$\bigcap_{i=0}^k B(s_i,r)$
is nonempty.
If $r>0$ is sufficiently small
that any nonempty intersection of the balls $B(x,r)$ is contractible, then the Nerve Theorem \cite{Borsuk1948} says
that
the simplicial homology of $\Cech_r(X)$
is isomorphic to the singular homology of
$\bigcup_{x \in X} B(x,r) \subseteq M$.
This straightforward geometric interpretation makes $\Cech_r(X)$
a popular candidate for a combinatorial model of a manifold $M$.

A less computationally expensive complex is the \emph{Vietoris-Rips complex} $VR_r(X)$,
whose underlying set is $X$,
and whose faces $S \in \mathcal{F}$ are precisely
the finite subsets of $X$ of diameter less than $r$.

Note that $\Cech_{r/2}(X)$ and $VC_r(X)$ have the same $1$-skeleton, i.e. the same set of $1$-simplices. Moreover, if $X=X^n$ is the point set associated to a Poisson point process on $M$ with intensity measure $ndM$, then this $1$-skeleton is precisely the random geometric graph $G^n_M(r)$.

One might hope that, since $VC_r(X^n)$ and $\Cech_{r/2}(X^n)$
incorporate $2$-dimensional information about the space $M$, their first homology groups more closely resemble $H_1(M)$ than those of $G^n_M(r)$.
However, in the thermodynamic regime, this can only be true in a very weak sense,
as we now show.
The intuitive idea is that, in order to make $Z_1(G_M^n(r))$ match $H_1(M)$, we must quotient out by cycles
of diameter $\Theta(r|\log r|)$;
but the $2$-simplices we
add to $G_M^n(r))$ to construct the \v{C}ech or Vietoris-Rips complex all have diameter
only $O(r)$.

For the following, let $\mathcal{C} \in \{\Cech_{r/2}(X^n), VR_r(X^n)\}$
be either the \v{C}ech or Vietoris-Rips complex associated to a Poisson point process on $M$
with intensity measure $ndM$.\footnote{If one likes, one can also just consider the giant component $\mathring{\mathcal{C}}$ of $\mathcal{C}$, and appropriate modifications of the same statements hold by the same arguments.
Also, standard de-Poissonization arguments show the same result when the point set $X$ consists of $n$ independently sampled points, rather that a Poisson point process.}
First, note that since $G^n_M(r)$ is the $1$-skeleton of $\mathcal{C}$,
we have
\[
    Z_1(\mathcal{C}) = Z_1(G_M^n(r)).
\]
Moreover, 
the boundary of any individual $2$-simplex in $\mathcal{C}$
has $d_M$-diameter at most $r$.
That is, for any $s \ge r$ we have
$B_1(\mathcal{C}) \subseteq Z_1^s(\mathcal{C}) = Z_1^s(G_M^n(r))$,
and therefore there exists a 
a well-defined surjection
\[
    H_1(\mathcal{C}) 
    \to
    Z_1(G^n_M(r))/Z_1^s(G_M^n(r)).
\]
If we define
\[
    H_1^s(\mathcal{C}) 
    :=
    \im(Z_1^s(\mathcal{C}) \to H_1(\mathcal{C}))
    =
    \im( Z_1^s(G^n_M(r)) \to H_1(\mathcal{C})),
\]
then for $s \ge r$ the map above induces
a well-defined \emph{isomorphism}
\begin{equation} \label{eq:complex iso}
    H_1(\mathcal{C})/H_1^s(\mathcal{C})
    \cong
    Z_1(G_M^n(r))/Z_1^s(G_M^n(r)).
\end{equation}

One can quickly deduce the following facts from the above and our main theorems:

\begin{cor}
Fix $\lam \in (0, \infty)$
and let $r = (\lam/n)^{1/d}$.
\begin{enumerate}
    \item If $\lam > \lam_c$,
    then with high probability as $n \to \infty$, the natural map $H_1(\mathcal{C}) \to H_1(M)$ is surjective.
    If $\lam < \lam_c$,
    with high probability
    it is the zero map.
    \item If $\lam > \lam_c$,
    there exists $1 \le Q' < \infty$
    such that $H_1(M) \cong H_1(\mathcal{C})/H_1^{Q' r|\log r|}(\mathcal{C})$
    with high probability as $n \to \infty$.
    \item For fixed $\lam < \infty$,
    there exists $q(\lam) > 0$
    such that, with high probability as $n \to \infty$,
    the natural map
    $H_1(\mathcal{C})/H_1^{qr|\log r|}(\mathcal{C}) \to H_1(M)$
    has nontrivial kernel (in particular, it is not an isomorphism).
    In this case, the natural map
    $H_1(\mathcal{C}) \to H_1(M)$ 
    is also not an isomorphism.\footnote{This last result also follows from Bobrowski \cite{Bobrowski22}.}
\end{enumerate}
\end{cor}
\begin{proof}
    (1) follows immediately from \cref{surjectivity}
    combined with the fact that
    the natural map
    \[
        Z_1(\mathcal{C}) =
        Z_1(G_M^n(r)) \to
        H_1(M)
    \]
    is equal to the composition of 
    natural maps
    \[
        Z_1(\mathcal{C}) \surj
        H_1(\mathcal{C}) \to H_1(M),
    \]
    with the first map always surjective by construction.

    (2) follows immediately
    from \cref{isomorphism}
    and \eqref{eq:complex iso}
    with $s=Q'r|\log r|$,
    once $r>0$ is sufficiently small
    that $Q'r|\log r| \ge r$.
    
    (3) follows immediately 
    from \cref{need at least rlogr}
    and \eqref{eq:complex iso}
    with $s = qr|\log r|$,
    once $r>0$ is sufficiently small
    that $qr|\log r| \ge r$.
    The very last statement
    follows from the fact 
    that the natural map $H_1(\mathcal{C}) \to H_1(M)$
    is equal to the composition of
    natural maps
    \[
    H_1(\mathcal{C}) \surj H_1(\mathcal{C})/H_1^{q r|\log r|}(\mathcal{C})\to H_1(M),
    \]
    with the first map always surjective
    by construction.
\end{proof}

Thus, without any modification, $H_1(\mathcal{C})$ is not a much better model for $H_1(M)$ than $H_1(G^n_M(r))$
in the thermodynamic regime.
Moreover, the way the way that we 
rectify this, i.e. quotienting out 
by cycles of diameter $\Theta(r|\log r|)$,
is the same as for $G_M^n(r)$.

\section{Identification of small cycles}
\label{sec: objections}
One might object that, since
the definition of $Z_1^s(G_M(n;r))$ depends on the metric $d_M$ on $M$, \Cref{isomorphism} does not actually tell us much about how to recover 
the topology of an \emph{unknown} manifold $M$ from either a point cloud or an associated random geometric graph.
Here we address this objection.

First, suppose that one wishes to apply this theorem in the setting of topological data analysis. That is, suppose $M$
is \emph{embedded} in some high dimensional space $\R^N$,
we wish to infer
the topology of $M$
from $n$ points independently
sampled from $M \subseteq \R^N$,
and we have access to the
\emph{Euclidean} distances between those points, but not the Riemannian distance $d_M$.
Let us denote by $G(n;r)$ the geometric graph with vertex set given by these $n$ independently sampled points,
and with edges joining pairs of vertices with \emph{Euclidean} distance at most $r$.
Note that the Euclidean distances between
points may be \emph{smaller}
than the Riemannian distances.

We want an analogue of the injectivity radius, and we can take
\[
    \mathrm{inj}'(M)
    := \inf \{ \|x - y\| : x,y \in M, d_M(x,y) \ge \mathrm{inj}(M)/2 \} > 0.
\]
The key point is that
for any $x \in M$,
\[
    M \cap \{ y \in \R^d : \|x - y\| \le \mathrm{inj}' \}
\]
is contractible.

Once we have made this replacement,
the proof of the analogue of \cref{isomorphism}
is exactly analogous to our proof for $G_M(n;r)$.
Indeed, the statement of \cref{lem:almostflat}
still holds if we replace $d_M$
by the Euclidean metric restricted to $M$; that is, the geodesic normal coordinates map has arbitrarily small metric distortion on sufficiently small patches.
The only facts that we use about $d_M$ throughout the entire proof are this small metric distortion, and the fact that sufficiently small balls are contractible, and the latter has been dealt with through the introduction of $\mathrm{inj}'$.

Thus, if we take ${Z'}_1^s(G(n;r))$
to be the submodule of $Z_1(G(n;r))$
generated by cycles of \emph{Euclidean} diameter at most $s$,
we indeed have
\begin{cor}
    Fix $\lam_0 > \lam_c(d)$.
    Given $0 < \kappa < \infty$,
    there exist $0< Q'(\lam_0,\kappa,d) <\infty$
    and $r_0(\lam_0,\kappa,M)>0$
    such that if $0<r \le r_0$
    and $n r^d \ge \lam_0$,
    \[
        \Prob(
        H_1(M)
        \cong
        Z_1(G(n;r))/{Z'}^{Q'r|\log r|}_1(G(n;r)))
        \ge 1 - r^\kappa.
    \]
\end{cor}
Again, the point of the above corollary
is that it does not make reference to the Riemannian metric on $M$, only Euclidean distances.\footnote{Note that $r_0$ \emph{does}
depend on the embedding of $M$ into $\R^N$
through $\mathrm{inj}'(M)$.
However, it does not seem possible for these theorems to get around some sort of assumption akin to assuming that the connectivity radius is sufficiently small compared to the injectivity radius.}

Next, let us interpret \Cref{isomorphism}
as saying that one can recover $H_1(M)$
with high probability from the \emph{graph isomorphism type} of $G_M(n;r)$.
Again one might object that the definition of small cycles in terms of $d_M$
requires knowledge of more than just the graph isomorphism type of $G_M(n;r)$.
However, note that on the
event $A_{Q,\rho}(n,r)$,
if $|\log r| \ge Q'/Q$, then
all cycles in $G^n_M(r)$ which have $d_M$-diameter at most $Q' r |\log r|$
will have $d_G$-diameter at most 
$Q' \rho |\log r|$.
In particular, if ${Z''}_1^s(G^n_M(r))$
is the submodule of $Z_1(G^n_M(r))$
generated by cycles of $d_G$-diameter
at most $s$, then on $A_{Q,\rho}(n,r)$
we have
\[
    Z_1^{Q'r|\log r|}(G^n_M(r)) \subseteq {Z''}_1^{Q' \rho |\log r|}(G^n_M(r)).
\]
Moreover, since $d_M \le rd_G$, as long as 
$Q' \rho r|\log r| \le \mathrm{inj}(M)$
we have
\[
    {Z''}_1^{Q' \rho |\log r|}(G^n_M(r))
    \subseteq
    \ker( Z_1(G^n_M(r)) \to H_1(M)).
\]
Thus,
if 
$Z_1^{Q'r|\log r|}(G^n_M(r)) = \ker( Z_1(G^n_M(r)) \to H_1(M))$,
then both are equal to 
${Z''}_1^{Q' \rho |\log r|}(G^n_M(r))$;
since the latter is defined in terms of
\emph{graph distance} $d_G$,
it only depends on the \emph{graph isomorphism type} of $G_M^n(r))$.
This observation---together with a standard ``de-Poissonization argument'', such as used in \Cref{sec:main proofs}, to go from $G^n_M(r)$ to $G_M(n;r)$)---therefore gives us:
\begin{cor}
    Fix $\lam_0 > \lam_c(d)$.
    Given $0 < \kappa < \infty$,
    there exist $0< Q''(\lam_0,\kappa,d) <\infty$
    and $r_0(\lam_0,\kappa,M)>0$
    such that if $0<r \le r_0$
    and $n r^d \ge \lam_0$,
    \[
        P(
        H_1(M)
        \cong
        Z_1(G_M(n;r))/{Z''}^{Q''|\log r|}_1(G_M(n;r))
        )
        \ge 1 - r^\kappa.
    \]
\end{cor}
Thus, we can indeed with high probability recover $H_1(M)$ from the graph isomorphism type of $G_M(n;r)$
in the supercritical thermodynamic regime.

\section{Future directions}
\label{sec:future}
\subsection{Higher dimensional homology groups}
The most obvious related question to explore is whether it is possible in the thermodynamic regime to detect $H_k(M), k \ge 2$ from an appropriate random geometric simplicial complex $\mathcal{C}$; for the sake of our discussion, let us here take $\mathcal{C}$ to be the Vietoris-Rips complex $VC_{r}(X^n)$.

Work of Bobrowski and Skraba \cite{BS2020} suggests that, for $\lam$ close to $\lam_c$,
the natural maps $H_k(\mathcal{C}) \to H_k(M)$
are \emph{not} surjective for $k \ge 2$.
Instead, for each $d \ge 2$, there should be a sequence of thresholds
\[
    \lam_c(d)=\lam_{c,1}(d)
    < \lam_{c,2}(d)
    < ...
    < \lam_{c,d-1}(d) < \infty
\]
such that below $\lam_{c,k}(d)$,
the natural map $H_k(\mathcal{C}) \to H_k(M)$
is zero with high probability,
and above $\lam_{c,k}$,
it is surjective
with high probability.
The $\lam_{c,k}(d)$ are called the \emph{homological percolation thresholds}.
Note that the existence of such thresholds has not yet been established; in \cite{BS2020},
the existence of supercritical and subcritical phases is established on the flat torus, and a sharp threshold for $k=1$ coinciding with the percolation threshold (again in the case of the flat torus) is established.
Work of Duncan, Kahle, and Schweinhart \cite{DKS2025} does show the existence of sharp thresholds for all $1 \le k \le d-1$ for a \emph{discrete} model on several lattices in the $d$-torus.
However, the present article is to our knowledge the first establishing the $k=1$ homological percolation threshold for an \emph{arbitrary} Riemannian manifold.
Even establishing the existence
of a sharp threshold for $k \ge 2$ will require new tools, and even the appropriate choices of definitions are not obvious.

However, the main theorems of this paper suggest the following conjecture:
\begin{conjecture}
    For each $d \ge 2$ and each $1 \le k \le d - 1$, there exists $0 < \lam_{c,k}<\infty$
    such that the following holds.
    For $\lam \in (0,\infty)$,
    let $X^n$ be a Poisson point process with intensity measure $ndM = (\lam/r^d) dM$
    on a Riemannian manifold of $M$ of dimension $d$ and volume 1,
    and let $
    \mathcal{C} = VC_r(X^n)$
    be the associated Vietoris-Rips complex
    with radius $r$.
    Then, if $ \lam > \lam_{c,k}$, then
    \[
        \Prob(\mbox{the natural map }
        H_k(\mathcal{C}) \to H_k(M)
        \mbox{ is surjective})
        \to 1
    \]
    as $r \to 0$. Moreover, there exists
    some function $s_\lam:[0,\infty) \to [0,\infty)$
    with $s_\lam(r) = o_{r \to 0}(1), s_\lam(r) = \omega(r)$
    such that
    \[
        \Prob(\mbox{the natural map }
        Z_k(\mathcal{C})/
        Z_k^{s_{\lam}(r)}(\mathcal{C})
        \to H_k(M)
        \mbox{ is an isomorphism})
        \to 1
    \]
    as $r \to 0$,
    where $Z_k^{s_\lam(r)}(\mathcal{C})$
    is the submodule
    of $Z_k(\mathcal{C})$
    generated by $k$-cycles of diameter
    at most $s_\lam(r)$.
    \\
    On the other hand,
    if $\lam < \lam_{c,k}$, then
        \[
        \Prob(\mbox{the natural map }
        H_k(\mathcal{C})
        \to H_k(M)
        \mbox{ is the zero map})
        \to 1
    \]
    as $r \to 0$.
\end{conjecture}
Several remarks are in order:
\begin{enumerate}
\item It may also be plausible to conjecture that $s_\lam(r) = Q(\lam)s(r)$
for some $s:[0,\infty) \to [0,\infty)$
not depending on $\lam$,
and perhaps even that $s = \Theta(\log r)$,
as we have in the case $k=1$; but a different order for $s$ is also plausible. However,
an argument similar to
\Cref{need at least rlogr} shows that we can construct nontrivial elements of the kernel of $Z_k(\mathcal{C})/Z_k^{s(r)}(\mathcal{C}) \to H_k(M)$
if $s(r) = q r|\log r|$ with $q>0$
sufficiently small;
therefore we must have $s_\lam(r) = \Omega_\lam(r|\log r|)$.
\item
If the above conjecture is true, then, again, this
would give
a combinatorial
model of $M$
which recovers $H_k(M)$
but is sparser
than a Vietoris-Rips complex dense enough to have $k$\textsuperscript{th}
homology isomorphic to $H_k(M)$; the latter only emerges in the \emph{dense} regime, see Bobrowski \cite{Bobrowski22}.
\item $Z_k(\mathcal{C})$
only depends on the $k$-skeleton of $\mathcal{C}$, meaning the above conjecture would allow us to infer $H_k(M)$
from only the $k$-skeleton of a geometric random complex, just as here we found that $H_1(M)$ is recoverable from the $1$-skeleton $G_M(n;r)$.
\item We can actually compress the data further: since the Vietoris-Rips complex 
is simply the \emph{clique complex}
of $G_M(n;r)$,
if the above conjecture is true,
then
all $H_k(M), 1 \le k \le d-1$ can be inferred from the \emph{graph} $G_M(n;r)$
in the thermodynamic regime as long as $\lam > \lam_{c,d-1}(d)$.
\item 
If the above conjecture is true, one expects it to be true for many different models of continuum percolation
and hence many different models of geometric random complexes, although possibly with different values for the thresholds.
The methods here suggest that once one has constructed an appropriate analog of \cref{approximate 1-chains},
the rest should follow.
\item 
On the other hand, actually proving a higher dimensional analogue of \Cref{approximate 1-chains} will require real work; for one thing, it is not clear how to define, e.g., higher dimensional analogues of the uniqueness of the giant component, let alone prove them. 
\end{enumerate}

\subsection{Exact scale of generators of the kernel}
In another direction, note that our theorems establish that the largest $1$-cycles we need to quotient out in order to recover $H_1(M)$ are precisely of order $\Theta(r \log r)$.
Precisely, we can define the random variable
\[
    Y(r) := \inf_{s \ge 0} \ker(Z_1(G^n_M(r)) \to H_1(M)) = Z_1^{s r|\log r|}(G^n_M(r)).
\]
For fixed $\lam = nr^d$,
\Cref{isomorphism}
and \Cref{need at least rlogr}
give us $0<q \le Q< \infty$
such that as $r \to 0$,
$Y(r)$ concentrates on $[q,Q]$;
thus the family $\{Y(r)\}_{r\in (0,1]}$ is tight, and all subsequential limits are supported on a compact interval $[q,Q]$.
One may then ask: does $Y(r)$
converge in distribution as $r \to 0$?
If so, does it converge to a constant? If it did, this would give a ``critical value'' of $s$ for correctly detecting the homology of $H_1(M)$,
and it could be interesting to investigate how this value depends on $\lambda > \lambda_c$.
If $Y(r)$ instead converges to a nontrivial distribution, it would be very interesting
to understand its properties.

\bibliographystyle{plain}
\bibliography{refs}

\appendix
\section{A modified barycentric subdivision procedure} \label{sec:subdivision}

Here we prove the following fact, used in the proof of \Cref{induction}, by using a modified barycentric subdivision procedure.

\begin{prop}
    Suppose $\sigma \in C_2(M)$ with $\partial \sigma = z$
    and each $c \in z$ has $\mathrm{diam}(|c|) \le r$.
    Then there exists $\sigma' \in C_2(M)$
    such that $\partial \sigma' = z$
    and each $c' \in \sigma'$ has $\mathrm{diam}(|c'|) \le 2r$.
\end{prop}
\begin{proof}
We will define a family of modified barycentric subdivision maps $S^z_k:C_k(M) \to C_k(M)$ for $k=0,1,2$;
the basic idea is to perform usual barycentric subdivision
(see e.g. the proof of Proposition 2.21 in Hatcher \cite{Hatcher}), except on $z$, which we leave unchanged.
The key point will be that after each such subdivision,
the diameter of the 2-simplices not intersecting $z$ will decrease,
and for the $2$-simplices intersecting a $1$-simplex in $z$,
their distance to that $1$-simplex decreases. Therefore, 
as we perform successive subdivisions,
the maximum diameter of all the singular $2$-simplices in the subdivision approaches the maximum diameter of the $1$-simplices in $z$, which is by assumption at most $r$.

Before we define the subdivision on singular chains,
let us recall some notation regarding \emph{linear} chains;
this will help us describe how to subdivide the \emph{standard} simplex,
and then to subdivide a singular simplex $s:\Delta_k \to M$ we will simply pushforward our construction by the defining continuous map $s$.

Let $K$ be a convex subset of a linear space. We can denote a linear map $\mathrm{conv}(e_0,...,e_k)=:\Delta_k \to K$ by $[v_0,...,v_k]$,
where $v_i \in K$ is the image of $e_i$, and this uniquely determines the map. 
Let $LC_k(K)$ be the free $\mathcal{S}$-module generated by
\emph{linear} maps $\Delta_k \to K$.
Any $b \in K$ gives a \emph{cone operator} $c_b:LC_k(K) \to LC_{k+1}(K)$ by taking $c_b[v_0,...,v_k] = [b,v_0,...,v_k]$
and extending by linearity.

Now we define our subdivision maps.
$S^z_0:C_0(M) \to C_0(M)$ is the identity.
For a singular $1$-simplex $c \in C_1(M)$,
if $c \in z$, then take $S_1(c) = c$.
Otherwise, let $S^z_1(c) = c \circ [b,e_1] - c \circ [b,e_0]$ be
the usual barycentric subdivision map, where here
$b= \frac{1}{2}e_0 + \frac{1}{2} e_1$ is the barycenter
of the standard $1$-simplex $\Delta_1$.

Now we want to define $S^z_2 c$ for a singular $2$-simplex $s \in C_2(M)$.
Let $b(\Delta_2) = \frac{1}{3}e_0 +\frac{1}{3}e_1 +\frac{1}{3}e_2$
be the barycenter of the standard 2-simplex of $\Delta_2$,
and let $b(v_0,v_1) = \frac{1}{2}v_0+\frac{1}{2}v_1$
by the barycenter of a $1$-simplex $[v_0,v_1] \in LC_1(\Delta_2)$.
Consider $\Delta_2$ as a linear $2$-chain in $LC_2(\Delta_2)$ by identifying it with the identity map $[e_0,e_1,e_2]:\Delta_2 \to \Delta_2$. For each $[v_0,v_1] \in \partial \Delta_2$,
set $S'^z[v_0,v_1] = [v_0,v_1]$ if $s \circ [v_0,v_1] \in z$,
and otherwise let $S'^z[v_0,v_1] = [b(v_0,v_1),v_1] - [b(v_0,v_1),v_0]$ be the standard barycentric subdivision of $[v_0,v_1]$. Extending $S'^z$ by linearity,
and recalling the cone map $c_b$, we finally define
\[
    S^z_2 s = s \circ c_{b(\Delta_2)}(S'^z \partial \Delta_2)
\]

For our purposes we do not need to define $S^z_k$ for $k \ge 3$.
Just as with usual barycentric subdivision, one can check
that $S^z$ is a chain map, that is, $\partial S^z = S^z \partial$,
and so in particular 
\[
    \partial S^z \sigma = S^z \partial \sigma
    = S^z z = z
\]
and by induction, for any $n \ge 1$ we have
$\partial (S^z)^{\circ n} \sigma = z$.

Lastly, we see that if we perform this
procedure with $M=K$ a convex subset of Euclidean space
(the key example being $K=\Delta_2$),
and $\sigma \in LC_2(K)$,
we can keep track of
\begin{align*}
    D_{1,n} &:= \sup \{ \mathrm{diam}(|c|) : c \in (S^z_2)^{\circ n} \sigma, \forall c' \in z, c' \notin \partial c \} \\
    D_{2,n} &:=  \sup \{ d_H(|c|,|c'|) : c \in (S^z_2)^{\circ n} \sigma, c' \in \partial c, c' \in z \}
\end{align*}
where $d_H$ is Hausdorff distance as usually defined on subsets of Euclidean space. One can verify that for all $n \ge 1$ we have
\begin{align*}
    D_{1,n} &\le (2/3) D_{1,n-1}, \\
    D_{2,n} &\le (1/3) D_{2,n-1},
\end{align*}
so that as $n \to \infty$, $D_{1,n}, D_{2,n} \to 0$.
This means that the diameter of the simplices not intersecting $z$
is tending to $0$, and the Hausdorff distance between each simplex intersecting $z$ and the $1$-simplex it intersects at tends to $0$, meaning in particular that for large $n$, the diameter of such simplices
is at most a small error plus the maximum diameter of a $1$-simplex in $z$.

Having established this for the linear case, the desired conclusion in the singular case follows immediately from the fact that all singular simplices $s:\Delta_2 \to M$ are necessarily uniformly continuous maps.

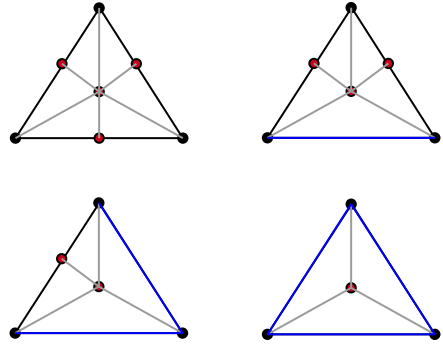
\begin{figure}
    \centering
\tikzset{every picture/.style={line width=0.75pt}} 

\begin{tikzpicture}[x=0.35pt,y=0.35pt,yscale=-1,xscale=1]

\draw   (120.5,80) -- (210,220) -- (30,220) -- cycle ;
\draw  [color={rgb, 255:red, 0; green, 0; blue, 0 }  ,draw opacity=1 ][fill={rgb, 255:red, 0; green, 0; blue, 0 }  ,fill opacity=1 ] (25,220) .. controls (25,217.24) and (27.24,215) .. (30,215) .. controls (32.76,215) and (35,217.24) .. (35,220) .. controls (35,222.76) and (32.76,225) .. (30,225) .. controls (27.24,225) and (25,222.76) .. (25,220) -- cycle ;
\draw  [color={rgb, 255:red, 0; green, 0; blue, 0 }  ,draw opacity=1 ][fill={rgb, 255:red, 0; green, 0; blue, 0 }  ,fill opacity=1 ] (205,220) .. controls (205,217.24) and (207.24,215) .. (210,215) .. controls (212.76,215) and (215,217.24) .. (215,220) .. controls (215,222.76) and (212.76,225) .. (210,225) .. controls (207.24,225) and (205,222.76) .. (205,220) -- cycle ;
\draw  [color={rgb, 255:red, 0; green, 0; blue, 0 }  ,draw opacity=1 ][fill={rgb, 255:red, 0; green, 0; blue, 0 }  ,fill opacity=1 ] (115,80) .. controls (115,77.24) and (117.24,75) .. (120,75) .. controls (122.76,75) and (125,77.24) .. (125,80) .. controls (125,82.76) and (122.76,85) .. (120,85) .. controls (117.24,85) and (115,82.76) .. (115,80) -- cycle ;
\draw  [color={rgb, 255:red, 0; green, 0; blue, 0 }  ,draw opacity=1 ][fill={rgb, 255:red, 208; green, 2; blue, 27 }  ,fill opacity=1 ] (75,140) .. controls (75,137.24) and (77.24,135) .. (80,135) .. controls (82.76,135) and (85,137.24) .. (85,140) .. controls (85,142.76) and (82.76,145) .. (80,145) .. controls (77.24,145) and (75,142.76) .. (75,140) -- cycle ;
\draw  [color={rgb, 255:red, 0; green, 0; blue, 0 }  ,draw opacity=1 ][fill={rgb, 255:red, 208; green, 2; blue, 27 }  ,fill opacity=1 ] (155,140) .. controls (155,137.24) and (157.24,135) .. (160,135) .. controls (162.76,135) and (165,137.24) .. (165,140) .. controls (165,142.76) and (162.76,145) .. (160,145) .. controls (157.24,145) and (155,142.76) .. (155,140) -- cycle ;
\draw  [color={rgb, 255:red, 0; green, 0; blue, 0 }  ,draw opacity=1 ][fill={rgb, 255:red, 208; green, 2; blue, 27 }  ,fill opacity=1 ] (115,220) .. controls (115,217.24) and (117.24,215) .. (120,215) .. controls (122.76,215) and (125,217.24) .. (125,220) .. controls (125,222.76) and (122.76,225) .. (120,225) .. controls (117.24,225) and (115,222.76) .. (115,220) -- cycle ;
\draw  [color={rgb, 255:red, 0; green, 0; blue, 0 }  ,draw opacity=1 ][fill={rgb, 255:red, 208; green, 2; blue, 27 }  ,fill opacity=1 ] (115,170) .. controls (115,167.24) and (117.24,165) .. (120,165) .. controls (122.76,165) and (125,167.24) .. (125,170) .. controls (125,172.76) and (122.76,175) .. (120,175) .. controls (117.24,175) and (115,172.76) .. (115,170) -- cycle ;
\draw [color={rgb, 255:red, 155; green, 155; blue, 155 }  ,draw opacity=1 ]   (30,220) -- (120,170) ;
\draw [color={rgb, 255:red, 155; green, 155; blue, 155 }  ,draw opacity=1 ]   (210,220) -- (120,170) ;
\draw [color={rgb, 255:red, 155; green, 155; blue, 155 }  ,draw opacity=1 ]   (120,170) -- (120,80) ;
\draw [color={rgb, 255:red, 155; green, 155; blue, 155 }  ,draw opacity=1 ]   (120,170) -- (160,140) ;
\draw [color={rgb, 255:red, 155; green, 155; blue, 155 }  ,draw opacity=1 ]   (120,170) -- (80,140) ;
\draw [color={rgb, 255:red, 155; green, 155; blue, 155 }  ,draw opacity=1 ]   (120,220) -- (120,170) ;
\draw   (391.5,79.67) -- (481,219.67) -- (301,219.67) -- cycle ;
\draw  [color={rgb, 255:red, 0; green, 0; blue, 0 }  ,draw opacity=1 ][fill={rgb, 255:red, 0; green, 0; blue, 0 }  ,fill opacity=1 ] (296,219.67) .. controls (296,216.91) and (298.24,214.67) .. (301,214.67) .. controls (303.76,214.67) and (306,216.91) .. (306,219.67) .. controls (306,222.43) and (303.76,224.67) .. (301,224.67) .. controls (298.24,224.67) and (296,222.43) .. (296,219.67) -- cycle ;
\draw  [color={rgb, 255:red, 0; green, 0; blue, 0 }  ,draw opacity=1 ][fill={rgb, 255:red, 0; green, 0; blue, 0 }  ,fill opacity=1 ] (476,219.67) .. controls (476,216.91) and (478.24,214.67) .. (481,214.67) .. controls (483.76,214.67) and (486,216.91) .. (486,219.67) .. controls (486,222.43) and (483.76,224.67) .. (481,224.67) .. controls (478.24,224.67) and (476,222.43) .. (476,219.67) -- cycle ;
\draw  [color={rgb, 255:red, 0; green, 0; blue, 0 }  ,draw opacity=1 ][fill={rgb, 255:red, 0; green, 0; blue, 0 }  ,fill opacity=1 ] (386,79.67) .. controls (386,76.91) and (388.24,74.67) .. (391,74.67) .. controls (393.76,74.67) and (396,76.91) .. (396,79.67) .. controls (396,82.43) and (393.76,84.67) .. (391,84.67) .. controls (388.24,84.67) and (386,82.43) .. (386,79.67) -- cycle ;
\draw  [color={rgb, 255:red, 0; green, 0; blue, 0 }  ,draw opacity=1 ][fill={rgb, 255:red, 208; green, 2; blue, 27 }  ,fill opacity=1 ] (346,139.67) .. controls (346,136.91) and (348.24,134.67) .. (351,134.67) .. controls (353.76,134.67) and (356,136.91) .. (356,139.67) .. controls (356,142.43) and (353.76,144.67) .. (351,144.67) .. controls (348.24,144.67) and (346,142.43) .. (346,139.67) -- cycle ;
\draw  [color={rgb, 255:red, 0; green, 0; blue, 0 }  ,draw opacity=1 ][fill={rgb, 255:red, 208; green, 2; blue, 27 }  ,fill opacity=1 ] (426,139.67) .. controls (426,136.91) and (428.24,134.67) .. (431,134.67) .. controls (433.76,134.67) and (436,136.91) .. (436,139.67) .. controls (436,142.43) and (433.76,144.67) .. (431,144.67) .. controls (428.24,144.67) and (426,142.43) .. (426,139.67) -- cycle ;
\draw  [color={rgb, 255:red, 0; green, 0; blue, 0 }  ,draw opacity=1 ][fill={rgb, 255:red, 208; green, 2; blue, 27 }  ,fill opacity=1 ] (386,169.67) .. controls (386,166.91) and (388.24,164.67) .. (391,164.67) .. controls (393.76,164.67) and (396,166.91) .. (396,169.67) .. controls (396,172.43) and (393.76,174.67) .. (391,174.67) .. controls (388.24,174.67) and (386,172.43) .. (386,169.67) -- cycle ;
\draw [color={rgb, 255:red, 155; green, 155; blue, 155 }  ,draw opacity=1 ]   (301,219.67) -- (391,169.67) ;
\draw [color={rgb, 255:red, 155; green, 155; blue, 155 }  ,draw opacity=1 ]   (481,219.67) -- (391,169.67) ;
\draw [color={rgb, 255:red, 155; green, 155; blue, 155 }  ,draw opacity=1 ]   (391,169.67) -- (391,79.67) ;
\draw [color={rgb, 255:red, 155; green, 155; blue, 155 }  ,draw opacity=1 ]   (391,169.67) -- (431,139.67) ;
\draw [color={rgb, 255:red, 155; green, 155; blue, 155 }  ,draw opacity=1 ]   (391,169.67) -- (351,139.67) ;
\draw [color={rgb, 255:red, 0; green, 0; blue, 255 }  ,draw opacity=1 ]   (301,219.67) -- (481,219.67) ;
\draw   (120.17,289.67) -- (209.67,429.67) -- (29.67,429.67) -- cycle ;
\draw  [color={rgb, 255:red, 0; green, 0; blue, 0 }  ,draw opacity=1 ][fill={rgb, 255:red, 0; green, 0; blue, 0 }  ,fill opacity=1 ] (24.67,429.67) .. controls (24.67,426.91) and (26.91,424.67) .. (29.67,424.67) .. controls (32.43,424.67) and (34.67,426.91) .. (34.67,429.67) .. controls (34.67,432.43) and (32.43,434.67) .. (29.67,434.67) .. controls (26.91,434.67) and (24.67,432.43) .. (24.67,429.67) -- cycle ;
\draw  [color={rgb, 255:red, 0; green, 0; blue, 0 }  ,draw opacity=1 ][fill={rgb, 255:red, 0; green, 0; blue, 0 }  ,fill opacity=1 ] (204.67,429.67) .. controls (204.67,426.91) and (206.91,424.67) .. (209.67,424.67) .. controls (212.43,424.67) and (214.67,426.91) .. (214.67,429.67) .. controls (214.67,432.43) and (212.43,434.67) .. (209.67,434.67) .. controls (206.91,434.67) and (204.67,432.43) .. (204.67,429.67) -- cycle ;
\draw  [color={rgb, 255:red, 0; green, 0; blue, 0 }  ,draw opacity=1 ][fill={rgb, 255:red, 0; green, 0; blue, 0 }  ,fill opacity=1 ] (114.67,289.67) .. controls (114.67,286.91) and (116.91,284.67) .. (119.67,284.67) .. controls (122.43,284.67) and (124.67,286.91) .. (124.67,289.67) .. controls (124.67,292.43) and (122.43,294.67) .. (119.67,294.67) .. controls (116.91,294.67) and (114.67,292.43) .. (114.67,289.67) -- cycle ;
\draw  [color={rgb, 255:red, 0; green, 0; blue, 0 }  ,draw opacity=1 ][fill={rgb, 255:red, 208; green, 2; blue, 27 }  ,fill opacity=1 ] (74.67,349.67) .. controls (74.67,346.91) and (76.91,344.67) .. (79.67,344.67) .. controls (82.43,344.67) and (84.67,346.91) .. (84.67,349.67) .. controls (84.67,352.43) and (82.43,354.67) .. (79.67,354.67) .. controls (76.91,354.67) and (74.67,352.43) .. (74.67,349.67) -- cycle ;
\draw  [color={rgb, 255:red, 0; green, 0; blue, 0 }  ,draw opacity=1 ][fill={rgb, 255:red, 208; green, 2; blue, 27 }  ,fill opacity=1 ] (114.67,379.67) .. controls (114.67,376.91) and (116.91,374.67) .. (119.67,374.67) .. controls (122.43,374.67) and (124.67,376.91) .. (124.67,379.67) .. controls (124.67,382.43) and (122.43,384.67) .. (119.67,384.67) .. controls (116.91,384.67) and (114.67,382.43) .. (114.67,379.67) -- cycle ;
\draw [color={rgb, 255:red, 155; green, 155; blue, 155 }  ,draw opacity=1 ]   (29.67,429.67) -- (119.67,379.67) ;
\draw [color={rgb, 255:red, 155; green, 155; blue, 155 }  ,draw opacity=1 ]   (209.67,429.67) -- (119.67,379.67) ;
\draw [color={rgb, 255:red, 155; green, 155; blue, 155 }  ,draw opacity=1 ]   (119.67,379.67) -- (119.67,289.67) ;
\draw [color={rgb, 255:red, 155; green, 155; blue, 155 }  ,draw opacity=1 ]   (119.67,379.67) -- (79.67,349.67) ;
\draw [color={rgb, 255:red, 0; green, 0; blue, 255 }  ,draw opacity=1 ]   (29.67,429.67) -- (209.67,429.67) ;
\draw [color={rgb, 255:red, 0; green, 0; blue, 255 }  ,draw opacity=1 ]   (119.67,289.67) -- (209.67,429.67) ;
\draw   (391.5,291) -- (481,431) -- (301,431) -- cycle ;
\draw  [color={rgb, 255:red, 0; green, 0; blue, 0 }  ,draw opacity=1 ][fill={rgb, 255:red, 0; green, 0; blue, 0 }  ,fill opacity=1 ] (296,431) .. controls (296,428.24) and (298.24,426) .. (301,426) .. controls (303.76,426) and (306,428.24) .. (306,431) .. controls (306,433.76) and (303.76,436) .. (301,436) .. controls (298.24,436) and (296,433.76) .. (296,431) -- cycle ;
\draw  [color={rgb, 255:red, 0; green, 0; blue, 0 }  ,draw opacity=1 ][fill={rgb, 255:red, 0; green, 0; blue, 0 }  ,fill opacity=1 ] (476,431) .. controls (476,428.24) and (478.24,426) .. (481,426) .. controls (483.76,426) and (486,428.24) .. (486,431) .. controls (486,433.76) and (483.76,436) .. (481,436) .. controls (478.24,436) and (476,433.76) .. (476,431) -- cycle ;
\draw  [color={rgb, 255:red, 0; green, 0; blue, 0 }  ,draw opacity=1 ][fill={rgb, 255:red, 0; green, 0; blue, 0 }  ,fill opacity=1 ] (386,291) .. controls (386,288.24) and (388.24,286) .. (391,286) .. controls (393.76,286) and (396,288.24) .. (396,291) .. controls (396,293.76) and (393.76,296) .. (391,296) .. controls (388.24,296) and (386,293.76) .. (386,291) -- cycle ;
\draw  [color={rgb, 255:red, 0; green, 0; blue, 0 }  ,draw opacity=1 ][fill={rgb, 255:red, 208; green, 2; blue, 27 }  ,fill opacity=1 ] (386,381) .. controls (386,378.24) and (388.24,376) .. (391,376) .. controls (393.76,376) and (396,378.24) .. (396,381) .. controls (396,383.76) and (393.76,386) .. (391,386) .. controls (388.24,386) and (386,383.76) .. (386,381) -- cycle ;
\draw [color={rgb, 255:red, 155; green, 155; blue, 155 }  ,draw opacity=1 ]   (301,431) -- (391,381) ;
\draw [color={rgb, 255:red, 155; green, 155; blue, 155 }  ,draw opacity=1 ]   (481,431) -- (391,381) ;
\draw [color={rgb, 255:red, 155; green, 155; blue, 155 }  ,draw opacity=1 ]   (391,381) -- (391,291) ;
\draw [color={rgb, 255:red, 0; green, 0; blue, 255 }  ,draw opacity=1 ]   (301,431) -- (481,431) ;
\draw [color={rgb, 255:red, 0; green, 0; blue, 255 }  ,draw opacity=1 ]   (391,291) -- (481,431) ;
\draw [color={rgb, 255:red, 0; green, 0; blue, 255 }  ,draw opacity=1 ]   (391,291) -- (301,431) ;

\end{tikzpicture}
    \caption{Subdivision of the 
    standard $2$-simplex $\Delta_2$.
    Here the blue lines represent
    $1$-simplices which we do not want to subdivide because they map to $1$-simplices in $z$.}
    \label{fig:subdivision}
\end{figure}

\end{proof}

\end{document}